\documentclass[10pt,a4paper,reqno]{amsart}
\usepackage[T1]{fontenc}
\usepackage[utf8]{inputenc}
\usepackage[mathscr]{euscript}
\usepackage[english]{babel}
\usepackage{amsfonts,amsmath,amsthm,amssymb,latexsym}
\usepackage{mathtools}
\usepackage{bm}
\usepackage{bbm}
\usepackage{graphicx}
\usepackage{esint}
\usepackage{enumitem}
\usepackage{caption}
\usepackage{xcolor}
\usepackage{multicol}
\usepackage{float} 
\usepackage{epsfig}
\usepackage{epstopdf}
\usepackage{cases}
\usepackage{geometry}
\DeclarePairedDelimiter{\abs}{\lvert}{\rvert}
\DeclarePairedDelimiter{\norma}{\lVert}{\rVert}
\DeclarePairedDelimiter{\abss}{\bigg\lvert}{\bigg\rvert}
\theoremstyle{definition}
\newtheorem{rmk}{Remark} 
\newtheorem{definition}{Definition}
\theoremstyle{plain}
\newtheorem{teorema}{Theorem} 

\newtheorem{lemma}{Lemma}
\newtheorem{proposizione}{Proposition}

\newcommand{\into}{\int_{\Omega}}
\newcommand{\Ldod}{L^2(\Omega)^2}

\title[Second order two-species systems with nonlocal interactions]{Second order two-species systems with nonlocal interactions: existence and large damping limits}
\author{Marco Di Francesco, Simone Fagioli, Valeria Iorio}

\address[Marco Di Francesco]{\newline Dipartimento di Ingegneria e Scienze dell’Informazione e Matematica\newline Universit\`a degli Studi dell’Aquila, Via Vetoio 1, 67100 Coppito, L’Aquila, It.}
\email{marco.difrancesco@univaq.it}
\address[Simone Fagioli]{\newline Dipartimento di Ingegneria e Scienze dell’Informazione e Matematica\newline Universit\`a degli Studi dell’Aquila, Via Vetoio 1, 67100 Coppito, L’Aquila, It.}
\email{simone.fagioli@univaq.it}
\address[Valeria Iorio]{\newline Dipartimento di Ingegneria e Scienze dell’Informazione e Matematica\newline Universit\`a degli Studi dell’Aquila, Via Vetoio 1, 67100 Coppito, L’Aquila, It.}
\email{valeria.iorio1@graduate.univaq.it}
\date{\today}

\begin{document}
\maketitle
\begin{abstract}
We study the mathematical theory of second order systems with two species, arising in the dynamics of interacting particles subject to linear damping, to nonlocal forces and to external ones, and resulting into a nonlocal version of the compressible Euler system with linear damping. Our results are limited to the $1$ space dimensional case but allow for initial data taken in a Wasserstein space of probability measures. We first consider the case of smooth nonlocal interaction potentials, not subject to any symmetry condition, and prove existence and uniqueness. The concept of solutions relies on a \emph{stickiness condition} in case of collisions, in the spirit of previous works in the literature. The result uses concepts from classical Hilbert space theory of gradient flows (cf. Brezis \cite{brezis}) and a trick used in \cite{gangbo}. We then consider a large-time and large-damping scaled version of our system and prove convergence to solutions to the corresponding first order system. Finally, we consider the case of Newtonian potentials - subject to symmetry of the cross-interaction potentials - and external convex potentials. After showing existence in the sticky particles framework in the spirit of \cite{gangbo}, we prove convergence for large times towards Dirac delta solutions for the two densities. All the results share a common technical framework in that solutions are considered in a \emph{Lagrangian} framework, which allows to estimate the behavior of solutions via $L^2$ estimates of the pseudo-inverse variables corresponding to the two densities. In particular, due to this technique, the large-damping result holds under a rather weak condition on the initial data, which does not require \emph{well-prepared}  initial velocities. We complement the results with numerical simulations.
\end{abstract}

\section{Introduction}

Nonlocal aggregation models touch various domains of science and technology such as astro-particle physics, microbiology, population biology, social sciences, artificial intelligence and machine learning. The use of integro-partial differential equations in this context, describing the evolution of a density of individuals $\rho(x,t)$ subject to nonlocal interaction forces, such as
\begin{equation}\label{eq:intro1}
  \frac{\partial \rho}{\partial t}=\mathrm{div}(\rho \nabla W\ast \rho)\,, 
\end{equation}
has become very popular in the literature. In \eqref{eq:intro1}, $x$ is a spatial variable typically ranging in $\mathbb{R}^d$, $t\geq 0$ is time, $W=W(x)$ is a given interaction potential accounting for attractive or repulsive drift among individuals. The modelling approach in \eqref{eq:intro1}
allows to formulate concepts of solutions in a ``probability measure landscape'', which includes the motion of ``pointy particles'' 
\begin{equation}\label{eq:intro1bis}
    \dot{x}_i= -\frac{1}{N}\sum_{j=1}^N \nabla W (x_i-x_j)
\end{equation}
as a special case of the ``continuum theory'', see for example the results in \cite{AGS, cdefs} framed within the Wasserstein gradient flow theory. 

At least on a formal level, a similar dichotomy exists in the case of a \emph{second-order} approach that takes into account inertial effects, namely
\begin{align}
 & \frac{\partial \rho}{\partial t}+\mathrm{div}(\rho v) = 0 \nonumber \\
   & \frac{\partial (\rho v)}{\partial t} + \mathrm{div}(\rho v\otimes v) = -\rho \nabla W\ast \rho\,,\label{eq:intro2}
\end{align}
the ``particle-counterpart'' of which is given by the second-order particle system
\begin{align}
    & \dot{x}_i= v_i \nonumber\\
    & \dot{v}_i= -\frac{1}{N}\sum_{j=1}^N \nabla W (x_i-x_j)\,.\label{eq:intro2bis}
\end{align}
In \eqref{eq:intro2}, $v=v(x,t)$ is the Eulerian velocity of the fluid-like ensemble of individuals. 

System \eqref{eq:intro2} can be considered as a nonlocal version of Euler system for gas-dynamics, in which the classical ``pressure term'' $-\nabla p(\rho)$ is replaced by a nonlocal interaction force $-\rho \nabla W\ast \rho$. A variant of \eqref{eq:intro2} includes a linear ``friction'' - or ``damping'' - term (with $\sigma>0$ a damping parameter) and an external force $-\nabla V$, the full model including nonlinear pressure thus looking like
\begin{align}
 & \frac{\partial \rho}{\partial t}+\mathrm{div}(\rho v) = 0 \nonumber \\
   & \frac{\partial (\rho v)}{\partial t} + \mathrm{div}(\rho v\otimes v) +\nabla p(\rho) = -\rho \nabla W\ast \rho-\rho V -\sigma \rho \nabla v\,.\label{eq:intro3}
\end{align}

The existence theory for systems of the form \eqref{eq:intro3} is a classical challenge of the analysis of PDEs, with strong links with the mathematical theory of systems of conservation laws. Since we are not concerned with pressure terms in this paper, we briefly list a few references on this matter such as \cite{WangYang,bertozzi,FangXu} for the multidimensional case and \cite{DafermosPan,DingChenLuo,HsiaoLuoYang,HuangPan} for the one-dimensional case. We refer to \cite{chen} for a survey on Euler equations.

In the pressure-less case $p=0$ in \eqref{eq:intro3}, the density $\rho$ is not forced to be absolutely continue with respect to Lebesgue measure. Hence, ``particle'' solutions in the spirit of \eqref{eq:intro2bis} are allowed. When two particles collide, a standard way to continue the solution after collision is the so-called \emph{sticky particle condition}, which forces particles to stay attached to each other after collision, with a post-collisional velocity that is uniquely determined by the conservation of momentum. The existence and uniqueness of such ``sticky particle'' solutions has attracted the attention of many researchers for decades. We refer to the recent \cite{bianchini} for a through explanation of the issues related with existence and uniqueness in the multi-dimensional case. A case with $W\neq 0$ that is particularly interesting in the applications it the Euler-Posson model, in which $W$ is the solution operator to Poisson equation, see \cite{Nguyen}. In one space dimension the theory is quite rich of results in the literature, we mention here \cite{BreGre,ns,gangbo}. In particular, \cite{ns,gangbo} first addressed the coupling with general nonlocal forces $W\neq 0$ in the context of Wasserstein gradient flows. 

At least on a formal level, models of the form
\begin{equation}\label{eq:intro4}
  \frac{\partial \rho}{\partial t}=\mathrm{div}(\nabla p(\rho) +\rho \nabla (V + W\ast \rho))
\end{equation}
can be obtained by rescaling time in \eqref{eq:intro3} as $t=\sigma \tau$ and letting $\sigma\rightarrow +\infty$. Such a singular limit regime is called ``overdamped limit'', or ``large friction limit'', and, in case $p\neq 0$, it is well-known in the literature of singular limits for systems of conservation laws as a ``diffusive relaxation limit''. Relevant examples arise in the case of porous-medium like pressures \cite{marcati} $p(\rho)=\rho^\gamma$, $\gamma>1$, see also \cite{lattanzio} for more general models, and the more recent result in \cite{choi}.

In recent years the attention of many researchers in this field turned to systems with \emph{many species}, motivated for example by opinion formation models \cite{during}, chemotaxis models with many species of cells and other aggregation phenomena in biology \cite{espejo,bruna_chapman}, pedestrian movements \cite{degond}. In many of those phenomena, the second order modelling approach via \eqref{eq:intro2} or similar seems more appropriate in that inertial effects, sometimes referred in these contexts as ``persistence'' effects, do play a role in the model's dynamics. However, while the mathematical theory of many species systems in the first-order modelling approach has been considered in many papers \cite{}, very little attention has been devoted to second order models with many species. 
In this paper we wish to provide a first contribution in that direction, by restricting for the moment to the one-dimensional case. 

More precisely, we will first of all tackle the existence theory of the system
\begin{equation} 
\label{sistema1}
\begin{dcases} 
\frac{\partial \rho}{\partial t} +\frac{\partial}{\partial x}(\rho v)=0, \\ \frac{\partial \eta}{\partial t}+\frac{\partial}{\partial x} (\eta w)=0, \\ \frac{\partial}{\partial t} (\rho v)+\frac{\partial}{\partial x}(\rho v^2)=-\sigma\rho v -\rho [K_\rho'\ast \rho +H_\rho'\ast \eta], \\ \frac{\partial}{\partial t} (\eta w)+\frac{\partial }{\partial x} (\eta w^2)=-\sigma\eta w -\eta [K_\eta'\ast\eta +H_\eta'\ast\rho ],
\end{dcases}
\end{equation}
equipped with initial data
\begin{equation}
\begin{dcases}
(\rho, v)(t=0)=(\overline{\rho},\overline{v}), \\ (\eta,w)(t=0)=(\overline{\eta},\overline{w}).
\end{dcases}
\end{equation}
In  system \eqref{sistema1}, $\rho$ and $\eta$ are probability measures modelling two species of agents, or individuals, $v$ and $w$ are the corresponding Eulerian velocities of the two species, $\sigma >0$ is the damping parameter, $H_\rho, H_\eta, K_\rho$, $K_\eta$ are smooth (to an extent to be specified later) given space-depending potentials. $K_\rho$ and $K_\eta$ are called \textit{self-interaction} potentials as they describe the interaction between the agents of same species, $H_\rho$ and $H_\eta$ are called \textit{cross-interaction} potentials and model the interaction between the agents of opposing species. The convolutions in \eqref{sistema1} are meant with respect to the space variable. All potentials appears in the system with their first derivative. This choice of ours is merely motivated by the fact that all those terms should be considered as gradients of potential energies.

System \eqref{sistema1} has a natural discrete \emph{particle} counterpart. Let us consider $x_1,	\ldots,x_N$ as $N$ particles of the first species with masses $m_1, \ldots, m_N,$ and $y_1, \ldots, y_M$ as $M$ particles of the second species with masses $n_1, \ldots, n_M.$ The dynamics of $x_i$ and $y_j$ is determined by the following equations
\begin{equation}
\label{neweq}
\begin{cases}
\ddot{x}_i(t)=-\sigma \dot{x}_i(t)-\sum_{k\neq i}m_k K'_\rho \big(x_i(t)-x_k(t)\big)-\sum_kn_k  H'_\rho \big(x_i(t)-y_k(t)\big), \\  \ddot{y}_j(t)=-\sigma \dot{y}_j(t)-\sum_{k\neq j}n_k K'_\eta \big(y_j(t)-y_k(t)\big)-\sum_km_k H'_\eta \big(y_j(t)-x_k(t)\big),
\end{cases}
\end{equation}
 with $i=1,\ldots ,N$ and $j=1,\ldots ,M$ and the following initial data
\[
\begin{dcases} x_i(0)=\overline{x}_i, \\ \dot{x}_i(0)=\overline{v}_i,
\end{dcases} \qquad \begin{dcases} y_j(0)=\overline{y}_j, \\ \dot{y}_j(0)=\overline{w}_j. \end{dcases}
\]

The candidate large-friction of \eqref{sistema1} is the first order system
\begin{equation}
\label{sistema2}
\begin{dcases} \frac{\partial \rho}{\partial t}=\frac{\partial}{\partial x}\big[ \rho K_\rho '\ast\rho +\rho H_\rho '\ast\eta \big], \\ \frac{\partial\eta}{\partial t} =\frac{\partial}{\partial x}\big[ \eta K_\eta '\ast\eta +\eta H_\eta '\ast\rho \big]\,,
\end{dcases}
\end{equation}
which was extensively studied in \cite{difrafag}, see also \cite{DFEF} for the case with cross-diffusion terms.

Our work contributes to the above line of research in what follows. We stress that our results only deal with the one space dimensional case.
\begin{itemize}
    \item [(i)] In the case of smooth interaction potentials, we provide a well-posedness result in the $2$-Wasserstein space of probability measures. Our main result is contained in Theorem \ref{main_th}.
    \item [(ii)] We then investigate the large damping limit and prove that, under a suitable rescaling, our system converges towards the corresponding first order model. Our result in this framework is stated in Theorem \ref{th:large_damping}. We observe that, at least to our knowledge, the technique used in Theorem \ref{th:large_damping} was never used in the one-species case.
    \item [(iii)] We then consider the case of Newtonian potentials for the self-interaction part and suitably coercive external potentials and prove a large-time collapse result in Theorem \ref{th:newtonian}.
\end{itemize}

The paper is structured as follows. In section \ref{sec:prelim} we introduce the main concepts of gradient flows in Wasserstein spaces that we need in our paper, we introduce the large-damping scaling limit, we state our model in a suitable Lagrangian framework, and we state our main results (see subsection \ref{subsec:mainres}). In section \ref{sec:existence} we prove Theorem \ref{main_th}, essentially following the classical strategy of \cite{brezis}, which also used in \cite{gangbo} for the one-species case. In section \ref{sec:large_damping} we prove Theorem  \ref{th:large_damping}. In Section \ref{sec:newtonian} we consider the case of Newtonian potentials and prove existence of sticky solutions and the large time collapse to Dirac deltas stated in Theorem \ref{th:newtonian}. Finally, in Section \ref{sec:numerics} we provide some numerical simulations.

\section{Preliminaries and main results}\label{sec:prelim}
In what follows we will set the notations, the assumptions, and introduce definitions that will be used throughout the paper, see Subsections \ref{subsec:wass} and \ref{subsec:assumptions}. Subsections \ref{subsec:particlessystem} and \ref{subsec:lagrangiandescription} are devoted to the precise description of system \eqref{sistema1} in terms of particles and Lagrangian coordinates respectively. In Subsection \ref{sec:large_dam} we provide a formal argument for the large damping limit of system \eqref{sistema1} towards \eqref{sistema2}. Finally Subsection \ref{subsec:mainres} collects the main results of the paper.

\subsection{One dimensional Wasserstein distance}\label{subsec:wass}
We start introducing some preliminaries and  definitions on the metric structure; the reader can refer to the classical references \cite{AGS,villani} for further details.
Let $\mathcal{P}(\mathbb{R}^n)$ be the set of Borel probability measures on $\mathbb{R}^n.$ Given $\mu\in\mathcal{P}(\mathbb{R}^n)$ and a Borel map $T:\mathbb{R}^n\rightarrow\mathbb{R}^m$, we denote by $\nu\coloneqq T\#\mu\in\mathcal{P}(\mathbb{R}^m)$ the \textit{push-forward of $\mu$ through $T$} defined by
\[ \nu (A)=\mu\big(T^{-1}(A)\big) \qquad \text{for all Borel sets $A\in\mathbb{R}^m.$}\]

We denote by $\mathcal{P}_2(\mathbb{R}^d)$ the set of probability measures on $\mathbb{R}^d$ with finite second moment, i.e., $\int_{\mathbb{R}^d}\abs{x}^2\,d\mu(x)<\infty$ for all $\mu\in\mathcal{P}_2(\mathbb{R}^d).$ The \textit{$2$-Wasserstein distance} $W_2(\mu,\nu)$ between two measures $\mu,\nu\in\mathcal{P}_2(\mathbb{R}^d)$ is defined by
\begin{equation}
\label{def:wassdist}
W_2^2(\mu,\nu)=\inf_{\bm\gamma\in\Pi(\mu,\nu)}\bigg\{ \iint_{\mathbb{R}^d\times\mathbb{R}^d} \abs{x-y}^2\,d\bm{\gamma}(x,y) \bigg\}, 
\end{equation}
where $\Pi(\mu,\nu)$ denotes the class of transport plans between $\mu$ and $\nu$, i.e., the probability measures $\bm{\gamma}$ on $\mathbb{R}^d\times\mathbb{R}^d$ that satisfy the conditions
\[\pi^1\#\bm{\gamma}=\mu, \qquad\pi^2\#\bm{\gamma}=\nu, \]
where $\pi^i$ is the projection operator on the $i$-th component of the product space. By introducing the class of optimal plans between $\mu$ and $\nu$, i.e., minimizers of \eqref{def:wassdist}, denoted by $\Pi_o(\mu,\nu)$, the Wasserstein distance can be rewritten as
\begin{equation}
\label{wass}
W_2^2(\mu,\nu)= \iint_{\mathbb{R}^d\times\mathbb{R}^d} \abs{x-y}^2\,d\bm{\gamma}(x,y) \qquad \bm{\gamma}\in\Pi_o(\mu,\nu). 
\end{equation}

In the one-dimensional case, there exists a unique optimal plan $\bm{\gamma}\in\Pi_o(\mu,\nu)$ for which the infimum in \eqref{def:wassdist} is attained, and it can be characterised by the monotone rearrangements of $\mu$ and $\nu$: given $\mu\in\mathcal{P}(\mathbb{R})$, its monotone rearrangement is
\[ X_\mu (m)\coloneqq \inf \{ x\, :\, M_\mu(x)>m\} \qquad \text{for all $m\in\Omega$}, \]
where $\Omega\coloneqq (0,1)$ and $M_\mu$ is the cumulative distribution of the measure $\mu$, i.e.,
\[ M_\mu(x)\coloneqq \mu\big( (-\infty,x] \big) \qquad \text{for all $x\in\mathbb{R}.$}\]
The map $X_\mu$ is right-continuous and nondecreasing and satisfies, by denoting the one-dimensional Lebesgue measure on $\Omega$ by $\mathfrak{m}$,
\[ \big( X_\mu \big)\#\mathfrak{m}=\mu, \qquad \int_\mathbb{R}\zeta(x)\,\mu (dx)=\int_\Omega \zeta\big( X_\mu(m) \big)\,dm \]
for all Borel maps $\zeta:\mathbb{R}\rightarrow\mathbb{R}.$ In particular, $\mu\in\mathcal{P}_2(\mathbb{R})$ if and only if $X_\mu\in L^2(\Omega)$. Moreover, the joint map $X_{\mu,\nu}:\Omega\rightarrow \mathbb{R}\times\mathbb{R}$ defined by $X_{\mu,\nu}(m)\coloneqq \big( X_\mu (m),X_\nu (m) \big)$ characterises the optimal transport plan $\bm\gamma\in\Pi_o(\mu,\nu)$ by the formula
\[ \bm{\gamma}=\big( X_{\mu,\nu} \big) \#\mathfrak{m},\]
according to which
\[ W_2^2(\mu,\nu)=\int_\Omega \abs{X_\mu (m)-X_\nu (m)}^2\,dm. \]
We further recall that, introducing the closed convex set of right-continuous non-decreasing functions in the Hilbert space $L^2(\Omega ),$ i.e.,
\begin{equation}\label{cone}
\mathcal{K}\coloneqq \{ X\in L^2(\Omega )\; :\; X \, \text{is non-decreasing} \},
\end{equation} 
the map 
\begin{equation}\label{isometry}
    \Psi :\mathcal{P}_2(\mathbb{R}) \ni\mu\mapsto X_\mu\in\mathcal{K}
\end{equation} 
is a distance-preserving bijection between the space of probability measures with finite second moments $\mathcal{P}_2(\mathbb{R})$ and the convex cone $\mathcal{K}$ of non-decreasing $L^2$-functions.


Since we are dealing with a two-species system, we will work on the product space $\mathcal{P}_2(\mathbb{R})\times\mathcal{P}_2(\mathbb{R})$.  For all $\bm{\mu}=(\mu _1,\mu _2),$ $\bm{\nu}=(\nu _1,\nu _2)\in\mathcal{P}_2(\mathbb{R})\times\mathcal{P}_2(\mathbb{R}),$ we define the \textit{product Wasserstein distance} as follows
\[ \mathcal{W}_2^2(\bm{\mu},\bm{\nu})=W_2^2(\mu_1,\nu_1)+W_2^2(\mu_2,\nu_2). \]



\subsection{Main assumptions}\label{subsec:assumptions}

In this subsection we collect the main assumptions we will need in the rest of the paper. We start by specifying the class of interaction potentials we are going to use.
\begin{definition}
A function $K:\mathbb{R}\to\mathbb{R}$ is called an \textit{admissible potential} if 
\begin{equation}\label{A}
K\in C^1(\mathbb{R}),\, K(0)=0 \mbox{ and }K(-x)=K(x).\tag{\textbf{A}} 
\end{equation}
An admissible potential $K$ is said to be \textit{sub-quadratic at infinity} if there exists a constant $C>0$ such that \begin{equation}\label{SQ}
K(x)\leq C(1+\abs{x}^2)\mbox{ for all }x\in\mathbb{R}.
\tag{\textbf{SQ}}
\end{equation} 
An admissible potential $K$ has a \emph{sub-linear gradient} if there exists $C>0$ such that 
\begin{equation}\label{SL}
K'(x)\leq C(1+|x|)\mbox{ for all }x\in \mathbb{R}.\tag{\textbf{SL}} 
\end{equation}
We call an admissible potential \emph{attractive} if 
\begin{equation}\label{XX}
K(x)=k(|x|)\geq 0,\mbox{ for all }x\in\mathbb{R}\mbox{ and }K'(r)r\geq 0\mbox{ for all }r\in\mathbb{R}.\tag{\textbf{AT}}  
\end{equation}




\end{definition}

In Section \ref{sec:newtonian} we will also take into account the action of \emph{external} potentials in the dynamics. More precisely, we consider  $A\,\in C^2(\mathbb{R})$ and assume that there exist the positive constants $\lambda$ and $\alpha$ such that 
\begin{equation}\label{H1}
 A(x)\geq\lambda \abs{x}^2\tag{\textbf{H1}}   \end{equation}
 and
 \begin{equation} \label{H2}
    xA'(x)\geq \alpha \abs{x}^2 \tag{\textbf{H2}}  
\end{equation}
for all $ x\in\mathbb{R}$.

Denoting with $\langle\cdot,\cdot\rangle_{\Ldod}$ the inner product on the space $\Ldod$, that is 
\[
\langle Z_1,Z_2\rangle_{\Ldod}=\into \big[ X_1(s)X_2(s)+Y_1(s)Y_2(s)\big]\, ds,
\]
for $Z_1=(X_1,Y_1)$ and $Z_2=(X_2,Y_2)$ in $\Ldod$, we recall below the notion of Fréchet sub-differential for a generic operator $\mathfrak{F}$ on a general Hilbert space.
\begin{definition} Let $H$ be a Hilbert space. For a given, proper and lower semi-continuous functional $\mathfrak{F}:H\rightarrow (-\infty,+\infty]$, we say that $Z\in H$ belongs to the sub-differential of $\mathfrak{F}$ at $\widetilde{Z}\in H$ if and only if
\[ \mathfrak{F}(R)-\mathfrak{F}(\widetilde{Z})\geq\langle Z,R-\widetilde{Z}\rangle_{H} +o(\norma{R-\widetilde{Z}}),
\]
as $\norma{R-\widetilde{Z}}\rightarrow 0$, with  $R\in H$. The sub-differential of $\mathfrak{F}$ at $\widetilde{Z}$ is denoted by $\partial\mathfrak{F}(\widetilde{Z}).$
\end{definition}

In particular, throughout the paper, we will usually consider  as Hilbert spaces $H=L^2(\Omega)$ or $H=L^2(\Omega)^2.$
 
Let $I_\mathcal{K}:L^2(\Omega )\rightarrow [0,+\infty )$ be the indicator function of the $L^2$-convex cone $\mathcal{K}$ introduced in \eqref{cone}, that is
\[
I_\mathcal{K}(X)=
\begin{cases} 0 &\text{if}\; X\in\mathcal{K}, \\ +\infty &\text{otherwise}.
\end{cases}
\]
For a given $X\in L^2(\Omega)$, the sub-differential of $I_\mathcal{K}$ in $X$ is given by
\[ 
\partial I_\mathcal{K}(X)= \bigg\{Z\in L^2(\Omega )\, :\, I_\mathcal{K}(\widetilde{X})\geq I_\mathcal{K}(X)+\int_\Omega Z(m)(\widetilde{X}(m)-X(m))\,dm \quad \text{for all} \; \widetilde{X}\in\mathcal{K}\bigg\},
\]
or in its alternative form
\[
\partial I_\mathcal{K}(X)=
\begin{cases} \{ Z\in L^2(\Omega )\, :\, 0\geq\int_\Omega Z(m)(\widetilde{X}(m)-X(m))\,dm \quad \text{for all} \; \widetilde{X}\in\mathcal{K}\} &\text{if}\,X\in\mathcal{K}, \\ \emptyset & \text{otherwise}.
\end{cases}
\]

We conclude this subsection with the following  definition, which we borrow from \cite{gangbo}.
\begin{definition}\label{def:bounded_unif_cont} An operator $F:\mathcal{K}\to L^2(\Omega )$ is \textit{bounded} if there exists a constant $C\geq0$ such that
\[ \norma{F[X]}_{L^2(\Omega )}\leq C(1+\norma{X}_{L^2(\Omega )}) \qquad \text{for all} \; X\in\mathcal{K}. \]
An operator $F:\mathcal{K}\to L^2(\Omega )$ is \textit{pointwise linearly bounded} if there exists a constant $C_p\geq0$ such that
\[ \abs{F[X](m)}\leq C_p\big(1+\abs{X(m)}+\norma{X}_{L^1(\Omega)}\big) \qquad \text{for a.e.}\; m\in\Omega\;\text{and all}\;X\in\mathcal{K}. \]
An operator $F:\mathcal{K}\to L^2(\Omega )$ is \textit{uniformly continuous} if there exists a modulus of continuity $\omega$ such that
\[ \norma{F[X_1]-F[X_2]}_{L^2(\Omega)} \leq\omega \big(\norma{X_1-X_2}_{L^2(\Omega)}\big) \qquad \text{for all}\; X_1, X_2\in\mathcal{K}.\]
\end{definition}

\subsection{Particles system}
\label{subsec:particlessystem}
We dedicate this subsection to the study of sticky solutions in the finite dimensional case. Let $x=(x_1,\ldots,x_N)\in\mathbb{R}^N$ and $y=(y_1,\ldots,y_M)\in\mathbb{R}^M$ be the positions of particles of the first and second species respectively. The ``sticky'' condition clearly preserves the ordering of the particles, therefore  their evolution is confined in the closed convex set
\[ \mathbb{K}^N\times\mathbb{K}^M=\{ (x,y)\in\mathbb{R}^N\times\mathbb{R}^M \, :\, x_1\leq\ldots\leq x_N,\, y_1\leq\ldots\leq y_M\}.
\]
Setting $v=(v_1,\ldots ,v_N)\in\mathbb{R}^N$ and $w=(w_1,\ldots ,w_M)\in\mathbb{R}^M$ as the velocity vectors of particles of the first species and second species respectively, we consider the following system 
\begin{equation}\label{eq:particles}
\begin{cases} \dot{x}_i(t)=v_i(t), \\
\dot{y}_j(t)=w_j(t), \\
\dot{v}_i(t)=a_{i}(x(t))+b_{i}(x(t),y(t))-\sigma v_i(t), \\ 
 \dot{w}_j(t)=c_{j}(y(t))+d_{j}((x(t),y(t))-\sigma w_j(t),
\end{cases}\,
\end{equation}
for $i=1,\ldots ,N$ and $j=1,\ldots ,M.$ In system \eqref{eq:particles}, 
\begin{align*}
    & a_i(x)=-\sum_{k=1}^N m_k K'_\rho(x_i-x_k)\,,\quad b_i(x,y)=-\sum_{k=1}^M n_k H'_\rho(x_i-y_k)\,,\quad i=1,\ldots,N\,,\\
    & c_j(y)=-\sum_{k=1}^M n_k K'_\eta(y_i-y_k)\,,\quad d_j(x,y)=-\sum_{k=1}^N m_k H_\eta'(y_j-x_k)\,,\quad j=1,\ldots,M\,.
\end{align*}
The vector field 
\[
a(x):x\in\mathbb{K}^N\rightarrow \big(a_{1}(x),\ldots ,a_{N}(x)\big)\in\mathbb{R}^N
\]
models the interactions between particles of the first species and the $i$-th particle of the first species, while the $i$-th component of the vector field \[
b(x,y):(x,y)\in\mathbb{K}^N\times\mathbb{K}^M\rightarrow(b_{1}(x,y),\ldots ,b_{N}(x,y))\in\mathbb{R}^N
\]describes the interactions between the $i$-th particle of the first species and particles of the second species. Similarly one can describe the $j$-th component of the terms \[
c(y):y\in\mathbb{K}^M\rightarrow (c_{1}(y),\ldots ,c_{M}(y))\in\mathbb{R}^M,
\]
and 
\[
d(x,y):(x,y)\in\mathbb{K}^N\times\mathbb{K}^M\rightarrow (d_{1}(x,y),\ldots ,d_{M}(x,y))\in\mathbb{R}^M,
\]
respectively. 

Assuming that all the potentials in \eqref{eq:particles} are smooth enough (for example with $C^2$ regularity), a unique solution to \eqref{eq:particles} exists as long as particles occupy distinct positions. When two or more particles collide, we apply the concept of sticky particle solution sketched in the introduction. Following \cite{gangbo,ns}, the precise formalisation of sticky collisions requires the definition of the following normal cones
\begin{align*}
N_x\mathbb{K}^N&\coloneqq \{l\in\mathbb{R}^N \, :\, l\cdot (\widetilde{x}-x)\leq 0 \quad \text{for all} \; \widetilde{x}\in\mathbb{K}^N \}, \\ N_y\mathbb{K}^M&\coloneqq \{n\in\mathbb{R}^M \, :\, n\cdot (\widetilde{y}-y)\leq 0 \quad \text{for all}\;\widetilde{y}\in\mathbb{K}^M \}.
\end{align*}
Note that the normal cone $N_x\mathbb{K}^N$ is equal to the sub-differential $\partial I _{\mathbb{K}^N}(x)$ of the indicator function of $\mathbb{K}^N$ at the point $x$. When two particles of the same species collide, an instantaneous force is released and the respective particles velocities evolve as elements of the normal cones $N_x\mathbb{K}^N$ and $N_y\mathbb{K}^M$ respectively.
Given these premises, we can consider the second-order system of differential inclusions
\begin{equation}
\label{incl1}
\begin{dcases} \dot{x}=v, \\ \dot{y}=w, \\ \dot{v}+N_x\mathbb{K}^N\ni a(x)+b(x,y)-\sigma v, \\ \dot{w}+N_y\mathbb{K}^M\ni c(y)+d(x,y)-\sigma w.
\end{dcases}
\end{equation}
System \eqref{incl1} is justified as follows. Introducing the vector $\mathcal{W}(t)=(V(t),W(t))=e^{\sigma t}(v(t),w(t))$, from \eqref{eq:particles} we get
\[\dot{\mathcal{W}}(t)=e^{\lambda t}\mathcal{A}(x(t),y(t))\,,\]
where $\mathcal{A}(x,y)$ is the vector in $\mathbb{R}^{N+M}$ with components $a(x)+b(x,y)$ and $c(y)+d(x,y)$ respectively. Now, due to the smoothness of the interaction potentials, the vector field $\mathcal{A}(x,y)$ can be extended by continuity to the boundary of the cone $\mathbb{K}^N\times \mathbb{K}^M$. Therefore, as $\mathcal{W}$ and $(v,w)$ only differ by a scalar factor, a suitable modified version of the differential equation for $\mathcal{W}$ that keeps the dynamics in $\mathbb{K}^N\times \mathbb{K}^M$ is the differential inclusion
\[\dot{\mathcal{W}}(t)\in e^{\lambda t}\mathcal{A}(x(t),y(t)) + N_{x(t)}\mathbb{K}^N\times N_{y(t)}\mathbb{K}^M\,,\]
which easily yields the last two differential inclusions in \eqref{incl1}.

According to \cite{gangbo}, if $x:[0,\infty)\to\mathbb{K}^N$ satisfies the \textit{global sticky condition}, i.e., particles are not allowed to split after colliding, then the following monotonicity property on the family of normal cones $N_{x(t)}\mathbb{K}^N$ holds:
\[
N_{x(s)}\mathbb{K}^N\subset N_{x(t)}\mathbb{K}^N\qquad \text{for all $s<t$.}
\]

Hence, for any function $\zeta: [0,\infty)\to\mathbb{R}^N$ such that $\zeta(t)\in N_{x(t)}\mathbb{K}^N$, we have
\[ \int_s^t \zeta (r)\,dr\in N_{x(t)}\mathbb{K}^N \qquad \text{for all $s<t$.} \]
Consequently, integrating the last two equations in \eqref{incl1} on a time interval $[s,t]$, one obtains
\begin{align}
v(t)+\sigma x(t)+N_{x(t)}\mathbb{K}^N&\ni v(s)+\sigma x(s) +\int_s^ta(x(r))\,dr+\int_s^tb(x(r),y(r))\,dr,\label{eq:diff_incl_1}  \\ 
w(t)+\sigma y(t)+N_{y(t)}\mathbb{K}^M&\ni w(s)+\sigma y(s) +\int_s^tc(y(r))\,dr+\int_s^t d(x(r),y(r))\,dr.\label{eq:diff_incl_2} 
\end{align}
System \eqref{incl1}, together with \eqref{eq:diff_incl_1} and \eqref{eq:diff_incl_2}, can be rewritten in a more compact form in the new variables $(x,y,p,q)$ where $p$ and $q$ are defined by
\begin{align*}
    p(t)&=\int_s^t a(x(r))\,dr + \int_s^t b(x(r),y(r))\,dr+v(s)+\sigma x(s), \\ q(t)&=\int_s^t c(y(r))\,dr + \int_s^t d(x(r),y(r))\,dr+w(s)+\sigma y(s),
\end{align*} 
yielding the following first-order system of differential inclusions
\[
\begin{dcases} \dot{x}+\sigma x+N_x\mathbb{K}^N\ni p, \\ \dot{y}+\sigma y+N_y\mathbb{K}^M\ni q, \\ \dot{p}=a(x)+b(x,y), \\ \dot{q}=c(y)+d(x,y),
\end{dcases}
\]
with the additional characterisation of $v$ and $w$ in terms of $p$ and $q$
\begin{align*} v(t)+\sigma\int_s^t v(r)\,dr+N_{x}\mathbb{K}^N &\ni p(t), \\
w(t)+\sigma\int_s^t w(r)\,dr+N_{y}\mathbb{K}^M &\ni q(t) .
\end{align*} 

\subsection{Time scaling and formal large damping limit}\label{sec:large_dam}
One of the purposes of the present work is to study system \eqref{sistema1} in the \emph{large time / large damping regime}, namely we aim to send $\sigma\to +\infty$ in \eqref{sistema1} after having suitably rescaled the time variable. We start performing the scaling at the level of particles, namely for system \eqref{neweq}.
Consider the new time variable $\tau$ defined by
\begin{equation}
\label{timescale_part}
\tau=\frac{t}{\sigma}, 
\end{equation}
and introduce the scaled particle trajectories as follows: 
\begin{align*}
         x_i(t)&=\chi_i(\tau)=\chi_i(t/\sigma),\quad i=1,\ldots,N, \\  y_j(t)&=\xi_j(\tau)=\xi_j(t/\sigma),\quad j=1,\ldots,M. 
\end{align*}
Notice that we can scale the initial velocities accordingly as
\[
 \dot{\chi}_i(0)\coloneqq \overline{\nu}_i=\sigma\overline{v}_i, \qquad\dot{\xi}_j(0)\coloneqq \overline{\omega}_j=\sigma\overline{w}_j.
\]
Hence, system \eqref{neweq} becomes
\begin{gather*}
\sigma^{-2}\ddot{\chi}_i(\tau)=-\dot{\chi}_i(\tau)-\sum_{k\neq i} m_k  K'_\rho\big(\chi_i(\tau)-\chi_k(\tau)\big)-\sum_kn_k  H'_\rho\big( \chi_i(\tau)-\xi_k(\tau)\big), \\ \sigma ^{-2}\ddot{\xi}_j(\tau)=-\dot{\xi}_j(\tau)-\sum_{k\neq j}n_k  K'_\eta\big(\xi_j(\tau)-\xi_k(\tau)\big)-\sum_km_k  H'_\eta \big(\xi_j(\tau)-\chi_k(\tau)\big).
\end{gather*}
A formal limit $\sigma\to+\infty$ leads to the following first-order system of differential equations for particle positions
\begin{gather*}
\dot{\chi}_i(\tau)=-\sum_{k\neq i} m_k K'_\rho \big(\chi_i(\tau)-\chi_k(\tau)\big)-\sum_kn_k H'_\rho \big(\chi_i(\tau)-\xi_k(\tau)\big), \\ \dot{\xi}_j(\tau)=-\sum_{k\neq j} n_k K'_\eta \big(\xi_j(\tau)-\xi_k(\tau)\big)-\sum_km_k H'_\eta \big(\xi_j(\tau)-\chi_k(\tau)\big).
\end{gather*}
A similar time scaling can be performed at the level of \eqref{sistema1}. Using the definition of $\tau$ in \eqref{timescale_part} and considering $(\widetilde{\rho},\widetilde{v},\widetilde{\eta},\widetilde{w})$ solution to
\begin{equation*} 
\begin{dcases} 
\frac{\partial \widetilde{\rho}}{\partial t} +\frac{\partial}{\partial x}(\widetilde{\rho}\widetilde{v})=0, \\ \frac{\partial \widetilde{\eta}}{\partial t}+\frac{\partial}{\partial x} (\widetilde{\eta} \widetilde{w})=0, \\ \frac{\partial}{\partial t} (\widetilde{\rho}\widetilde{v})+\frac{\partial}{\partial x}(\widetilde{\rho}\widetilde{v}^2)=-\sigma \widetilde{\rho}\widetilde{v} -\widetilde{\rho} [K_\rho'\ast \widetilde{\rho} +H_\rho'\ast \widetilde{\eta}], \\ \frac{\partial}{\partial t} (\widetilde{\eta}\widetilde{w})+\frac{\partial }{\partial x} (\widetilde{\eta}\widetilde{w}^2)=-\sigma \widetilde{\eta}\widetilde{w} -\widetilde{\eta} [K_\eta'\ast \widetilde{\eta} +H_\eta'\ast \widetilde{\rho} ],
\end{dcases}
\end{equation*}
we can introduce the rescaled densities and velocities as
\begin{align*}
& \rho(\tau, x)=\widetilde{\rho}(t,x), \qquad v(\tau,x)=\sigma \widetilde{v}(t,x), \\ 
& \eta(\tau ,x)=\widetilde{\eta}(t,x), \qquad w(\tau,x)=\sigma \widetilde{w}(t,x).
\end{align*}
Then the quadruple $(\rho,v,\eta,w)$ solves
\begin{equation}
\label{eulerresc}
\begin{dcases} \frac{\partial \rho}{\partial \tau}+\frac{\partial}{\partial x}(\rho v)=0, \\ \frac{\partial \eta}{\partial \tau}+\frac{\partial}{\partial x}(\eta w)=0, \\ \sigma^{-2}\bigg[ \frac{\partial}{\partial \tau}(\rho  v)+\frac{\partial}{\partial x}(\rho v^2)\bigg]=-\rho v-\rho[K_{\rho }'\ast\rho+H_{\rho }'\ast\eta ], \\ \sigma ^{-2}\bigg[ \frac{\partial}{\partial \tau}(\eta w)+\frac{\partial}{\partial x}(\eta w^2)\bigg]=-\eta w -\eta [K_{\eta }'\ast\eta +H_{\eta }'\ast\rho ],
\end{dcases}
\end{equation}
and formally, as $\sigma\to\infty$, we have
\begin{equation}
\label{systlimit}
\begin{dcases} \frac{\partial \rho}{\partial \tau}=\frac{\partial}{\partial x}\big[ \rho K_\rho '\ast\rho +\rho H_\rho '\ast\eta \big], \\ \frac{\partial\eta}{\partial \tau} =\frac{\partial}{\partial x}\big[ \eta K_\eta '\ast\eta +\eta H_\eta '\ast\rho \big].
\end{dcases}
\end{equation}

\subsection{Lagrangian description of the continuum model}
\label{subsec:lagrangiandescription}
We now transpose the  considerations above in terms of a Lagrangian description for system \eqref{sistema1}. 
For any $X\in\mathcal{K}$, where $\mathcal{K}$ denotes the convex cone introduced in \eqref{cone}, we define the set
\begin{equation}
\Omega_X\coloneqq \left\{ m\in \Omega \, : \, \mbox{$X$ is constant in an open neighborhood of $m$} \right\},
\end{equation}
and the closed subspace
\begin{equation} 
\mathcal{H}_X=\{ Z\in L^2(0,1) \, :\, \text{$Z$ is constant on each interval $(a,b)\in\Omega_X$}\}. 
\end{equation}
A crucial quantity in the following analysis is the projection $\mathsf{P}_{\mathcal{H}_X}:L^2\rightarrow \mathcal{H}_X$ given by
\begin{equation}\label{projoper}
\mathsf{P}_{\mathcal{H}_X}(U)=\begin{cases}
\displaystyle \fint_a^b U(m)\,dm & \text{in any maximal interval $(a,b)\in\Omega_X$},\\
\displaystyle
U &\text{a.e. in $\Omega\setminus\Omega_X$}, 
\end{cases}
\end{equation}
for all $U\in L^2(\Omega)$. The proof of the following Lemma is an easy consequence of Jensen's inequality, see \cite[Lemma 2.2]{gangbo}.
\begin{lemma}[$\mathcal{H}_X$-contraction] \label{lemmacontraction}  Let $\psi :\mathbb{R}\to [0,\infty)$ be a convex l.s.c. function. Then $\mathsf{P}_{\mathcal{H}_X}$ is dominated by $X$, namely
\[ \int_\Omega \psi\big(\mathsf{P}_{\mathcal{H}_X}(Y)\big)\,dm\leq \int_\Omega \psi (Y)\,dm \qquad \text{for all $X\in\mathcal{K}$ and all $Y\in L^2(\Omega)$}, \]
and we write $\mathsf{P}_{\mathcal{H}_X}\prec X$.
\end{lemma}

Consider a quadruple $(\rho,\eta,v,w)$ solution to \eqref{sistema1} and define the maps $X,Y:[0,\infty)\times \Omega\to\mathbb{R}$ and the velocities $V,W:[0,\infty)\times \Omega\to\mathbb{R}$ as follows
\begin{align*}
&X(t,\cdot)=\Psi(\rho (t,\cdot)), \quad
&V(t,\cdot)=v(t,X(t,\cdot))=\partial _tX(t,\cdot),
\\
&Y(t,\cdot)=\Psi(\eta (t,\cdot)), \quad  &W(t,\cdot)=w(t,Y(t,\cdot))=\partial _tY(t,\cdot),
\end{align*}
where $\Psi$ is the isometry defined in \eqref{isometry} that associates to a probability measure its monotone rearrangement. In the new unknowns $(X,Y,V,W),$ system \eqref{sistema1} can be (formally) rephrased as
\[
\begin{dcases}
\partial_t X(t)=V(t), \\
\partial_t Y(t)=W(t), \\
\partial_t V(t)=-\int_\Omega K_\rho' \big(X(m)-X(m')\big)\,dm'-\int_\Omega H_\rho'\big( X(m)-Y(m')\big)\,dm' -\sigma V(t),
\\
\partial_t W(t)=-\int_\Omega K_\eta' \big(Y(m)-Y(m')\big)\,dm'-\int_\Omega H_\eta'\big( Y(m)-X(m')\big)\,dm'-\sigma W(t) .
\end{dcases}
\]
Similarly to Section \ref{subsec:particlessystem}, one can show that the previous system can be reformulated in terms of differential inclusions to incorporate particles collisions. Moreover, since we will investigate on the large-damping limit, through the paper we consider the Lagrangian counterpart of the rescaled system \eqref{eulerresc}. Then, according to the previous calculations, we get   the system
\begin{equation}
\label{sisteps}
\begin{dcases}
\varepsilon\dot{X}(t,m)+X(t,m)+\partial I_\mathcal{K}(X(t,m))\ni\varepsilon\overline{V}(m)+\overline{X}(m) +\int_0^t F[X(\cdot,r),Y(\cdot,r)](m)\,dr, \\ \varepsilon\dot{Y}(t,m)+Y(t,m)+\partial I_\mathcal{K}(Y(t,m))\ni\varepsilon\overline{W}(m)+\overline{Y}(m) +\int_0^t G[X(\cdot,r),Y(\cdot,r)](m)\,dr,
\end{dcases}
\end{equation}
with  $\varepsilon \coloneqq \sigma ^{-2}$ and where we have denoted by 
\[F:\mathcal{K}\times\mathcal{K}\rightarrow L^2(\Omega)\,\mbox{ and }\, G:\mathcal{K}\times\mathcal{K}\rightarrow L^2(\Omega)\]
the operators
\begin{align}
    & F[X,Y](m)= -\int_\Omega K_\rho ' \big(X(r,m)-X(r,m')\big)\,dm'-\int_\Omega H_\rho' \big(X(r,m)-Y(r,m')\big)\,dm', \label{eq:F}\\
    & G[X,Y](m)=-\int_\Omega K_\eta ' \big(Y(r,m)-Y(r,m')\big)\,dm'-\int_\Omega H_\eta' \big(Y(r,m)-X(r,m')\big)\,dm'\,. \label{eq:G}
\end{align}

We observe that if $K_\rho, H_\rho, K_\eta, H_\eta$ are $C^1$ functions that satisfy \eqref{A} and \eqref{SL} then the two operator $F$ and $G$ defined in \eqref{eq:F} and \eqref{eq:G} are uniformly continuous and bounded  according to Definition \ref{def:bounded_unif_cont}.

\begin{definition}[Lagrangian solutions]
Let $H_\rho, K_\rho, H_\eta, K_\eta \in C^1(\mathbb{R})$ potentials satisfying \eqref{A} and \eqref{SL}. Let $\overline{X},\overline{Y}\in\mathcal{K}$ and  $\overline{V},\overline{W} \in L^2(\Omega )$ be given. A \textit{Lagrangian solution} to \eqref{sisteps} with initial data $(\overline{X},\overline{Y},\overline{V},\overline{W})$ is a pair $(X,Y)\in \text{Lip}_\textsubscript{loc}([0,\infty);\mathcal{K})\times\text{Lip\textsubscript{loc}}([0,\infty);\mathcal{K})$ satisfying $X(0)=\overline{X}, Y(0)=\overline{Y}$ and \eqref{sisteps} for a. e. $t\in [0,\infty).$
\end{definition}


In order to consider the case of Newtonian potentials, we introduce the following notion of \emph{generalised Lagrangian solutions} for system \eqref{sisteps} under globally sticky dynamics, see \cite{gangbo}. 

\begin{definition}\label{def:generalised}
A \textit{generalised solution} to the system
\eqref{sisteps} is a pair $(X,Y)\in\text{Lip}_\textsubscript{loc}([0,\infty);\mathcal{K})\times\text{Lip}_\textsubscript{loc}([0,\infty);\mathcal{K})$ such that
\begin{enumerate}
    \item \textit{Differential inclusion:}
    \[
    \begin{dcases} \varepsilon\dot{X}(t)+X(t)+\partial I_\mathcal{K}(X(t))\ni\varepsilon \overline{V}+\overline{X}+\int_0^t \Theta(s)\,ds, 
    \\
    \varepsilon\dot{Y}(t)+Y(t)+\partial I_\mathcal{K}(Y(t))\ni \varepsilon\overline{W}+\overline{Y}+\int_0^t \Xi(s)\,ds,
    \end{dcases}
    \]
    holds for a.e. $t\in (0,\infty),$
    for some maps $\Theta,\Xi\in L_ \textsubscript{loc}^\infty([0,\infty);L^2(\Omega))\times L_ \textsubscript{loc}^\infty([0,\infty);L^2(\Omega))$ with
    \begin{equation} 
    \label{eq:theta}
    \Theta-F[ X(t),Y(t)]\in H_{X(t)}^\bot \qquad\text{and}\qquad \Theta\prec F[X(t),Y(t)] \quad\text{for a.e. $t\in (0,\infty)$} 
    \end{equation}
    and, similarly,
    \begin{equation} 
    \label{eq:xi}
    \Xi-G[ X(t),Y(t)]\in H_{Y(t)}^\bot \qquad\text{and}\qquad \Xi\prec G[X(t),Y(t)] \quad\text{for a.e. $t\in (0,\infty),$} 
    \end{equation}
    where $F[ X(t),Y(t)]$ and $G[ X(t),Y(t)]$ are the operators defined in \eqref{eq:F} and \eqref{eq:G}.
    \item \textit{Semigroup property:} for all $t\geq t_1\geq 0,$ the right derivatives $V=\frac{d^+}{dt}X$ and $W=\frac{d^+}{dt}Y$ satisfy

    \begin{numcases}{}
       \varepsilon V(t)+X(t)+\partial I_\mathcal{K}( X(t)) \ni \varepsilon V(t_1)+X(t_1)+\int_{t_1}^t \Theta(s)\,ds, \label{def:semigroupx}\\ \varepsilon W(t)+Y(t)+\partial I_\mathcal{K}(Y(t))\ni \varepsilon W(t_1)+Y(t_1)+\int_{t_1}^t \Xi(s)\,ds.
       \label{def:semigroupy}
    \end{numcases}
    \item \textit{Projection formula:} for all $t\geq t_1\geq 0$
    \begin{numcases}{} X(t)=\mathsf{P}_\mathcal{K} \bigg(X(t_1)+ \frac{1}{\varepsilon}
    (t-t_1)\big(X(t_1)+\varepsilon V(t_1)\big) -\frac{1}{\varepsilon}\int_{t_1}^t X(s)\,ds +\frac{1}{\varepsilon}\int_{t_1}^t (t-s)\Theta(s)\,ds \bigg), \label{def:projx}
    \\
    Y(t)=\mathsf{P}_\mathcal{K} \bigg( Y(t_1)+ \frac{1}{\varepsilon} (t-t_1)\big(Y(t_1)+\varepsilon W(t_1)\big) -\frac{1}{\varepsilon}\int_{t_1}^t Y(s)\,ds +\frac{1}{\varepsilon}\int_{t_1}^t (t-s)\Xi(s)\,ds \bigg). \label{def:projy}
    \end{numcases}
\end{enumerate}
\end{definition}
Note that if we choose $\Theta (t)\coloneqq F[X(t),Y(t)]$ and $\Xi (t) \coloneqq G[X(t),Y(t)]$ with $F$ and $G$ as in \eqref{eq:F} and \eqref{eq:G} and the interaction potentials $K_\rho$, $K_\eta$, $H_\rho$ and $H_\eta$ satisfying \eqref{A} and \eqref{SL}, then any Lagrangian solution is a generalised Lagrangian solution.

In the following we will make use of the auxiliary variables
\begin{equation}
\label{eq:P}
P(t,m)=\varepsilon\overline{V}(m)+\overline{X}(m) +\int_0^tF[X(\cdot,r),Y(\cdot,r)](m)\,dr,\\
\end{equation}
and
\begin{equation}
\label{eq:Q}
Q(t,m)=\varepsilon\overline{W}(m)+\overline{Y}(m) +\int_0^t G[X(\cdot,r),Y(\cdot,r)](m)\,dr,
\end{equation}
that allow to rephrase  system \eqref{sisteps} in the equivalent form
\begin{equation}
\label{sistpq}
\begin{dcases}
\varepsilon \dot{X}+X+\partial I_\mathcal{K}(X)\ni P, \\ \varepsilon \dot{Y}+Y+\partial I_\mathcal{K}(Y)\ni Q, \\ \dot{P}=F[X,Y], \\ \dot{Q}=G[X,Y].
\end{dcases}
\end{equation}

\subsection{Main results}\label{subsec:mainres}
We collect in this subsection the main results presented in the paper. The first result concerns the well-posedness of system \eqref{sistema1} in the $2$-Wasserstein space of probability measures and in the sense of \emph{sticky solutions}, under smoothness assumptions on the interaction kernels.

\begin{teorema}
\label{main_th}
Let $T>0$ and suppose that the kernels $H_\rho, K_\rho, H_\eta, K_\eta\in C^1(\mathbb{R})$ satisfy \eqref{A} and \eqref{SL}. Let $\overline{\rho}, \overline{\eta} \in \mathcal{P}_2(\mathbb{R})$ and $\overline{v}\in L^2(d\overline{\rho})$ and  $\overline{w}\in L^2(d\overline{\eta})$. Then, there exists a unique quadruple $$(\rho, \eta, v, w) \in \emph{Lip}\big([0,T]; \mathcal{P}_2(\mathbb{R})\times\mathcal{P}_2(\mathbb{R})\times L^2(d\rho(t))\times L^2(d\eta(t))\big)$$ that is a distributional solutions to system \eqref{sistema1} 
such that
\begin{align*}
&\lim_{t\downarrow 0}\rho (t,\cdot )=\overline{\rho} \quad \textit{in} \quad \mathcal{P}_2(\mathbb{R}), \qquad &\lim_{t\downarrow 0}\rho (t,\cdot)v(t,\cdot )=\overline{\rho}\overline{v} \quad \textit{in}\quad \mathcal{M}(\mathbb{R}), \\ &\lim_{t\downarrow 0}\eta (t,\cdot )=\overline{\eta} \quad \textit{in} \quad \mathcal{P}_2(\mathbb{R}), \qquad &\lim_{t\downarrow 0}\eta (t,\cdot)w(t,\cdot )=\overline{\eta}\overline{w} \quad \textit{in}\quad \mathcal{M}(\mathbb{R}).
\end{align*}
\end{teorema}

We then address the $\sigma\to \infty$ limit of \eqref{sistema1} towards \eqref{sistema2} using the rescaling in \eqref{eulerresc}, making rigorous the formal argument presented in Section 2.4. This task is performed at the level of the Lagrangian system \eqref{sisteps} sending the parameter $\varepsilon=\sigma^{-2}\to 0$, coming back to the Eulerian description through the isometry \eqref{isometry}. The following result is proved in Section \ref{sec:large_damping}.

\begin{teorema}\label{th:large_damping}
Let $T>0$ and suppose that the kernels $H_\rho, K_\rho, H_\eta, K_\eta\in C^1(\mathbb{R})$ satisfy \eqref{A} and \eqref{SL}. Let $(\rho_\varepsilon,\eta_\varepsilon,v_\varepsilon,w_\varepsilon)$ be  solution to  system \eqref{eulerresc} with $\varepsilon=\sigma^{-2}$ under the initial condition $(\overline{\rho}_\varepsilon,  \overline{\eta}_\varepsilon, \overline{v}_\varepsilon, \overline{w}_\varepsilon)$ and let $(\rho,\eta)$ be  solution to  system \eqref{systlimit} with initial data $(\overline{\rho},\overline{\eta}).$
 Furthermore,  assume that
 \begin{itemize}
     \item[(i)] $\overline{\rho}_\varepsilon\to \overline{\rho}$ and $\overline{\eta}_\varepsilon \to \overline{\eta}$ as $\varepsilon\to 0$ in  $\mathcal{P}_2(\mathbb{R})$;
     \item[(ii)]  $\overline{v}_\varepsilon=o(1/\varepsilon )$ in $L^2(d\overline{\rho}_\varepsilon )$ and $\overline{w}_\varepsilon=o(1/\varepsilon )$ in $L^2(d\overline{\eta}_\varepsilon )$ as $\varepsilon\to 0$.
 \end{itemize}
 Then,
 \begin{equation*}
     \lim _{\varepsilon\to 0} \int_0^T \mathcal{W}_2^2\big( (\rho_\varepsilon,\eta_\varepsilon),(\rho,\eta) \big)\,dt=0.
 \end{equation*}
 \end{teorema}

\begin{rmk}[Initial data are not \emph{well-prepared} in the velocity variable]  In Theorem \ref{th:large_damping}, recalling that $\overline{v}_\varepsilon=\frac{1}{\sqrt{\varepsilon}}\overline{v}$, assumption $(ii)$ is satisfied in case $\overline{v}\in L^2(d\rho)$ and $\overline{w}\in L^2(d\eta)$ are given and independent of $\varepsilon$. Therefore, assumption (ii) is quite general in the context of singular limits. Assumption (i) instead imposes that the initial density should converge to the one of the limiting first order system.
\end{rmk}

Lastly, under the action of Newtonian self-interaction kernels, $K_\rho(x)=K_\eta(x)=N(x)\coloneqq |x|$, \emph{symmetric} and attractive cross-interactions, $H_\rho(x)=H_\eta(x)=H(x)$  and suitably coercive external potentials, we focus on a different aspect, that is the convergence to stationary solutions of \eqref{sistema1}. More precisely, we will consider the following system
\begin{equation} 
\label{sistema1_New}
\begin{dcases} 
\frac{\partial \rho}{\partial t} +\frac{\partial}{\partial x}(\rho v)=0, \\ \frac{\partial \eta}{\partial t}+\frac{\partial}{\partial x} (\eta w)=0, \\ \frac{\partial}{\partial t} (\rho v)+\frac{\partial}{\partial x}(\rho v^2)=-\sigma\rho v -\rho [N'\ast \rho +H'\ast \eta+A_\rho], \\ \frac{\partial}{\partial t} (\eta w)+\frac{\partial }{\partial x} (\eta w^2)=-\sigma\eta w -\eta [N'\ast\eta +H'\ast\rho+A_\eta ],
\end{dcases}
\end{equation}
and its Lagrangian counterpart
\begin{equation}
\label{systempseudoi_intro}
\begin{dcases}
\partial _t X(t,m)=V(t,m), \\ \partial _t Y(t,m)=W(t,m), \\ \partial _t V(t,m)=-\int_\Omega \text{sign}\big(X(t,m)-X(t,m')\big)\,dm' \\ \hspace{1.9cm} -\int_\Omega H'\big(X(t,m)-Y(t,m')\big)\,dm'-\sigma V(t,m) -A_\rho'(X),\\ \partial _t W(t,m)=-\int_\Omega \text{sign}\big(Y(t,m)-Y(t,m')\big) \,dm' \\ \hspace{2cm} -\int_\Omega H'\big(Y(t,m)-X(t,m')\big)\,dm'-\sigma W(t,m) -A_\eta'(Y).
\end{dcases}
\end{equation}
Stationary solutions in this case are $(\rho_s,\eta_s)=(\delta_0,\delta_0)$ where $\delta$ is the Dirac measure, which corresponds to $(X_s,Y_s)=(0,0)$ in terms of the Lagrangian description. 
The last result we present in the paper shows that solutions to \eqref{sistema1_New} converge to the stable stationary solution in the $2-$Wasserstein distance. 

\begin{teorema}\label{th:newtonian}
Let $H$ be an interaction potential under assumptions \eqref{A}, \eqref{SL} and \eqref{XX}. Consider $A_\rho\,,A_\eta\in C^2(\mathbb{R})$ under assumptions \eqref{H1} and \eqref{H2}. Let $(X,Y)\in\emph{Lip}_{\emph{loc}}([0,\infty);\mathcal{K})^2$  be a generalised Lagrangian solution to  \eqref{systempseudoi_intro} in the sense of Definition \ref{def:generalised}. Assume that the initial positions $(\overline{X},\overline{Y})\in\mathcal{K}^2$ and velocities $(\overline{V},\overline{W})\in\left(L^2(\Omega)\right)^2$ satisfy 
\[\norma{\overline{X}}_{L^2}+\norma{\overline{Y}}_{L^2}+\norma{\overline{V}}_{L^2}+\norma{\overline{W}}_{L^2}<\infty,\] then
\[\lim _{t\to\infty}\bigg( \norma{X}_{L^2}+\norma{Y}_{L^2}+\norma{V}_{L^2}+\norma{W}_{L^2} \bigg)=0. \]
Furthermore calling  $\rho (t,\cdot)\coloneqq \Psi^{-1}(X(t,\cdot))$ and $\eta (t,\cdot)\coloneqq \Psi^{-1}(Y(t,\cdot))$, where $\Psi$ is the isometry defined in \eqref{isometry}, we have
\[ \lim_{t\to \infty} \mathcal{W}_2^2\big( (\rho,  \eta),(\rho_s,\eta_s) \big)=0. \]
\end{teorema}

\section{Existence and uniqueness for smooth potentials}\label{sec:existence}

In this section we prove Theorem \ref{main_th}, namely existence and uniqueness of solution to system \eqref{sistema1}. To perform this task, we pass through existence of solutions to the Lagrangian system \eqref{sisteps}, where we apply the theory of Maximal Monotone Operators subject to Lipschitz perturbations in the spirit of \cite[Theorem 3.17]{brezis}, see Proposition \ref{th1} below. The result in the original variables is then proved using the properties contained in Proposition \ref{th:property} below.



We start proving the following Lemma.
\begin{lemma} \label{lemma2}
Let $(X,Y)$,$(\widetilde{X},\widetilde{Y}) \in\mathcal{K}\times\mathcal{K}$ be given. Consider the interaction kernels $H_\rho, K_\rho, H_\eta, K_\eta$ under assumptions \eqref{A} and \eqref{SL} and let $F$ and $G$ be the operators defined in \eqref{eq:F} and \eqref{eq:G} respectively. Then there exist two positive constants $C_1$ and $C_2$ depending on the Lipschitz constants of the kernels, such that 
\begin{enumerate} [label=(\roman{*}), ref=(\roman{*})]
    \item $ \norma{F[X,Y]-F[\widetilde{X},\widetilde{Y}]}^2_{L^2(0,1)}\leq C_1 \big( \norma{X-\widetilde{X}}^2_{L^2(0,1)}+\norma{Y-\widetilde{Y}}^2_{L^2(0,1)} \big), $
    \item $ \norma{G[X,Y]-G[\widetilde{X},\widetilde{Y}]}^2_{L^2(0,1)}\leq C_2 \big( \norma{X-\widetilde{X}}^2_{L^2(0,1)}+\norma{Y-\widetilde{Y}}^2_{L^2(0,1)} \big). $
\end{enumerate}
\end{lemma}
\begin{proof}
We only prove $(i)$ since $(ii)$ follows from a similar argument. By the definition of $F$ in \eqref{eq:F} we have
\begin{align}\label{eq:lemma_proof1}
\begin{aligned}
 &\norma{ F[X,Y]-F[\widetilde{X},\widetilde{Y}]}^2_{L^2(0,1)} \\ &
= \int_\Omega \abss{-\int_\Omega K_\rho'\big( X(r,m)-X(r,m')\big)\,dm'-\int_\Omega H_\rho'\big( X(r,m)-Y(r,m')\big)\,dm' \\ & +\int_\Omega K_\rho'\big( \widetilde{X}(r,m)-\widetilde{X}(r,m')\big)\,dm'+\int_\Omega H_\rho'\big( \widetilde{X}(r,m)-\widetilde{Y}(r,m')\big)\,dm'} ^2 \,dm. 
\end{aligned}
\end{align}
Using the fact that $\abs{x+y}^2\leq 2(\abs{x}^2+\abs{y}^2)$, the right hand side of \eqref{eq:lemma_proof1} can be controlled by 
\begin{align}
\begin{aligned}\label{eq:lemma_proof2}
& 2 \int_\Omega \bigg(  \abss{ \int_\Omega \big[ K_\rho' \big(X(r,m)-X(r,m')\big)-K_\rho'(\widetilde{X}(r,m)-\widetilde{X}(r,m')\big)\big]\,dm'}^2  \\  &+\abss{ \int_\Omega \big[ H_\rho'\big( X(r,m)-Y(r,m')\big)-H_\rho'\big(\widetilde{X}(r,m)-\widetilde{Y}(r,m')\big)\big]\,dm'}^2 \bigg)\,dm  \\ & 
\leq 2 \int_\Omega \bigg( \int_\Omega \abs{ K_\rho' \big(X(r,m)-X(r,m')\big)-K_\rho'(\widetilde{X}(r,m)-\widetilde{X}(r,m')\big)}\,dm'\bigg)^2  \\ &+\bigg( \int_\Omega \abs{ H_\rho' \big(X(r,m)-Y(r,m')\big)-H_\rho'(\widetilde{X}(r,m)-\widetilde{Y}(r,m')\big)}\,dm'\bigg)^2\,dm. 
\end{aligned}
\end{align}
Let $L(K_\rho')$ and $L(H_\rho')$ be the Lipschitz constants of $K_\rho'$ and $H_\rho'$ respectively, then, using Jensen's inequality, the right hand side of \eqref{eq:lemma_proof2} is bounded by
\begin{align*}
& 2\int_\Omega \bigg( \int_\Omega L(K_\rho')\big( \abs{X(r,m)-\widetilde{X}(r,m)}+\abs{X(r,m')-\widetilde{X}(r,m')}\big)\,dm'\bigg)^2 \\&
+ \bigg( \int_\Omega L(H_\rho') \big( \abs{X(r,m)-\widetilde{X}(r,m)} + \abs{Y(r,m')-\widetilde{Y}(r,m')} \big)\,dm' \bigg)^2\,dm \\
& \leq 4 \int_\Omega \bigg( \int_\Omega \big[ L(K_\rho')^2\abs{X(r,m)-\widetilde{X}(r,m)}^2
+L(K_\rho')^2\abs{X(r,m')-\widetilde{X}(r,m')}^2\big]\,dm' \bigg) \\ & 
+\bigg( \int_\Omega \big[ L(H_\rho')^2\abs{X(r,m)-\widetilde{X}(r,m)}^2
+L(H_\rho')^2\abs{Y(r,m')-\widetilde{Y}(r,m')}^2\big]\,dm' \bigg)\,dm.
\end{align*}
Thus, there exists a positive constant $C_1=C_1\big(L(K_\rho'),L(H_\rho')\big)$ such that
\[
\norma{F[X,Y]-F[\widetilde{X},\widetilde{Y}]}^2_{L^2(0,1)}\leq C_1 \big( \norma{X-\widetilde{X}}^2_{L^2(0,1)}+\norma{Y-\widetilde{Y}}^2_{L^2(0,1)} \big).
\]
Analogously, one can prove the inequality $(ii)$, we omit the details.
\end{proof}

We are now ready to state existence result for Lagrangian solution to system \eqref{sisteps}.

\begin{proposizione} \label{th1} Let $T>0$ and suppose that the kernels $H_\rho, K_\rho, H_\eta, K_\eta\in C^1(\mathbb{R})$ satisfy \eqref{A} and \eqref{SL}. Then, for every $(\overline{X},\overline{Y},\overline{V},\overline{W})\in\mathcal{K}^2\times L^2(0,1)^2$ there exists a unique Lagrangian solution $(X,Y)$ to \eqref{sisteps} in $[0,T].$ 
\end{proposizione}

\begin{proof} According to the discussion in Section \ref{subsec:lagrangiandescription}, system \eqref{sisteps} can be rewritten in the following equivalent form
\begin{equation}
\label{exun}
\begin{dcases} \dot{X}+\partial\bigg( I_\mathcal{K}(X)+\frac{\abs{X}^2}{2\varepsilon}\bigg)\ni \frac{P}{\varepsilon}, \\ \dot{Y}+\partial\bigg( I_\mathcal{K}(Y)+\frac{\abs{Y}^2}{2\varepsilon}\bigg)\ni \frac{Q}{\varepsilon}, \\ \dot{P}=F[X,Y], \\ \dot{Q}=G[X,Y],
\end{dcases}
\end{equation}
where $P$ and $Q$ are definded in \eqref{eq:P} and \eqref{eq:Q} respectively. In order to prove the result we will follow the strategy in  \cite[Theorem 3.17]{brezis}. Consider  the  operator 
 \[ \mathcal{A} (X,Y,P,Q)\coloneqq I_\mathcal{K}(X)+I_\mathcal{K}(Y)+\frac{\abs{X}^2}{2\varepsilon}+\frac{\abs{Y}^2}{2\varepsilon}\]
 defined on the Hilbert space $H\coloneqq L^2(\Omega)^2\times L^2(\Omega)^2$. Note that $\mathcal{A}$ is convex and bounded from below. 
Consider the iterative sequence defined as follows: fix $U_0\coloneqq (\overline{X},\overline{Y},\overline{P},\overline{Q}) \equiv (\overline{X},\overline{Y},\varepsilon \overline{V}+\overline{X},\varepsilon\overline{W}+\overline{Y})$ and, for $n\geq1$ construct $U_{n+1}(t)\coloneqq (X_{n+1}(t),Y_{n+1}(t),P_{n+1}(t),Q_{n+1}(t))$ recursively as the weak solution to the implicit-explicit system
\begin{equation}
\label{existence}
\begin{dcases}
\dot{X}_{n+1}+\partial \bigg(I_\mathcal{K}(X_{n+1})+\frac{\abs{X_{n+1}}^2}{2\varepsilon}\bigg)\ni \frac{P_n}{\varepsilon}, \qquad &X_{n+1}(0)=\overline{X}, \\ \dot{Y}_{n+1}+\partial \bigg(I_\mathcal{K}(Y_{n+1})+\frac{\abs{Y_{n+1}}^2}{2\varepsilon}\bigg)\ni \frac{Q_n}{\varepsilon}, \qquad &Y_{n+1}(0)=\overline{Y}, \\ \dot{P}_{n+1}=F[X_n,Y_n], \qquad &P_{n+1}(0)=\overline{P}, \\ \dot{Q}_{n+1}=G[X_n,Y_n], \qquad &Q_{n+1}(0)=\overline{Q}.
\end{dcases}
\end{equation}
Setting $R(U_n)=\big(P_n /\varepsilon, Q_n/\varepsilon, F[X_n,Y_n], G[X_n,Y_n]\big),$ the previous system \eqref{existence} can be rewritten in the following compact form
\begin{equation}\label{eq:compact}
  \dot{U}_{n+1}+\partial\mathcal{A} (U_{n+1})\ni R(U_n).  
\end{equation}
Since the functional $\mathcal{A}$ is convex, its sub-differential is a maximal monotone operator in the sense of \cite{brezis} and $R$ can be seen as a Lipschitz perturbation of it, see \cite[Lemma 3.1]{brezis}. A direct computation shows that
\[
    \frac{1}{2}\frac{d}{dt} \norma{U_{n+1}-U_n}^2_{L^2(0,1)}\leq \big( U_{n+1}-U_n, R(U_n)-R(U_{n-1}) \big),
\]
then
proceeding as in \cite[Lemma A.5]{brezis}, we have that
\[ \norma{U_{n+1}-U_n}_{L^2(0,1)}\leq \int_0^t \norma{R(U_n)-R(U_{n-1})}_{L^2(0,1)}\,dr. \]
Invoking Lemma \ref{lemma2} and the definitions for $P$ and $Q$ in \eqref{eq:P} and \eqref{eq:Q} respectively, we can say that there exists a positive constant $C$ depending on $T$, $\varepsilon$ and on the Lipschitz constants of the kernels $L(K_\rho'),$ $L(H_\rho'),$ $L(K_\eta'),$ $L(H_\eta')$ such that
\[ \norma{U_{n+1}-U_n}_{L^2(0,1)}\leq C \int_0^t \norma{U_n(r)-U_{n-1}(r)}_{L^2(0,1)}\,dr \qquad \text{for $0\leq t\leq T$}. \]
An easy iterative procedure implies that
\[ \norma{U_{n+1}-U_n}_{L^2(0,1)}\leq\frac{(Ct)^n}{n!}\norma{U_1-U_0}_{L^2(0,1)}, \]
thus, $U_n$ uniformly converges on $[0,T]$ to some $U$. Due to the Lemma \ref{lemma2}, $R$ is continuous in $L^2$ in each component. Moreover, since the subdifferential of $\mathcal{A}$ is closed, we can pass to the limit in \eqref{eq:compact} and obtain that $U$ is a weak solution to the system \eqref{exun}. 

Concerning uniqueness, let $U_1=(X_1,Y_1,P_1,Q_1)$ and $U_2=(X_2,Y_2,P_2,Q_2)$ be two solutions to system \eqref{exun} with the same initial condition $\overline{U}_1=\overline{U}_2=\overline{U}$.
Proceeding in an analogous way as before, we can argue  that
\[ \norma{U_1-U_2}_{L^2(0,1)}\leq C\int_0^t\norma{U_1-U_2}_{L^2(0,1)}\,dr \qquad \text{for $\;0\leq t\leq T,$} \]
where the positive constant $C$ depends on $T$, $\varepsilon$, $L(K_\rho'),$ $L(H_\rho'),$ $L(K_\eta'),$ $L(H_\eta').$
This implies that $$\norma{U_1-U_2}_{L^2(0,1)}\leq e^{Ct}\norma{\overline{U}_1-\overline{U}_2}_{L^2(0,1)}=0, $$ that proves the uniqueness.
\end{proof}

The following Proposition collects some properties of Lagrangian solution.

\begin{proposizione}
\label{th:property}
Let $F,G:\mathcal{K}\times\mathcal{K}\to L^2(\Omega)$ be uniformly continuous operators in \eqref{eq:F} and \eqref{eq:G} and let $(X,Y)$ be the Lagrangian solution to \eqref{sisteps}. Then, the following properties hold:
\begin{enumerate}[label=(\roman{*}), ref=(\roman{*})]
\item The right-derivatives
\begin{equation}
\label{derivatedx}
V=\frac{d^+}{dt}X, \qquad W=\frac{d^+}{dt}Y
\end{equation}
exist for all $t\geq 0$.
\item $V$ and $W$ are the unique elements of minimal norm in the closed convex sets 
\begin{equation}
\label{Vmin}
V(t)=\bigg( \frac{1}{\varepsilon} \big(P(t)-\partial I_\mathcal{K}(X(t))-X(t)\big) \bigg) ^\circ
\end{equation}
and
\begin{equation}
\label{Wmin}
W(t)=\bigg( \frac{1}{\varepsilon} \big(Q(t)-\partial I_\mathcal{K}(Y(t))-Y(t)\big) \bigg) ^\circ
\end{equation}
respectively. In particular, by replacing $\dot{X}$ by $V$ and $\dot{Y}$ by $W$, \eqref{sisteps} and \eqref{sistpq} hold for all $t\geq 0$.
\item The functions $t\mapsto V(t)$ and $t\mapsto W(t)$ are right-continuous for all $t\geq 0.$
\item If $\mathcal{T}^0_X\subset (0,\infty)$ and $\mathcal{T}^0_Y\subset (0,\infty)$ denote the subsets of all times at which the maps $s\to\norma{V(s)}_{L^2(\Omega )}$ and $s\to\norma{W(s)}_{L^2(\Omega )}$ respectively are continuous, then $(0,\infty)\setminus \mathcal{T}^0_X$ and $(0,\infty)\setminus \mathcal{T}^0_Y$ are negligible, $V$ and $W$ are continuous, $X$ and $Y$ are differentiable in $L^2(\Omega )$ at every point of $\mathcal{T}^0_X$ and $\mathcal{T}^0_Y$ respectively.
\item Setting $\rho (t,\cdot)\coloneqq \Psi^{-1}(X(t,\cdot))$ and $\eta (t,\cdot)\coloneqq \Psi^{-1}(Y(t,\cdot))$ where $\Psi$ is the isometry introduced in \eqref{isometry}, there exist a unique map $v(t,\cdot)\in L^2(\mathbb{R},\rho)$ and a unique map $w(t,\cdot)\in L^2(\mathbb{R},\eta)$ such that
\begin{equation}
\label{xpunto}
\dot{X}(t)=V(t)=\mathsf{P}_{\mathcal{H}_{X(t)}} \bigg(\frac{1}{\varepsilon}\big(P(t)-X(t)\big)\bigg)=v(t,X(t))\in\mathcal{H}_{X(t)},
\end{equation}
for every $t\in \mathcal{T}^0_X$, and
\begin{equation}
\label{ypunto}
\dot{Y}(t)=W(t)=\mathsf{P}_{\mathcal{H}_{Y(t)}} \bigg(\frac{1}{\varepsilon}\big(Q(t)-Y(t)\big)\bigg)=w(t,Y(t))\in\mathcal{H}_{Y(t)}, 
\end{equation}
for every $t\in \mathcal{T}^0_Y$.
\end{enumerate} 
\end{proposizione}
\begin{proof}
The results in $(i), (ii), (iii)$ are consequences of the general theory of \cite[Theorem 3.5]{brezis}.
Concerning $(iv)$ and $(v)$, we follow \cite[Theorem 3.5]{gangbo}. We prove only \eqref{xpunto}, since the proof of \eqref{ypunto} is similar.  By applying \cite[Remark 3.9]{brezis}, one can see that if $t$ is a point of differentiability of $X$, the derivative with respect to time of $X$ in $t$ is the projection of $0$ onto the affine space generated by $P(t)-\partial I_\mathcal{K}(X(t))-X(t),$ i.e., the orthogonal projection of $P(t)-X(t)$ onto the orthogonal complement of the space generated by $\partial I_\mathcal{K}(X(t)).$ By using \cite[Lemma 2.5]{gangbo}, we obtain \eqref{xpunto}. Since any element of $\mathcal{H}_{X(t)}$ can be written as $v\circ X$, where $v\in L^2(\Omega )$ is a suitable Borel map, we have that there exists a Borel map $v:[0,\infty )\times\mathbb{R}\to\mathbb{R}$ such that $v(t,\cdot)\in L^2(\mathbb{R},\rho (t,	\cdot))$ and $V(t,\cdot)=v(t,X(t))$ for $t\in\mathcal{T}^0_X.$
\end{proof}

We are now in the position of proving the main result of this Section, namely Theorem \ref{main_th}, that concerns existence and uniqueness of the solution to system \eqref{sistema1}.
\begin{proof}[Proof of Theorem \ref{main_th}]
 Let $\overline{\rho}, \overline{\eta} \in \mathcal{P}_2(\mathbb{R})$ and $\overline{v}\in L^2(d\overline{\rho})$, $\overline{w}\in L^2(d\overline{\eta})$ be given initial conditions. Define the $L^2(\Omega)$-functions $\overline{X}=\Psi(\overline{\rho})$ and  $\overline{Y}=\Psi(\overline{\eta})$  and the compositions $\overline{V}=\overline{v}\circ \overline{X}$ and $\overline{W}=\overline{w}\circ \overline{Y}$.
 Then $(\overline{X},\overline{Y},\overline{V},\overline{W})$ is an admissible initial condition for system \eqref{sisteps}, thus Proposition \ref{th1} ensures existence and uniqueness of a couple $(X,Y)$ that is the Lagrangian solution to \eqref{sisteps}. According to Proposition \ref{th:property} we can define the right-continuous functions $V$ and $W$ such that \eqref{derivatedx} holds for all $t\geq 0$ and introduce $\rho (t,\cdot)\coloneqq \Psi^{-1}(X(t,\cdot))$ and $\eta (t,\cdot)\coloneqq \Psi^{-1}(Y(t,\cdot))$.
Let $v(t,\cdot)$ be the map given by Proposition \ref{th:property} and $\varphi$ be a test function on $(0,T)\times\mathbb{R},$ then
\begin{align}\label{eq:corollario_1}
\begin{aligned}
& \int_o^\infty\int_\mathbb{R} \varepsilon \big( \partial _t\varphi (t,x)+\partial _x\varphi (t,x)v(t,x)\big)v(t,x)\rho (t,dx)\,dt  \\ & =\int_0^\infty \int_\Omega\varepsilon\big(\partial _t\varphi (t,X(t,m))+\partial _x\varphi (t,X(t,m))v(t,X(t,m))\big)v(t,X(t,m))\,dm\,dt. 
\end{aligned}
\end{align}
Using \eqref{xpunto} and integrating by parts, the r.h.s. of \eqref{eq:corollario_1} is equal to
\begin{align}\begin{aligned}
& \int_0^\infty \int_\Omega \bigg(\frac{d}{dt}\varphi (t,X(t,m))\bigg)\big(P(t,m)-X(t,m)\big)\,dm\,dt  \\ &=\int _0^\infty\int_\Omega \varphi(t,X(t,m))\big(\dot{X}(t,m)-\dot{P}(t,m)\big)\,dm\,dt. \label{eq:corollario_2}
\end{aligned}
\end{align}
As proved in Proposition \ref{th:property} we have that $\dot{X}(t,m)=V(t,m)$ and from  the definition of the operator $P(t,m)$ in \eqref{eq:P}, one obtains that \eqref{eq:corollario_2} equals
\begin{align*}
& \int _0^\infty\int_\Omega \varphi(t,X(t,m))\bigg( V(t,m)+\int_\Omega K_\rho'\big(X(s,m)-X(s,m')\big)\,dm' \\& \qquad +\int_\Omega H_\rho '\big(X(s,m)-Y(s,m')\big)\,dm'\bigg)\,dm\,dt\\& =\int_o^\infty\int_\mathbb{R}\varphi(t,x)\big( v(t,x)+K_\rho'\ast\rho (t,x) +H_\rho '\ast\eta (t,x)\big)\,\rho(t,dx)\,dt,
\end{align*}
that is the distributional formulation of the momentum equation in \eqref{sistema1}.
Similarly, for the continuity equation we have
\begin{align*}
&\int _0^\infty \int _0^1 \bigg( \frac{d}{dt}\varphi (t,X(t,m))\bigg)\,dm\,dt \\ & = \int _0^\infty \int _0^1 \big( \partial _t\varphi (t,X(t,m))+\partial _x\varphi (t,X(t,m))V(t,m)\big)\,dm\,dt \\ & = \int_0^\infty\int_\mathbb{R} \big(\partial _t\varphi (t,x)+\varphi _x(t,x)v(t,x)\big)\,\rho (t,dx)\,dt=0.
\end{align*}
Concerning the  initial conditions, since $\lim_{t\downarrow 0}X(t)=\overline{X}$ in $L^2(\Omega )$ for Proposition \ref{th1} and $\overline{X}=\Psi(\overline{\rho})$, we have that $\rho\to\overline{\rho}$ in $\mathcal{P}_2(\mathbb{R})$ as $t\to 0$. Moreover, $\overline{V}=\overline{v}\circ \overline{X},$ so that $\lim_{t\downarrow 0}V(t)=\overline{V}$ in $L^2(\Omega )$, therefore for every $\varphi\in C_b(\mathbb{R})$ we have
\begin{align*}
 & \int_\mathbb{R}\varphi (x)\overline{v}(x)\overline{\rho}(dx)=\int _0^1\varphi(\overline{X}(m))\overline{V}(m)\,dm \\ & = \lim _{t\downarrow 0} \int _0^1\varphi(X(t,m))V(t,m)\,dm = \lim_{t\downarrow 0} 	\int_\mathbb{R}\varphi (t,x)v(t,x)\rho(t,dx).
\end{align*}
A similar argument can be used for the pair $(\eta,w)$.
\end{proof}

\section{Large-damping limit}\label{sec:large_damping}

In this section we study the large-damping limit of  system \eqref{sistema1} for the damping parameter $\sigma \to \infty$ as stated in Theorem \ref{th:large_damping}. In particular, we aim at making the formal argument introduced in Section \ref{sec:large_dam} rigorous, and showing that solutions to  system \eqref{eulerresc} converge to the ones of the first-order system
\begin{equation}
\label{eq:limitsystem}
\begin{dcases}
\frac{\partial\rho}{\partial t}=\frac{\partial}{\partial x} [\rho K_\rho'\ast\rho+\rho H_\rho'\ast\eta], \\ \frac{\partial \eta}{\partial t}=\frac{\partial}{\partial x}[\eta K_\eta'\ast\eta + \eta H_\eta'\ast \rho].
\end{dcases}
\end{equation}

In what follows we will assume that the potentials $H_\rho$, $H_\eta$ $K_\rho,$ $K_\eta$ are under assumptions \eqref{A} and \eqref{SL}. 

Recalling the definition of $F[X,Y](m)$ and $G[X,Y](m)$ in \eqref{eq:F} and \eqref{eq:G}, we introduce the operator
\[ L\big( (X,Y) \big)(m)\coloneqq \begin{pmatrix} \displaystyle F[X,Y](m) \\ \displaystyle  G[X,Y](m) \end{pmatrix}. \]
By setting $Z_\varepsilon=(X_\varepsilon,Y_\varepsilon),$ $\overline{Z}_\varepsilon=(\overline{X}_\varepsilon,\overline{Y}_\varepsilon),$ $U_\varepsilon=(V_\varepsilon,W_\varepsilon)$ and $\overline{U}_\varepsilon=(\overline{V}_\varepsilon,\overline{W}_\varepsilon),$  system \eqref{sisteps} can be rewritten in the following compact form
\begin{equation}
\label{eq:inclusionepsilon}
\varepsilon\dot{Z}_\varepsilon(t,m)+Z_\varepsilon(t,m)+\partial I_{\mathcal{K}^2}(Z_\varepsilon(t,m))\ni \varepsilon\overline{U}_\varepsilon(m)+\overline{Z}_\varepsilon(m)+\int_0^t L (Z_\varepsilon(r,m))\,dr.
\end{equation}
We are now in the position of proving Theorem \ref{th:large_damping}.

 
 \begin{proof}[Proof of Theorem \ref{th:large_damping}] Let  $(\rho,\eta)$ be a solution to system \eqref{eq:limitsystem} subject to the initial condition $(\overline{\rho},\overline{\eta}),$ and $(\rho_\epsilon,\eta_\epsilon,v_\epsilon,w_\epsilon)$ be a solution to system \eqref{eulerresc} subject to the initial condition $(\overline{\rho}_\varepsilon,  \overline{\eta}_\varepsilon, \overline{v}_\varepsilon, \overline{w}_\varepsilon),$ 
 Define $X_0=\Psi(\rho)$ and $Y_0=\Psi(\eta)$, then $Z_0=(X_0,Y_0)$ is a solution to 
\begin{equation}
\label{eq:inclusionlimit}
 Z_0(t,m)+\partial I_{\mathcal{K}^2}(Z_0(t,m))\ni \overline{Z}_0(m)+\int_0^t L (Z_0(r,m))\,dr,
 \end{equation}
with $\overline{Z}_0=(\overline{X}_0,\overline{Y}_0)=(\Psi(\overline{\rho}),\Psi(\overline{\eta}))$. Similarly, consider $Z_\varepsilon=(X_\varepsilon,Y_\varepsilon)$ that solves \eqref{eq:inclusionepsilon}, with $X_\varepsilon=\Psi(\rho_\varepsilon)$ and $Y_\varepsilon=\Psi(\eta_\varepsilon)$. Adding $\varepsilon\dot{Z}_0(t,m)$ to both sides of \eqref{eq:inclusionlimit} and taking the difference between  \eqref{eq:inclusionepsilon} and \eqref{eq:inclusionlimit}, we get

\begin{align}\label{eq:th4_2}
\begin{aligned} 
&\varepsilon\big(\dot{Z}_\varepsilon(t,m)-\dot{Z}_0(t,m)\big)+Z_\varepsilon (t,m)-Z_0(t,m)+\partial I_{\mathcal{K}^2}(Z_\varepsilon (t,m))-\partial I_{\mathcal{K}^2}(Z_0(t,m))\\ &\ni\varepsilon\overline{U}_\varepsilon(m)+\overline{Z}_\varepsilon (m)-\overline{Z}_0(m)-\varepsilon\dot{Z}_0(t,m)+\int_0^t\big[L (Z_\varepsilon (r,m))-L(Z_0(r,m)) \big]\,dr. 
\end{aligned}
\end{align}

Multiplying both members of \eqref{eq:th4_2} by $Z_\varepsilon -Z_0$, integrating over $m\in [0,1]$ and using the monotonicity of $\partial I_\mathcal{K}$, we obtain

\begin{align}\label{eq:th4_3}
\begin{aligned}
&\frac{\varepsilon}{2}\frac{d}{dt}\int_\Omega\big(Z_\varepsilon (t,m)-Z_0(t,m)\big)^2\,dm+\int_\Omega\big(Z_\varepsilon (t,m)-Z_0(t,m)\big)^2\,dm  \\ &\leq \int_\Omega \big[ \varepsilon \overline{U}_\varepsilon(m)+\overline{Z}_\varepsilon(m)-\overline{Z}_0(m) \big]\big( Z_\varepsilon (t,m)-Z_0(t,m) \big)\,dm  \\&-\varepsilon\int_\Omega\dot{Z}_0(t,m)\big( Z_\varepsilon (t,m)-Z_0(t,m) \big)\,dm  \\& +\int_0^t\int_\Omega\big[ L( Z_\varepsilon(r,m))- L ( Z_0(r,m) )] \big( Z_\varepsilon (t,m)-Z_0(t,m)\big)\,dm\,dr. 
\end{aligned}
\end{align}
Using a weighted Young inequality and the bounds in Lemma \ref{lemma2}, \eqref{eq:th4_3} becomes
\begin{align*}
&\frac{\varepsilon}{2}\frac{d}{dt}\int_\Omega\big(Z_\varepsilon (t,m)-Z_0(t,m)\big)^2\,dm+\int_\Omega\big(Z_\varepsilon (t,m)-Z_0(t,m)\big)^2\,dm  \\ &\leq \frac{1}{2}\int_\Omega \big[ \varepsilon \overline{U}_\varepsilon(m)+\overline{Z}_\varepsilon(m)-\overline{Z}_0(m) \big]^2\,dm + \frac{1}{2}\int_\Omega \big( Z_\varepsilon(t,m)-Z_0(t,m) \big)^2\,dm  \\&+ \frac{\varepsilon}{2}\int_\Omega\dot{Z}_0^2(t,m)\,dm  +\frac{\varepsilon}{2} \int_\Omega \big( Z_\varepsilon (t,m)-Z_0(t,m) \big)^2\,dm  \\& +\frac{1}{2}\int_0^t\int_\Omega\big[ L( Z_\varepsilon(r,m))- L ( Z_0(r,m))]^2\,dm\,dr +\frac{1}{2}\int_0^t\int_\Omega \big( Z_\varepsilon (t,m)-Z_0(t,m)\big)^2\,dm\,dr\,, 
\end{align*}
which implies
\begin{align*}
    &\frac{\varepsilon}{2}\frac{d}{dt}\int_\Omega\big(Z_\varepsilon (t,m)-Z_0(t,m)\big)^2\,dm+\frac{1-\varepsilon}{2}\int_\Omega\big(Z_\varepsilon (t,m)-Z_0(t,m)\big)^2\,dm  \\ &\leq
    \frac{1}{2}\int_\Omega \big[ \varepsilon \overline{U}_\varepsilon(m)+\overline{Z}_\varepsilon(m)-\overline{Z}_0(m) \big]^2\,dm + \frac{\varepsilon}{2}\int_\Omega\dot{Z}_0^2(t,m)\,dm \\
    & \qquad + C \frac{1}{2}\int_0^t\int_\Omega \big( Z_\varepsilon (t,m)-Z_0(t,m)\big)^2\,dm\,dr\,,
\end{align*}
where $C$ is a fixed constant depending on the operator $L$ and coming from Lemma \ref{lemma2}. 
Integrating over $[0,T]$ and denoting
\begin{align*} A(\varepsilon, T) &\coloneqq  (2\varepsilon+4T) \int_\Omega
\big(\overline{Z}_\varepsilon (m)-\overline{Z}_0(m)\big)^2\,dm + 4T\int_\Omega\big[ \varepsilon \overline{U}_\varepsilon (m)\big]^2\,dm \\& + 2\varepsilon\int_0^T\int_\Omega\dot{Z}_0^2(t,m)\,dm\,dt,
\end{align*}
assuming $\varepsilon<1/2$, by using Cauchy-Schwarz inequality we have that
\begin{align*}
& \int_0^T\int_\Omega\big(Z_\varepsilon(t,m)-Z_0(t,m)\big)^2\,dm\,dt  \\
&\leq C \int_0^T\int_0^t\int_\Omega\big( Z_\varepsilon (r,m)-Z_0(r,m)\big)^2\,dm\,dr\,dt+A(\varepsilon,T)\,, 
\end{align*}
by suitably renaming the constant $C$.  By applying Gronwall's lemma we get
\[ \int_0^T\int_\Omega \big(Z_\varepsilon (t,m)-Z_0(t,m)\big)^2\,dm\,dt\leq A(\varepsilon,T) e^{C T}. \]
In order to conclude it is enough to see that $A(\varepsilon)\to 0$ as $\varepsilon\to 0$. We recall assumption $(i)$ reads $\overline{\rho}_\varepsilon\to \overline{\rho}$
and $\overline{\eta}_\varepsilon \to \overline{\eta}$ in $\mathcal{P}_2(\mathbb{R})$, thus $\overline{Z}_\varepsilon\to \overline{Z}_0$ as $\varepsilon\to 0$ in $\Ldod$. Assumption $(ii)$ implies initial velocities under the following conditions
\[
\overline{v}_\varepsilon=o(1/\varepsilon )\,\mbox{ in }\, L^2(d\overline{\rho}_\varepsilon )\,\mbox{ and }\, \overline{w}_\varepsilon=o(1/\varepsilon )\,\mbox{ in }\,L^2(d\overline{\eta}_\varepsilon )\,\mbox{ as }\, \varepsilon\to 0,
\]
thus $\varepsilon\overline{U}_\varepsilon \to 0$ as $\varepsilon\to 0$. Finally, the last term in $A(\varepsilon,T)$ converges to zero since $\dot{Z}_0$ does not depend on $\varepsilon$.
\end{proof}

\section{Newtonian potentials}\label{sec:newtonian}
This section is devoted to study existence of solutions and asymptotic property of system \eqref{sistema1} when self-attractive forces are driven by Newtonian potentials, i.e., $K_\rho (x)=K_\eta (x)=\abs{x}$. In order to proceed, we need to restrict the analysis to the  case of equal cross potentials, namely $H_\rho=H_\eta \eqqcolon   H$. We also consider two uniformly convex external potentials $A_\rho$ and $A_\eta$ acting on the system. More precisely, we assume $A_\rho\,,A_\eta\in C^2(\mathbb{R})$ under assumptions \eqref{H1} and \eqref{H2}. These additional terms don't affect the study of existence of solutions, in the \emph{generalised} sense specified in Definition \ref{def:generalised}, but are only required in the study of asymptotic behaviour in Theorem \ref{th:newtonian}. The system we are dealing with can be expressed in Lagrangian coordinates as follows
\begin{equation}
\label{systempseudoi}
\begin{dcases}
\partial _t X(t,m)=V(t,m), \\ \partial _t Y(t,m)=W(t,m), \\ \partial _t V(t,m)=-\int_\Omega \text{sign}\big(X(t,m)-X(t,m')\big)\,dm' \\ \hspace{1.9cm} -\int_\Omega H'\big(X(t,m)-Y(t,m')\big)\,dm'-\sigma V(t,m) -A_\rho'(X),\\ \partial _t W(t,m)=-\int_\Omega \text{sign}\big(Y(t,m)-Y(t,m')\big) \,dm' \\ \hspace{2cm} -\int_\Omega H'\big(Y(t,m)-X(t,m')\big)\,dm'-\sigma W(t,m) -A_\eta'(Y).
\end{dcases}
\end{equation}

We can associate to the system \eqref{systempseudoi} the following  functional
\begin{equation}
\label{eq:energyfunctional}
\begin{split}
     \mathfrak{F}(X,Y)=&\frac{1}{2}\int_\Omega\int_\Omega \abs{X(m)-X(m')}\,dm'\,dm+\frac{1}{2}\int_\Omega\int_\Omega \abs{Y(m)-Y(m')}\,dm'\,dm \\&+\int_\Omega\int_\Omega H\big( Y(m)-X(m')\big)\,dm'\,dm +\int_\Omega A_\rho(X(m))\,dm+\int_\Omega A_\eta( Y(m))\,dm.
\end{split}
\end{equation}
In particular, we write
\[ \mathfrak{F}(X,Y)\coloneqq S(X)+S(Y)+K(X,Y), \]
where
\begin{align*} 
&S(X)\coloneqq \frac{1}{2}\int_\Omega\int_\Omega\abs{X(m)-X(m')}\,dm'\,dm, \\ &S(Y)\coloneqq \frac{1}{2}\int_\Omega\int_\Omega\abs{Y(m)-Y(m')}\,dm'\,dm, \\
&K(X,Y)\coloneqq \int_\Omega\int_\Omega H\big(Y(m)-X(m')\big)\,dm'\,dm +\int_\Omega A_\rho(X(m))\,dm+\int_\Omega A_\eta( Y(m))\,dm. 
\end{align*}

As shown in \cite{bonaschi,cdefs}, it is easy to prove that the self-interaction contributions in $\mathfrak{F}$ are linear when restricted to $\mathcal{K}$. 

\begin{lemma} If $X\in\mathcal{K}$, then
\[ S(X)=\int_\Omega (2m-1)X(m)\,dm. \]
\end{lemma}
\begin{proof}A direct computation shows that
\begin{align*}
S(X)&=\frac{1}{2}\int_\Omega\int_\Omega \abs{X(m)-X(s)}\,ds\,dm =\int\int_{\left\{X(m)\geq X(s)\right\}} \big(X(m)-X(s)\big)\,dm\,ds.
\end{align*}
Since $X\in\mathcal{K},$ $X$ is non-decreasing, then the set $\left\{X(m)\geq X(s)\right\}$ can be characterised as follows
\[ \{X(m)\geq X(s)\}= \{m\geq s\}\cup \{ m\leq s\leq S(m) \}, \]with\[ S(m)=\sup\{ s\in [0,1] \; :\; X(s)=X(m) \}.\]
Moreover, $X(s)=X(m)$ on $\{ m\leq s\leq S(m) \},$ then
\begin{align*}
S(X)&=\int\int_{m\geq s} \big( X(m)-X(s)\big)\,dm\,ds \\  &=\bigg( \int_\Omega\int_0^m X(m)\,ds\,dm-\int_\Omega\int_s^1 X(s)\,dm\,ds \bigg)\\&=\int_\Omega mX(m)\,dm-\int_\Omega (1-s)X(s)\,ds \\ &= \int_\Omega (2m-1)X(m)\,dm.
\end{align*}
\end{proof}

The first result in this Section consists in proving the existence of a map $t\mapsto (X(t),Y(t))$ that is a generalised Lagrangian solution to \eqref{sisteps} with respect to the choice $\Theta=\mathsf{P}_{\mathcal{H}_X}(F_1)(t,m)$ and $\Xi=\mathsf{P}_{\mathcal{H}_Y}(F_2)(t,m)$, i.e., the system \eqref{systempseudoi} can be written as follows
\begin{equation}
\label{projsys}
\begin{dcases}
\partial _t X(t,m)=\mathsf{P}_{\mathcal{H}_X}(V)(t,m), \\ \partial _t Y(t,m)=\mathsf{P}_{\mathcal{H}_Y}(W)(t,m), \\ \partial _t V(t,m)=-\mathsf{P}_{\mathcal{H}_X}(F_1[X,Y])(m)-\sigma V(t,m), \\ \partial _t W(t,m)=-\mathsf{P}_{\mathcal{H}_Y}(F_2[X,Y])(m)-\sigma W(t,m),
\end{dcases}
\end{equation}
where
\begin{equation}
\label{eq:F1}
    F_1[X,Y](m)= 2m-1+\int_\Omega H'\big(X(m)-Y(m')\big)\,dm'+A_\rho'(X)
\end{equation}
and
\begin{equation} 
\label{eq:F2}
    F_2[X,Y](m)= 2m-1+\int_\Omega H' \big(Y(m)-X(m')\big)\,dm' +A_\eta'(Y) 
\end{equation}
are the force operators and describe the external and interaction forces that act on the system.

The following proposition ensures that a generalised Lagrangian solution exists.

\begin{proposizione}\label{th:ex_newtonian}
Assume  the cross-potential $H$ under assumptions \eqref{A} and \eqref{SL}. Assume the external potentials $A_\rho, A_\eta\in C^2(\mathbb{R})$. Then for every $(\overline{X},\overline{Y},\overline{V},\overline{W})\in \mathcal{K}^2\times \mathcal{H}_{\overline{X}}\times\mathcal{H}_{\overline{Y}}$ there exists a generalised Lagrangian solution to system \eqref{systempseudoi} with initial data $(\overline{X},\overline{Y},\overline{V},\overline{W})$ in the sense of Definition \ref{def:generalised}.
\end{proposizione}

\begin{proof} The proof is based on a discretization argument, inspired by the result in \cite[Theorem 4.5]{gangbo}. Consider the following two partitions of $\Omega$:
\[ 0\eqqcolon   l_0<l_1<\cdots<l_N\coloneqq 1, \quad \mbox{ and }\quad 0\eqqcolon   z_0<z_1<\cdots<z_M\coloneqq 1,  \]
with
\[ l_i\coloneqq \sum_{j=1}^i m_j, \quad\mbox{ and }\quad  z_j\coloneqq \sum_{i=1}^j n_i, \]
for $i=1,\ldots,N-1$ and $j=1,\ldots,M-1$, and introduce the  piecewise constant functions
\begin{gather}
\label{eq:particlexv}
X(t,\cdot)=\sum_{i=1}^N x_i(t)\mathbbm{1}_{L_i}\,, \qquad V(t,\cdot)=\sum_{i=1}^N v_i(t)\mathbbm{1}_{L_i}, \\ \label{eq:particleyw}
 Y(t,\cdot)=\sum_{j=1}^M y_j(t)\mathbbm{1}_{Z_j}\,, \qquad W(t,\cdot)=\sum_{j=1}^M w_j(t)\mathbbm{1}_{Z_j},
\end{gather}
defined on the intervals $L_i\coloneqq [l_{i-1},l_i)$ and $Z_j\coloneqq [z_{j-1},z_j),$ for $i=1,\ldots,N-1$ and $j=1,\ldots,M-1$. Consider the finite dimensional Hilbert set
\[ \mathcal{H}_m\times\mathcal{H}_n\coloneqq \bigg\{ (X,Y)=\bigg(\sum_{i=1}^N x_i\mathbbm{1}_{L_i},\sum_{j=1}^M y_j\mathbbm{1}_{Z_j}\bigg) \; :\; (x,y)\in\mathbb{R}^N\times\mathbb{R}^M\bigg\}\subset L^2(\Omega)\times L^2(\Omega) \]
and its closed convex cone
\[ \mathcal{K}_m\times\mathcal{K}_n\coloneqq \bigg\{ (X,Y)=\bigg(\sum_{i=1}^N x_i\mathbbm{1}_{L_i},\sum_{j=1}^M y_j\mathbbm{1}_{Z_j}\bigg) \; :\; (x,y)\in\mathbb{K}^N\times\mathbb{K}^M\bigg\}\subset \mathcal{K}\times \mathcal{K}. \]
Note that the projected forces
\[ F_m[X,Y]\coloneqq \mathsf{P}_{\mathcal{H}_m}(F_1[X,Y]) \qquad \text{and} \qquad F_n[X,Y]\coloneqq \mathsf{P}_{\mathcal{H}_n}(F_2[X,Y]) \]
are well defined and Lipschitz continuous  according to the definitions in \eqref{eq:F1}-\eqref{eq:F2} and assumptions \eqref{A} and \eqref{SL}.

Now, assume that the initial condition $(\overline{X},\overline{Y},\overline{V},\overline{W})\in \mathcal{K}_m\times\mathcal{K}_n\times\mathcal{H}_{\overline{X}}\times\mathcal{H}_{\overline{Y}}$ doesn't hit the boundary of $\mathcal{K}_m\times\mathcal{K}_n$. Consider the time interval $[0,t_1)$ with
\[
 t_1=\inf\left\{t > 0 \,:\, X(t)\in\partial \mathcal{K}_m \ \text{or} \ Y(t)\in \partial \mathcal{K}_n \right\},
\]
then, we obtain  \eqref{eq:particlexv}-\eqref{eq:particleyw} by solving
\begin{equation}
\begin{aligned}\label{eq:existece1}
&\dot{X}(t)=V(t), \quad \dot{V}(t)=\mathsf{P}_{\mathcal{H}_{m}}  \bigg( \frac{1}{\varepsilon} \big( F_1[X(t),Y(t)] - V(t) \big) \bigg) ,   \\ 
& \dot{Y}(t)=W(t),\quad \dot{W}(t)=\mathsf{P}_{\mathcal{H}_{n}} \bigg( \frac{1}{\varepsilon} \big( F_2[X(t),Y(t)] - W(t) \big) \bigg) .
\end{aligned}
\end{equation}
We have that $\mathcal{H}_{m}=\mathcal{H}_{X(t)}$ and $\mathcal{H}_{n}=\mathcal{H}_{Y(t)}$ in $[0,t_1)$, thus the projection onto the set $\mathcal{H}_{m}$ yields piecewise constant functions defined on the same intervals as   $(X,V)$, and similarly the projection onto $\mathcal{H}_{n}$.
Taking $t_1$ as the new initial time, we can consider a new initial condition $(\overline{X}',\overline{Y}',\overline{V}',\overline{W}')\in \mathcal{K}_{m'}\times\mathcal{K}_{n'}\times\mathcal{H}_{\overline{X}'}\times\mathcal{H}_{\overline{Y}'}$ of dimensions $N'< N$ and $M'< M$ and, proceeding in the same fashion, we can define $t_2>t_1$ and consider the evolution in the time interval $[t_1,t_2)$. Iterating the procedure, we obtain a sequence of collision times $0\eqqcolon  t_0<t_1<\cdots <t_K\coloneqq \infty$ and the quadruple $(X,Y,V,W)$ such that
\begin{equation}
\label{existence2}
\begin{aligned}
&\dot{X}(t)=V(t), \quad \dot{V}(t)=
\mathsf{P}_{\mathcal{H}_{X(t)}}  \bigg(  \frac{1}{\varepsilon} \big( F_1[X(t),Y(t)] - V(t) \big)\bigg),\\& \dot{Y}(t)=W(t),  \quad \dot{W}(t)=
\mathsf{P}_{\mathcal{H}_{Y(t)}} \bigg(  \frac{1}{\varepsilon} \big( F_2[X(t),Y(t)] - W(t) \big) \bigg),
\end{aligned}
\end{equation}
for all $t\in [t_{k-1},t_k)$, $k=1,\ldots ,K$
with
\begin{equation}
\label{eq:spaces}
\mathcal{H}_{X(t)}=\mathcal{H}_{X(t_{k-1})}, \quad \mathcal{H}_{Y(t)}=\mathcal{H}_{Y(t_{k-1})}. 
\end{equation}
When an inelastic collision occurs, we have that
\begin{equation}
    \begin{aligned}
        \label{collision}
&X(t_k+)=X(t_k-), \qquad V(t_k+)=\mathsf{P}_{\mathcal{H}_{X(t_k)}} (V(t_k-)), \\& Y(t_k+)=Y(t_k-), \qquad W(t_k+)=\mathsf{P}_{\mathcal{H}_{Y(t_k)}} (W(t_k-)).
    \end{aligned}
\end{equation}


In order to prove inclusion \eqref{def:semigroupx},  it is not restrictive to assume $t_1=0$. We proceed by induction on the collision times. In the first time interval $[0,t_1)$, inclusion \eqref{def:semigroupx} holds by considering the empty set for the subdifferential $\partial I_{\mathcal{K}} (X(t))$. 
Now, suppose that \eqref{def:semigroupx} is satisfied in $[t_{k-1},t_k)$. Hence, by induction assumption,
\begin{equation}
\label{eq:proof1}
 \varepsilon V(t_k-)+ X(t_k-)+\xi =\varepsilon \overline{V}+\overline{X}+\int_0^{t_k} \mathsf{P}_{\mathcal{H}_{X(s)}}(F_1[X(s),Y(s)])\,ds
\end{equation}
with $\xi\in\partial I_\mathcal{K} (X(t_k)).$ By \eqref{existence2},
\begin{align} 
\label{eq:proof2}
\begin{aligned}
\varepsilon \dot{X}(t)+X(t) =& X(t_k+)+ \varepsilon V(t_k+)+\int_{t_k}^t \mathsf{P}_{\mathcal{H}_{X(s)}}(F_1[X(s),Y(s)])\,ds
\\=& X(t_k+)+ \varepsilon \big( V(t_k+)-V(t_k-) \big) +\varepsilon V(t_k-) \\ 
&+\int_{t_k}^t \mathsf{P}_{\mathcal{H}_{X(s)}}(F_1[X(s),Y(s)])\,ds
\end{aligned}
\end{align}
for any $t\in [t_k,t_{k+1}).$ Combining  equations \eqref{eq:proof1} and \eqref{eq:proof2} we get 
\[
\varepsilon \dot{X}(t)+X(t)+\varepsilon\big( V(t_k-)-V(t_k+) \big) + \xi =\varepsilon\overline{V}+\overline{X}+\int_0^t \mathsf{P}_{\mathcal{H}_{X(s)}}(F_1[X(s),Y(s)])\,ds.
\]
Invoking again \eqref{existence2}, we have 
\[ V(t_k-)=\lim_{h\to 0^+} \frac{X(t_k)-X(t_k-h)}{h}, \]
hence using \eqref{collision}, we derive
\begin{align*}
V(t_k-)-V(t_k+)=& V(t_k-)-\mathsf{P}_{\mathcal{H}_{X(t_k)}}(V(t_k-)) \\ =& \lim_{h\to 0^+}\frac{X(t_k)-X(t_k-h)-\mathsf{P}_{\mathcal{H}_{X(t_k)}}\big( X(t_k)-X(t_k-h)\big)}{h} \\ =& \lim_{h\to 0^+}\frac{\mathsf{P}_{\mathcal{H}_{X(t_k)}} (X(t_k-h))-X(t_k-h)}{h}.
\end{align*}
Applying \cite[Lemma 2.6]{gangbo}, we find that $V(t_k-)-V(t_k+)\in\partial I_\mathcal{K}(X(t_k))$, and using the monotonicity property of the sub-differential, one obtains that
\[ \xi + V(t_k-)-V(t_k+)\in\partial I_\mathcal{K}(X(t))\]
for all $t\in [t_k, t_{k+1}).$ Therefore inclusion \eqref{def:semigroupx} is satisfied.
Now, let us prove that \eqref{def:projx} holds.
Consider system \eqref{sistpq} with $P$ replaced by
\[ 
P_1 (t,m) = \varepsilon\overline{V} (m) + \overline{X} (m) + \int_0^t F_1 [ X(\cdot,r),Y(\cdot ,r)] (m)\,dr.
\]
Thus, we have that for any $t\geq s\geq 0$, 
\[ \frac{1}{\varepsilon}[ P_1(s)-X(s)]-V(s)\in\partial I_\mathcal{K}(X(s))\subset \partial I_\mathcal{K}(X(t)), \]
where we used the monotonicity of the sub-differential. Integrating on $s\in[0,t]$ we obtain
\[ \int_0^t \frac{1}{\varepsilon}[P_1(s) -X(s)]\,ds+\overline{X}-X(t)\in\partial I_\mathcal{K}(X(t)) \]
for a.e. $t\geq 0.$ 
Since the following property holds (cfr. \cite{gangbo})
\[ Y=\mathsf{P}_{\mathcal{K}}(X)\iff X-Y\in\partial I_\mathcal{K}(Y), \]
we derive
\[ X(t)=\mathsf{P}_\mathcal{K}\bigg( \overline{X} -\frac{1}{\varepsilon}\int_0^t X(s)\,ds +\frac{1}{\varepsilon}t\big( \varepsilon \overline{V}+\overline{X} \big)+\frac{1}{\varepsilon} \int_0^t (t-s)F_1[X(s),Y(s)]\,ds \bigg). \]
A similar proof holds for the equations \eqref{def:semigroupy} and \eqref{def:projy}.
Finally, since the construction above starts form descrete initial data in the form of the piecewise constant functions as in \eqref{eq:particlexv}-\eqref{eq:particleyw}, and since these functions are dense in $L^2(\Omega)$, we can approximate any given initial data and then combine the procedure into the proof with the stability Theorem 4.4 in \cite{gangbo}.
\end{proof}

Now, we provide an estimate on the total energy of the system \eqref{projsys}, used in the proof of Theorem \ref{th:newtonian} below.

\begin{lemma}
Let $(X,Y,V,W)\in\mathcal{K}^2\times L^2(0,1)^2$ be the solution to the system \eqref{projsys} with initial data $(\overline{X},\overline{Y},\overline{V},\overline{W}).$ Then, the following uniform estimate holds:
\begin{equation}
\label{eq:toten}
\sup_{t\geq 0} \bigg( \mathfrak{F}(X,Y)+\frac{1}{2}\norma{V}^2_{L^2(\Omega)}+\frac{1}{2}\norma{W}^2_{L^2(\Omega)}\bigg)\leq\mathfrak{F}(\overline{X},\overline{Y})+\frac{1}{2}\norma{\overline{V}}^2_{L^2(\Omega)}+\frac{1}{2}\norma{\overline{W}}^2_{L^2(\Omega)}.
\end{equation}
\end{lemma}
\begin{proof}The proof is based on an estimate of the following total energy
\[
  \mathfrak{E}(X,Y,V,W)=\frac{1}{2}\int_\Omega \abs{V}^2\,dm+\frac{1}{2}\int_\Omega \abs{W}^2\,dm+\mathfrak{F}(X,Y).
\]
Considering $(X,Y,V,W)$ generalised solution to \eqref{projsys}, we have 
\begin{align}\label{eq:lemma4_1}
    \begin{aligned}
    \frac{d^+}{dt}\mathfrak{E}(X,Y,V,W)  
     =& -\sigma \int_\Omega \big(\abs{V}^2+\abs{W}^2\big)\,dm\\&-\int_\Omega V \mathsf{P}_{\mathcal{H}_X}(F_1)\,dm-\int_\Omega W\mathsf{P}_{\mathcal{H}_Y}(F_2)\,dm  \\& + \int_\Omega \mathsf{P}_{\mathcal{H}_X}(V) \big[ 2m-1+\int_\Omega H'\big(X(m)-Y(m')\big)\,dm'+A_\rho'(X) \big]\,dm  \\& + \int_\Omega \mathsf{P}_{\mathcal{H}_Y}(W) \big[ 2m-1+\int_\Omega H'\big(Y(m)-X(m')\big)\,dm'+A_\eta'(Y) \big]\,dm. 
    \end{aligned}
\end{align}
Thanks to the definitions of $F_1[X,Y]$ and $F_2[X,Y]$ in \eqref{eq:F1}-\eqref{eq:F2}, we obtain that  
    \begin{align}\label{eq:lemma4_2}
    \begin{aligned}
        \frac{d^+}{dt}\mathfrak{E}(X,Y,V,W)=& -\sigma \int_\Omega \big(\abs{V}^2+\abs{W}^2\big)\,dm-\int_	\Omega V \mathsf{P}_{\mathcal{H}_X}(F_1)\,dm\,dm  \\& -\int_\Omega W\mathsf{P}_{\mathcal{H}_Y}(F_2)+\int_\Omega \mathsf{P}_{\mathcal{H}_X} (V) F_1 \,dm +\int_\Omega \mathsf{P}_{\mathcal{H}_Y} (W) F_2 \,dm .
    \end{aligned}
\end{align}
By definition of the projection operator in \eqref{projoper},
\[ \int_\Omega\mathsf{P}_{\mathcal{H}_X}(V)\big(\mathsf{P}_{\mathcal{H}_X}(F_1[X,Y])-F_1[X,Y]\big)\,dm=0=\int_\Omega \mathsf{P}_{\mathcal{H}_X}(F_1[X,Y])\big(\mathsf{P}_{\mathcal{H}_X}(V)-V\big)\,dm, \]
then
\[  \int_\Omega \big( F_1[X,Y]\mathsf{P}_{\mathcal{H}_X}(V)-V\mathsf{P}_{\mathcal{H}_X}(F_1[X,Y]) \big)\,dm=0,\]
and similarly 
\[ \int_\Omega \big( F_2[X,Y]\mathsf{P}_{\mathcal{H}_Y}(W)-W\mathsf{P}_{\mathcal{H}_Y}( F_2[X,Y]) \big) \,dm=0,
\]
therefore \eqref{eq:lemma4_2}  reduces to
\begin{equation}
\label{eq:bounded}
    \frac{d^+}{dt}\mathfrak{E}(X,Y,V,W)=-\sigma\int_\Omega\abs{V}^2\,dm-\sigma\int_\Omega\abs{W}^2\,dm\leq 0,
\end{equation}
from which we can easily deduce the uniform estimate \eqref{eq:toten}.
\end{proof}

We can now provide the proof of  Theorem \ref{th:newtonian}.

\begin{proof} [Proof of Theorem \ref{th:newtonian}]
Integrating in time the equation \eqref{eq:bounded}, we find that for all $T> 0$
\begin{align*}
   &\mathfrak{E}(X,Y,V,W) \big|_{t=T}+\sigma \int_0^T \int_\Omega \big(\abs{V}^2+\abs{W}^2\big)\,dm\,dt  =\mathfrak{E}(X,Y,V,W)\big|_{t=0}.
\end{align*}
Thanks to the non-negativity of the cross-potential $H$,  assumption \eqref{H1} and the fact that
\begin{equation}
\label{stimaperparti}
\int_\Omega (2m-1)(X+Y)\,dm=-\int_\Omega(m^2-m)(\partial_mX+\partial_mY) \, dm\geq 0,
\end{equation}
which holds since $m^2-m\leq 0$ for $m\in (0,1)$ and $\partial_m X+\partial_mY\geq 0$ for $X,Y\in\mathcal{K},$
we obtain that
\[
    \sigma \int_0^T\int_\Omega\big( \abs{V}^2+\abs{W}^2\big)\,dm\,dt\leq -\lambda\int_\Omega \abs{X}^2\big|_{t=T}\,dm -\mu\int_\Omega\abs{Y}^2\big|_{t=T}\,dm +C_1,
\]
where $C_1$ is a constant depending on initial data. Thus
\begin{equation}
\label{boundedvel}
\int_0^\infty \int_\Omega \big(\abs{V}^2+\abs{W}^2\big)\,dm\,dt < +\infty.
\end{equation} 
Computing the temporal derivative  of the $L^2$-distance between $(X,Y)$ and $(X_s,Y_s),$ we derive
\begin{align}\label{eq:th5_0}
\begin{aligned}
    \frac{1}{2} \frac{d}{dt}\int_\Omega \big( \abs{X}^2+\abs{Y}^2\big)\,dm=&\int_\Omega X \mathsf{P}_{\mathcal{H}_X}(V)\,dm+\int_\Omega Y \mathsf{P}_{\mathcal{H}_Y}(W)\,dm  \\ =&\int_\Omega X\big( \mathsf{P}_{\mathcal{H}_X}(V)-V\big)\,dm+\int_\Omega Y\big( \mathsf{P}_{\mathcal{H}_Y}(W)-W\big)\,dm \\ &+\int_\Omega (XV+YW)\,dm  \\ =&\int_\Omega ( XV+YW )\,dm. 
    \end{aligned}
\end{align}
In order to control the last term in the chain of equality above we compute
\begin{align}\label{eq:th5_1}
\begin{aligned}
    \frac{d}{dt}\int_\Omega (XV+YW)\,dm = & \int_\Omega X \big[ -\sigma V-\mathsf{P}_{\mathcal{H}_X}(F_1)\big]\,dm+\int_\Omega V\mathsf{P}_{\mathcal{H}_X}(V)\,dm  \\& +\int_\Omega Y\big[ -\sigma W -\mathsf{P}_{\mathcal{H}_Y}(F_2)\big]\,dm + \int_\Omega W\mathsf{P}_{\mathcal{H}_Y}(W)\,dm.
\end{aligned}
\end{align}
Using the definitions of $F_1$ and $F_2$ in \eqref{eq:F1} and \eqref{eq:F2} and the property for the projection operator we have
\begin{align}\label{eq:th5_2}
\begin{aligned}
     \frac{d}{dt}\int_\Omega (XV+YW)\,dm
   =&\int_\Omega \big( -\sigma XV-\sigma YW +\abs{V}^2+\abs{W}^2\big)\,dm -\int_\Omega (2m-1)(X+Y)\,dm  \\& -\int_\Omega\int_\Omega X(m) H'\big(X(m)-Y(m')\big)\,dm'\,dm \\& -\int_\Omega\int_\Omega Y(m) H'\big(Y(m)-X(m')\big)\,dm'\,dm  \\& -\int_\Omega XA_\rho'(X)\,dm -\int_\Omega Y A_\eta'(Y)\,dm. 
\end{aligned}
\end{align}
Using assumption \eqref{XX} we can bound the terms involving the cross-interaction potential $H$ as follows
\begin{align*}
    & -\int_\Omega \int_\Omega X(m)H'\big( X(m)-Y(m')\big)\,dm'\,dm-\int_\Omega\int_\Omega Y(m)H'\big(Y(m)-X(m')\big)\,dm'\,dm \\
    & =-\int_\Omega\int_\Omega H'\big( X(m)-Y(m')\big) \big(X(m)-Y(m')\big)\,dm'\,dm\leq 0,
\end{align*}
thus, using assumption \eqref{H2} and \eqref{stimaperparti}, \eqref{eq:th5_2} can be bounded from above by 
\begin{equation}
\label{eq:th5_3}
\frac{d}{dt}\int_\Omega (XV+YW)\,dm\leq\int_\Omega \big( -\sigma XV-\sigma YW +\abs{V}^2+\abs{W}^2-\alpha\abs{X}^2-\beta\abs{Y}^2 \big)\,dm.
\end{equation}
Note that for any $A > 0$ 
 we  have $-XV\leq X^2A^2+\frac{V^2}{4A^2}$. Then,  applying this inequality to $-\sigma XV$ and $-\sigma YW$, we obtain the following inequality holding for any $A_1, A_2>0$:
\begin{equation}
\label{eq:controllostima}
\begin{aligned}
     \int_\Omega & \big( -\sigma XV-\sigma YW +\abs{V}^2+\abs{W}^2-\alpha\abs{X}^2-\beta\abs{Y}^2 \big)\,dm \\  \leq &  -\int_\Omega \abs{X}^2 \big( \alpha -\sigma A_1^2\big)\,dm -\int_\Omega \abs{Y}^2\big( \beta-\sigma A_2^2\big)\,dm \\& +\int_\Omega \abs{V}^2 \big( 1+\frac{\sigma}{4A_1^2}\big)\,dm +\int_\Omega \abs{W}^2\big( 1+\frac{\sigma}{4A_2^2}\big)\,dm.
\end{aligned}
\end{equation}
By taking sufficiently small $A_1$ and $A_2$, we have that \eqref{eq:th5_3} is bounded from above by
\begin{equation} 
\label{eq:th5_4}
\frac{d}{dt}\int_\Omega (XV+YW)\,dm\leq-\overline{C}_1 \int_\Omega \big( \abs{X}^2+\abs{Y}^2\big)\,dm+\overline{C}_2\int_\Omega \big( \abs{V}^2+\abs{W}^2 \big)\,dm
\end{equation}
for some constants $\overline{C}_1,\overline{C}_2>0.$
Putting together  estimates \eqref{eq:th5_0} and \eqref{eq:th5_4}, we have that
\begin{align}\label{eq:th5_5}
\begin{aligned}
    & \frac{d}{dt}\int_\Omega \big( \abs{X}^2+\abs{Y}^2+XV+YW\big)\,dm  \\& \leq 2\int_\Omega \big( XV+YW \big) \,dm-\overline{C}_1\int_\Omega \big( \abs{X}^2+\abs{Y}^2\big)\,dm+\overline{C}_2\int_\Omega \big( \abs{V}^2+\abs{W}^2\big)\,dm. 
    \end{aligned}
\end{align}
Integrating in time inequality \eqref{eq:th5_5}, for all $T> 0$ we obtain
\begin{align*}
    \int_\Omega & \big( \abs{X}^2+\abs{Y}^2+XV+YW\big)\,dm\big|_{t=T}-\int_\Omega  \big( \abs{X}^2+\abs{Y}^2+XV+YW\big)\,dm\big|_{t=0} \\  \leq & 2\int_0^T\int_\Omega \big( XV+YW\big)\,dm\,dt-\overline{C}_1\int_0^T\int_\Omega \big( \abs{X}^2+\abs{Y}^2\big)\,dm\,dt \\& +\overline{C}_2\int_0^T\int_\Omega \big( \abs{V}^2+\abs{W}^2\big)\,dm\,dt,
\end{align*}
thus
\begin{align*}
\overline{C}_1\int_0^T\int_\Omega \big( \abs{X}^2+\abs{Y}^2 \big)\,dm\,dt  \leq & \overline{C}_2\int_0^T\int_\Omega \big( \abs{V}^2+\abs{W}^2 \big)\,dm\,dt \\ & + 2\int_0^T\int_\Omega \big( XV+YW\big)\,dm\,dt \\ & -\int_\Omega \big( \abs{X}^2+\abs{Y}^2+XV+YW \big)\,dm \big|_{t=T}+ C_2,
 \end{align*}
where $C_2$ is a constant which depends on initial data.
Proceeding as in \eqref{eq:controllostima} and using the bound in \eqref{boundedvel}, we have that 
\begin{equation} 
\label{boundedpos}
\int_0^\infty \int_\Omega \big(\abs{X}^2+\abs{Y}^2\big)\,dm < +\infty.
\end{equation}
Combining  estimates \eqref{boundedvel} and \eqref{boundedpos} we get
\[ \int_0^\infty\int_\Omega \big( \abs{X}^2+\abs{Y}^2+\abs{V}^2+\abs{W}^2 \big)\,dm\,dt < +\infty,   \]
hence, there exists a sub-sequence $\{ t_k\}_k$ such that
\begin{equation}
\label{eq:subsequence}
    \int_\Omega \big( \abs{X(t_k)}^2+\abs{Y(t_k)}^2+\abs{V(t_k)}^2+\abs{W(t_k)}^2 \big)\,dm \to 0 \qquad \text{as $t_k\to +\infty.$}
\end{equation}
Since the operator $\mathfrak{F}$ defined in \eqref{eq:energyfunctional} is a monotone operator, then
\[ \mathfrak{F}(X,Y)+\frac{1}{2}\int_\Omega \abs{V}^2\,dm+\frac{1}{2}\int_\Omega\abs{W}^2\,dm\to\ell>0 \qquad\text{as $t\to+\infty$} \]
 and $\ell$ is unique. Moreover, Lemma \ref{lemma2} ensures that the operator $\mathfrak{F}$ is continuous, thus
 \[ \frac{1}{2}\int_\Omega \abs{V}^2\,dm+\frac{1}{2}\int_\Omega \abs{W}^2\,dm+\mathfrak{F}(X,Y) \big|_{t=t_k}\to \ell \qquad \text{as $t_k\to +\infty.$}
\]
Using the coercivity of the external potentials $A_\rho$ and $A_\eta$ and  \eqref{eq:subsequence}, we have that $\ell$ is necessarily zero, hence the statement holds.
\end{proof}

\section{Simulations}\label{sec:numerics}
This last section is devoted to provide some numerical  example on the behaviour of solutions to system \eqref{sistema1}. Numerical simulations will be performed  by using the discrete particle counterpart of \eqref{sistema1}, 
namely solving numerically \eqref{neweq}. We recall that the system of ODEs we are dealing with is the following 
\begin{equation}
\begin{dcases}
\label{syst:numerics}
        \dot{x}_i(t)=v_i(t), \\ \dot{y}_j(t)=w_j(t), \\
        \dot{v}_i(t)=-\sigma v_i(t)-\sum_{k\neq i}m_k\nabla K_\rho \big(x_i(t)-x_k(t)\big)- \sum_kn_k\nabla H_\rho \big(x_i(t)-y_k(t)\big), \\
        \dot{w}_j(t)=-\sigma w_j(t)-\sum_{k\neq j}n_k\nabla K_\eta \big(y_j(t)-y_k(t)\big)- \sum_km_k\nabla H_\eta \big(y_j(t)-x_k(t)\big),
\end{dcases}
\end{equation}
where $x_i$ and $y_j$ denotes the particles positions of first and second species respectively, $v_i$ and $w_j$ their velocities and $m_i$ and $n_j$ their masses, for $i=1,\ldots,N$ and $j=1,\ldots,M$. By a normalisation in the masses the total number of particles for each species, $N$ and $M$ respectively, will be modified in each of the examples below in order to highlights possible different changes in the solutions. 
\begin{figure}[ht]
\centering
\includegraphics[width=.49\textwidth]{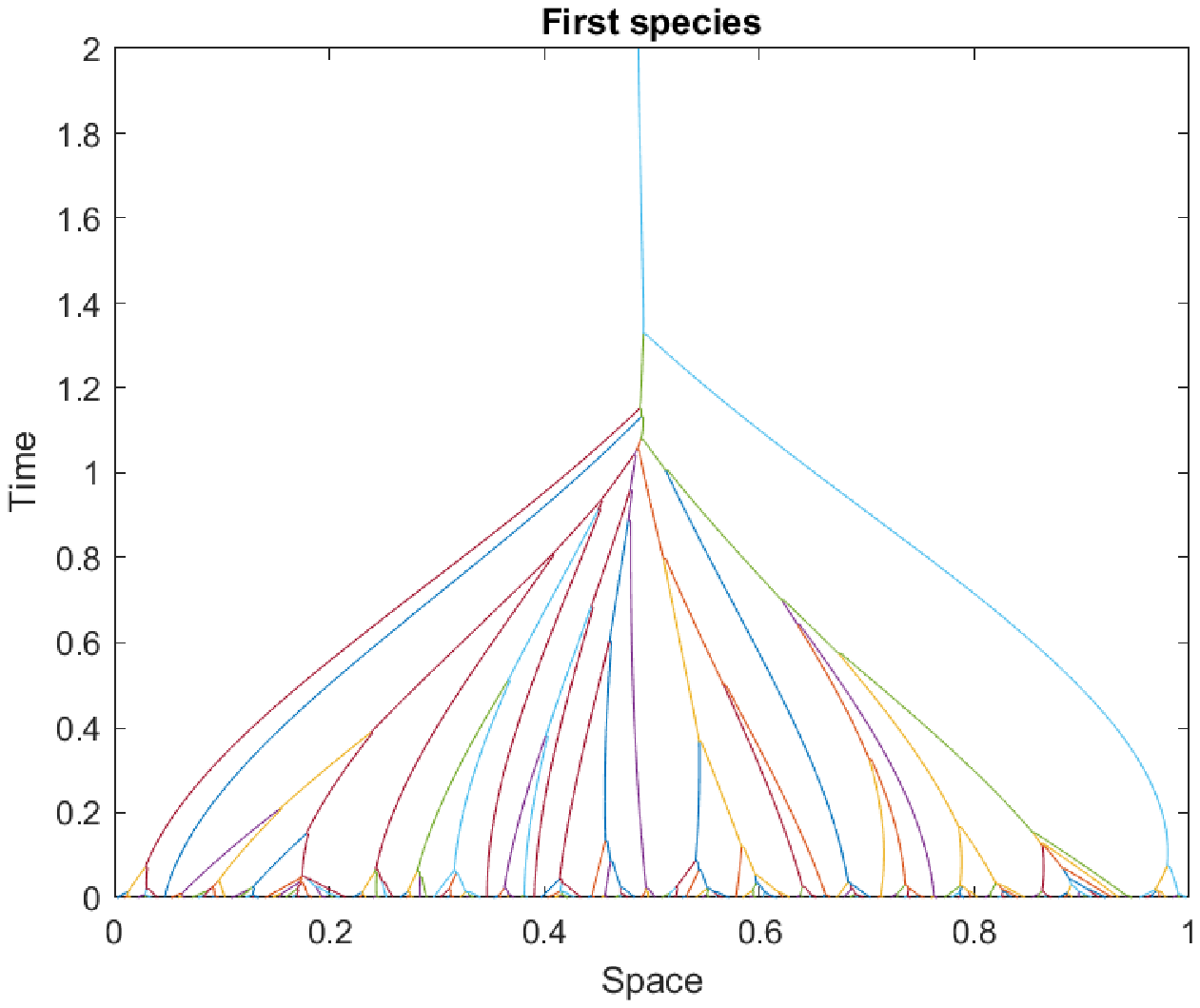}
\includegraphics[width=.49\textwidth]{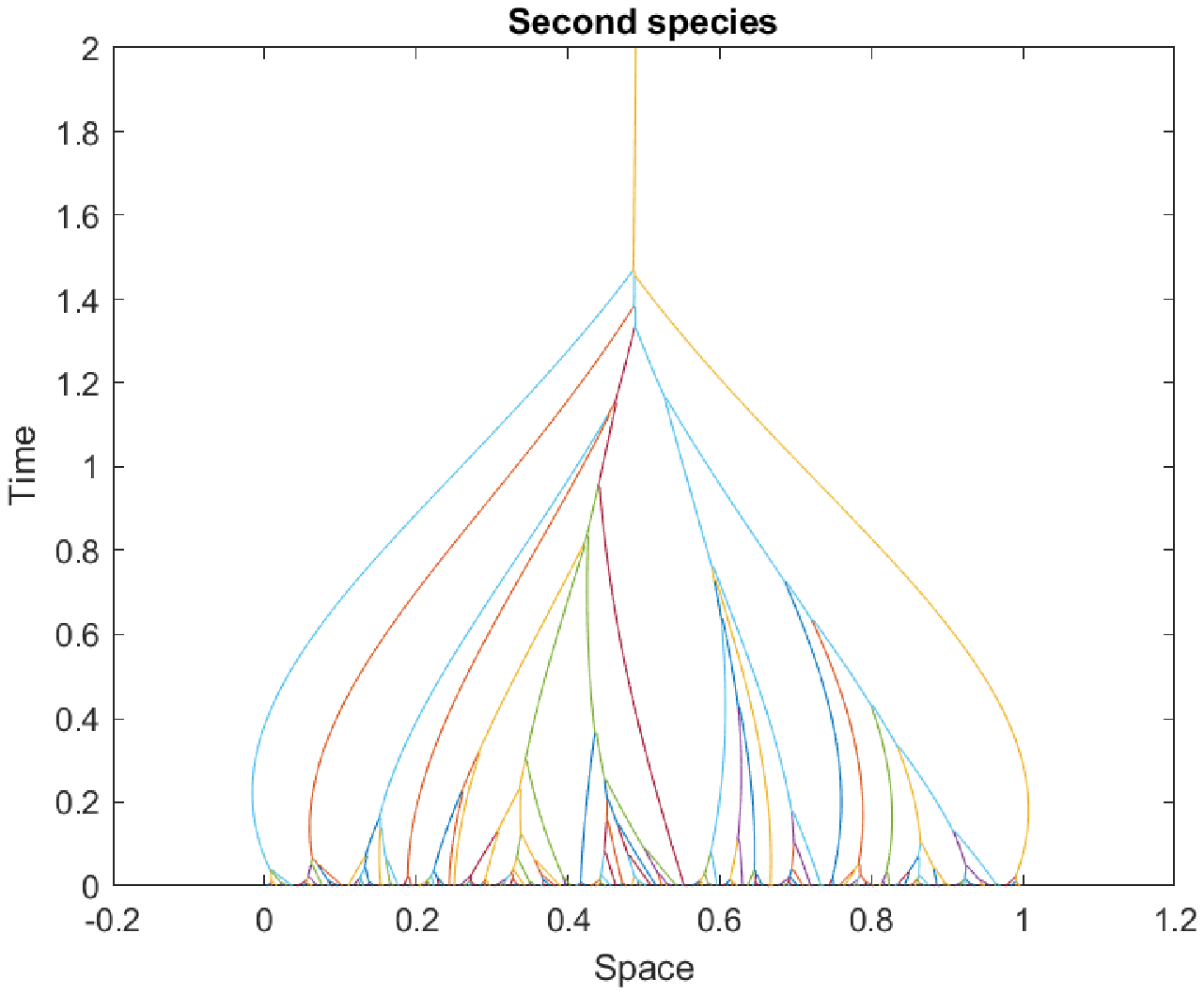}
\caption{In this first example, we fix $N=160$, and $M=150$. All the potentials are attractive. In particular they we set $K_\rho(x)=-e^{-\abs{x}^3},$ $K_\eta(x)=-e^{-\abs{x}^4},$ $H_\rho(x)=H_\eta(x)=-e^{-\abs{x}^2}.$}
\label{figure1}
\end{figure}

\begin{figure}[ht]
\centering
\includegraphics[width=.49\textwidth]{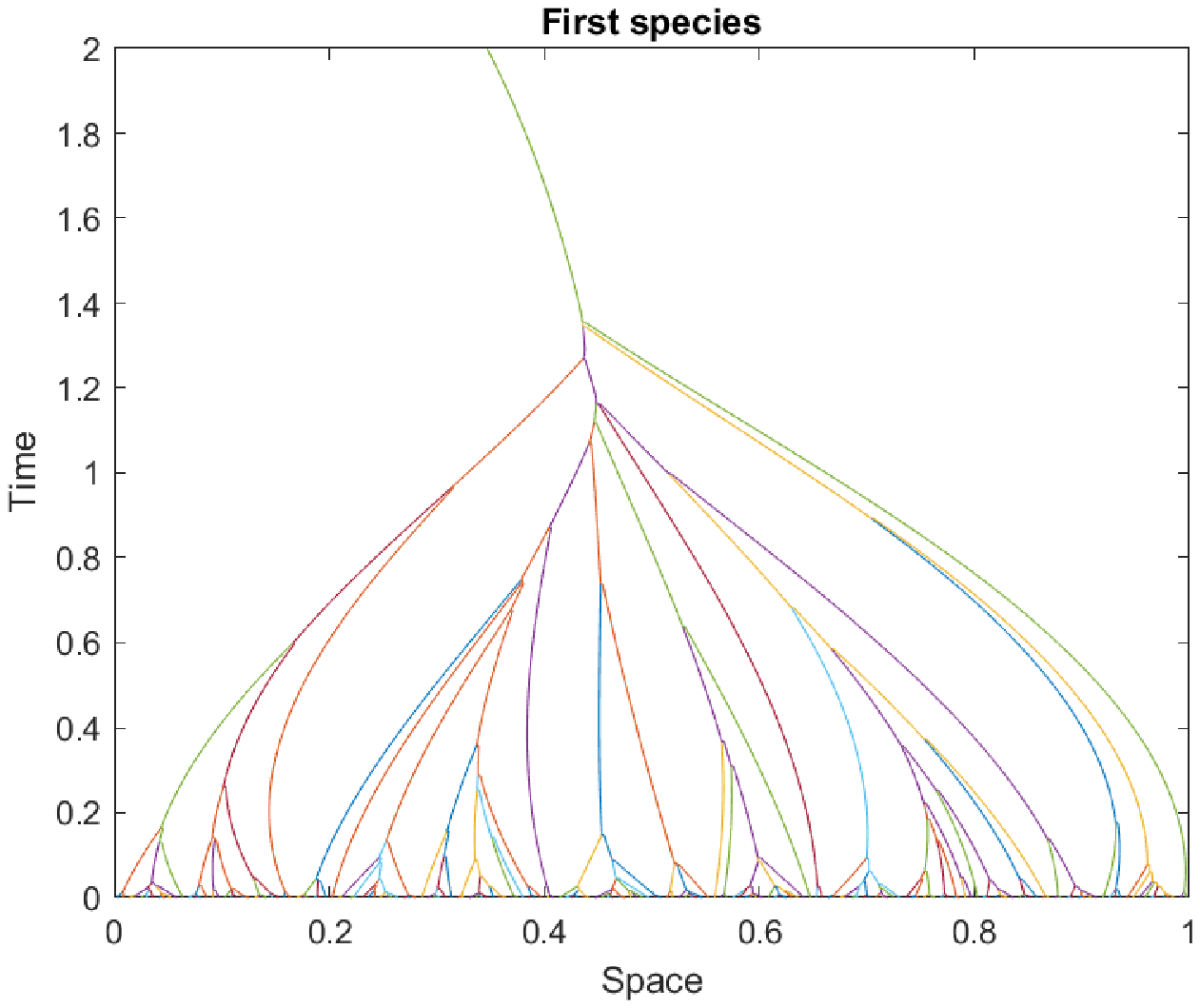}
\includegraphics[width=.49\textwidth]{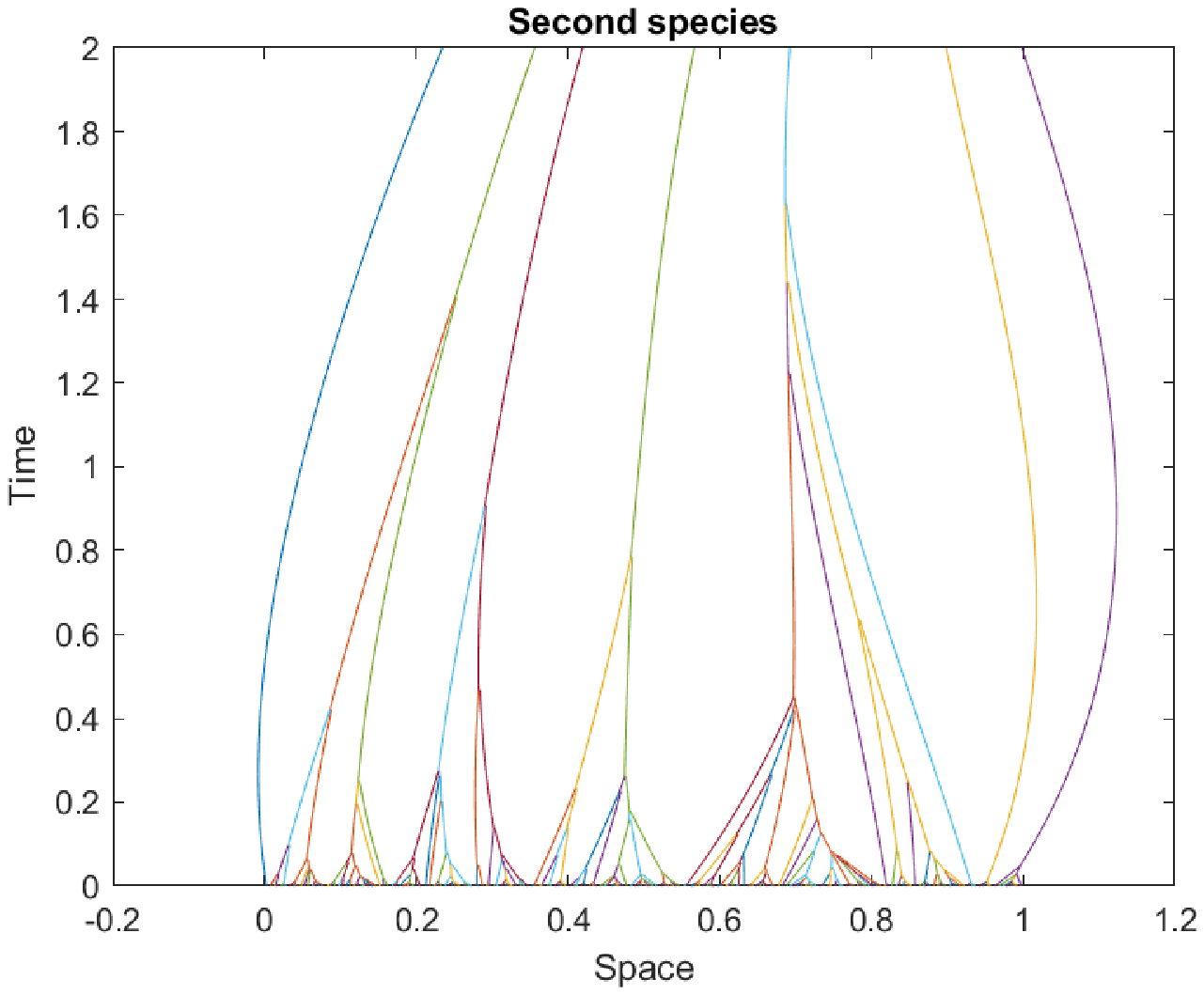}
\caption{Evolution under the action of attractive self potentials given by $K_\rho(x)=-3e^{-\abs{x}^2},$ and $K_\eta(x)=-2e^{-2\abs{x}^3},$ and repulsive cross-potentials $H_\rho(x)=-\abs{x}^2,$ $H_\eta(x)=e^{-\abs{x}^2}.$ In this example, $N=180,$ $M=200$.}
\label{figure2}
\end{figure}

\begin{figure}[ht]
\centering
\includegraphics[width=.49\textwidth]{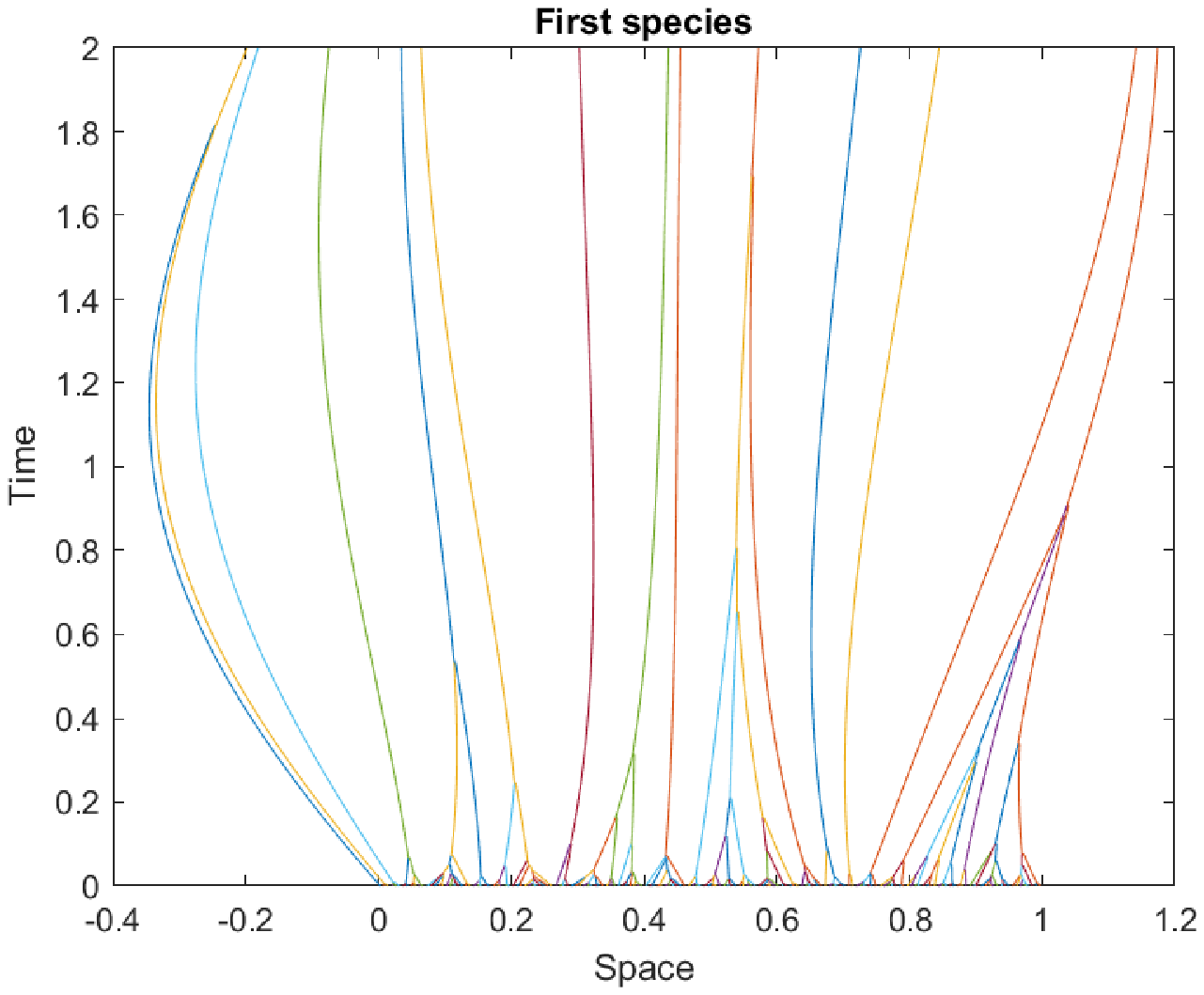}
\includegraphics[width=.49\textwidth]{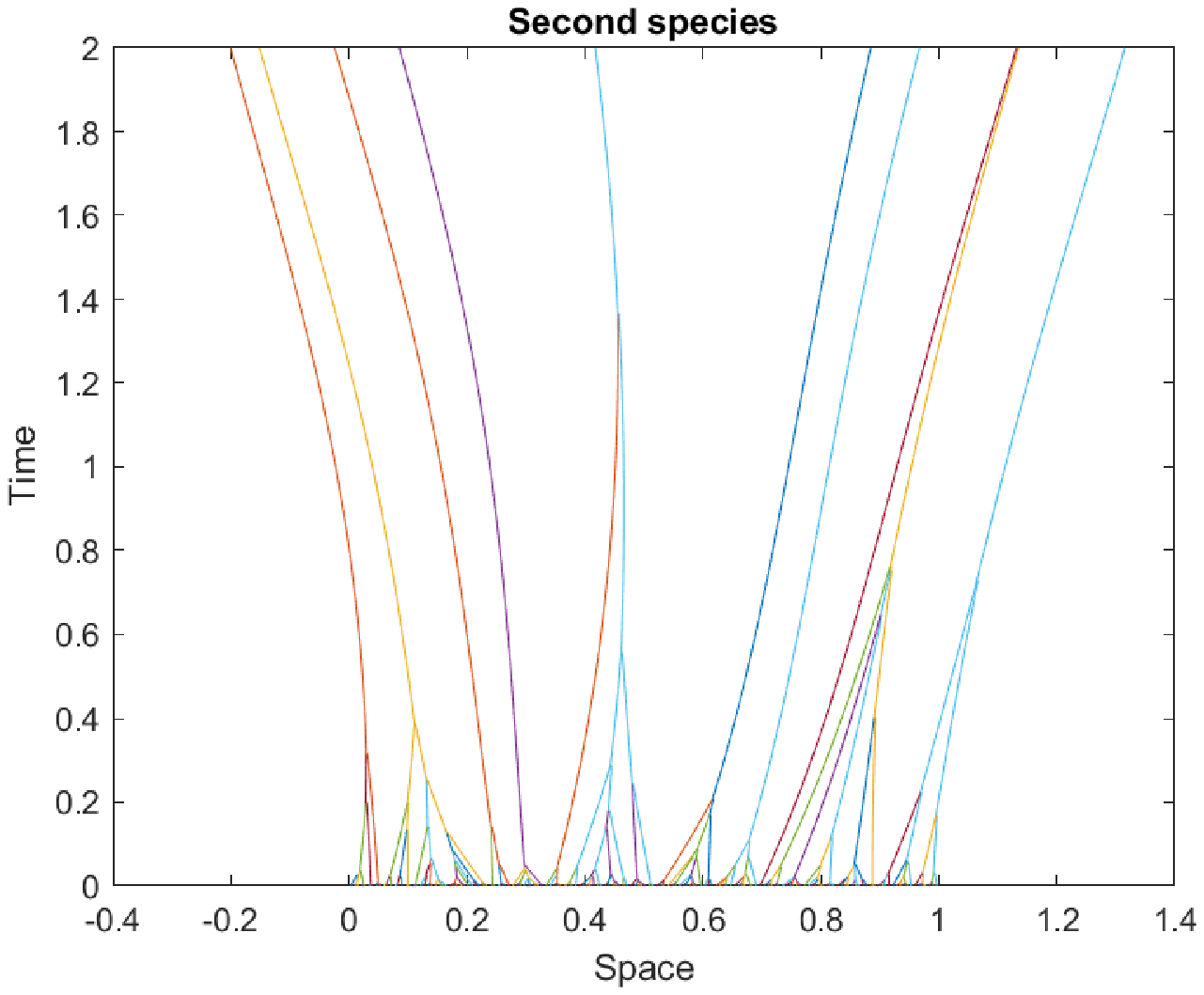}\\
\includegraphics[width=.49\textwidth]{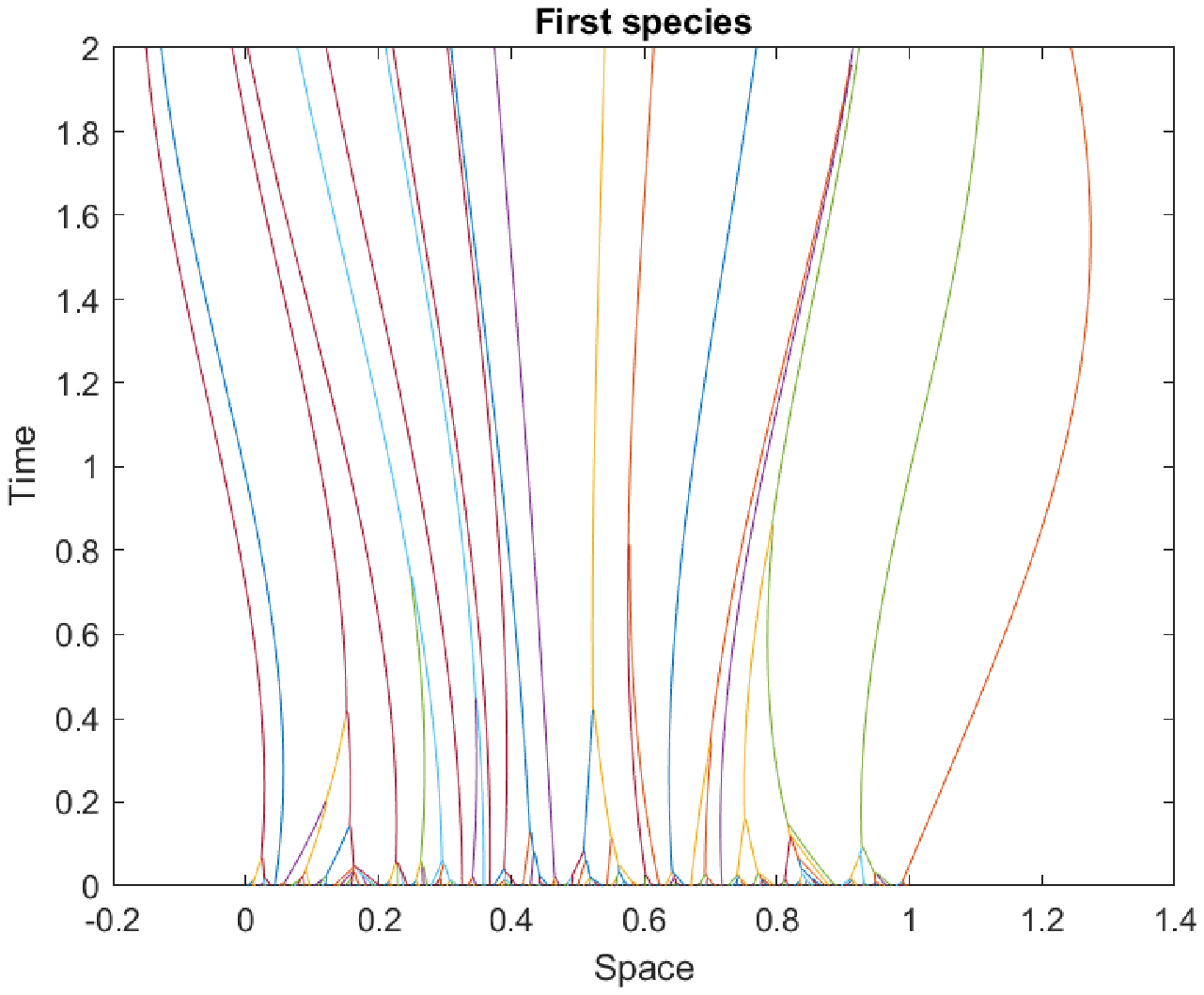}
\includegraphics[width=.49\textwidth]{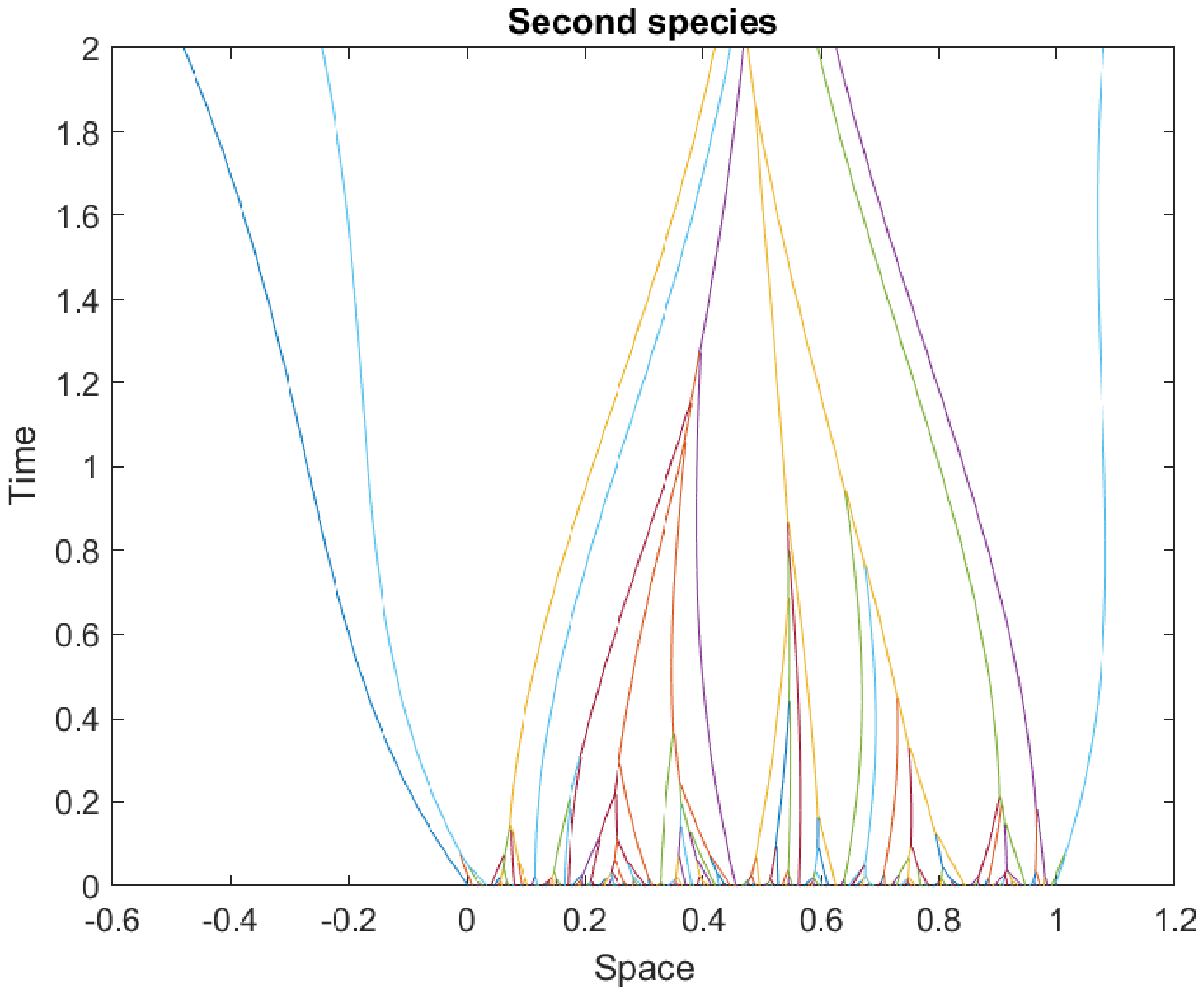}
\caption{Two possible outcomes (top and bottom) for the evolution of the system under the action of self-repulsive potentials $K_\rho(x)=2e^{-\abs{x}^2}$ and $K_\eta(x)=e^{-\abs{x}^3}$ and attractive cross-potentials $H_\rho(x)=\abs{x}^2$ and $H_\eta(x)=-e^{-3\abs{x}^2}$. In both the simulations the numbers of particles are fixed as $N=170$ and $M=160$, but initial velocities change (randomly).  }
\label{figure3}
\end{figure}

\begin{figure}[ht]
\centering
\includegraphics[width=.49\textwidth]{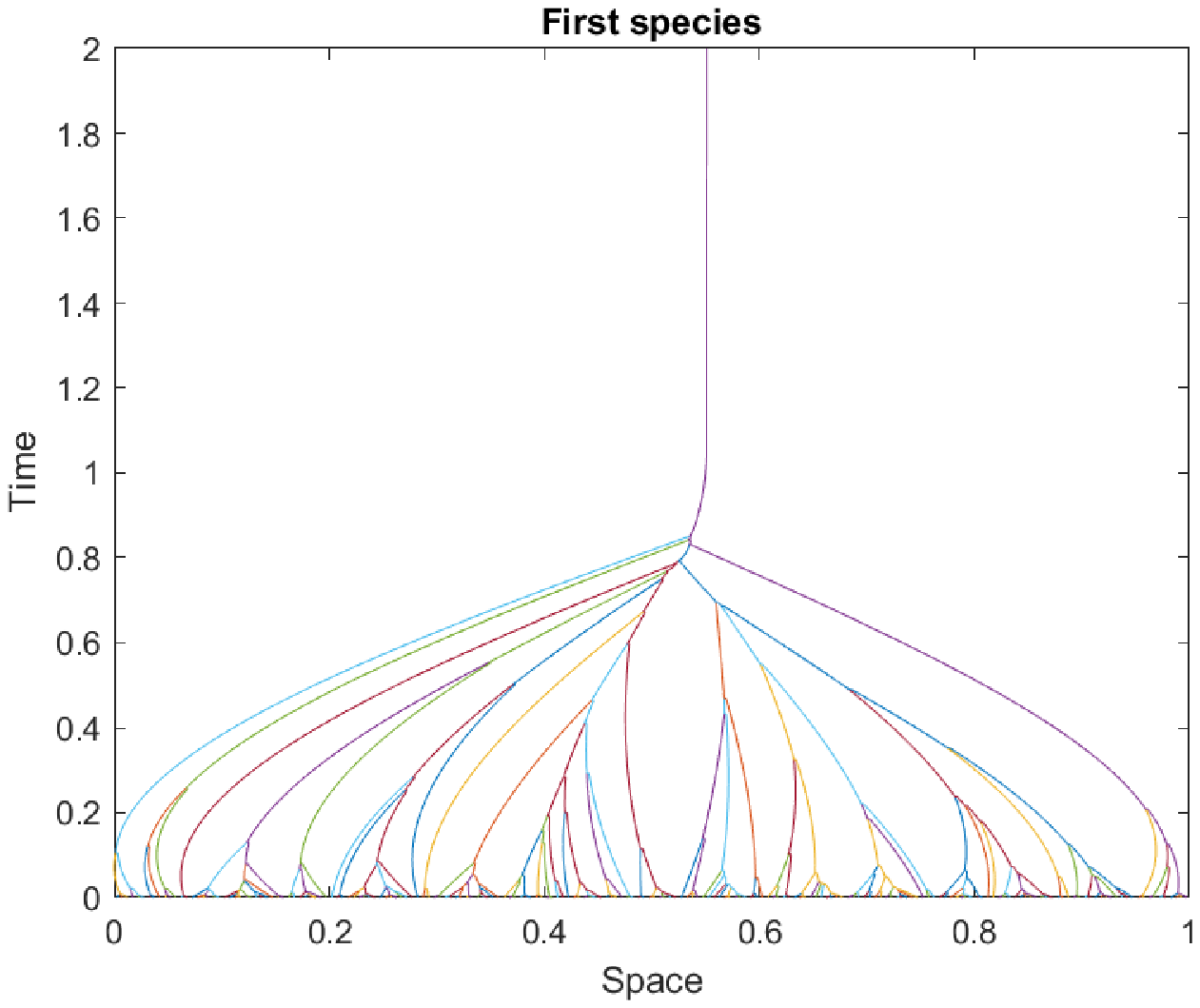}
\includegraphics[width=.49\textwidth]{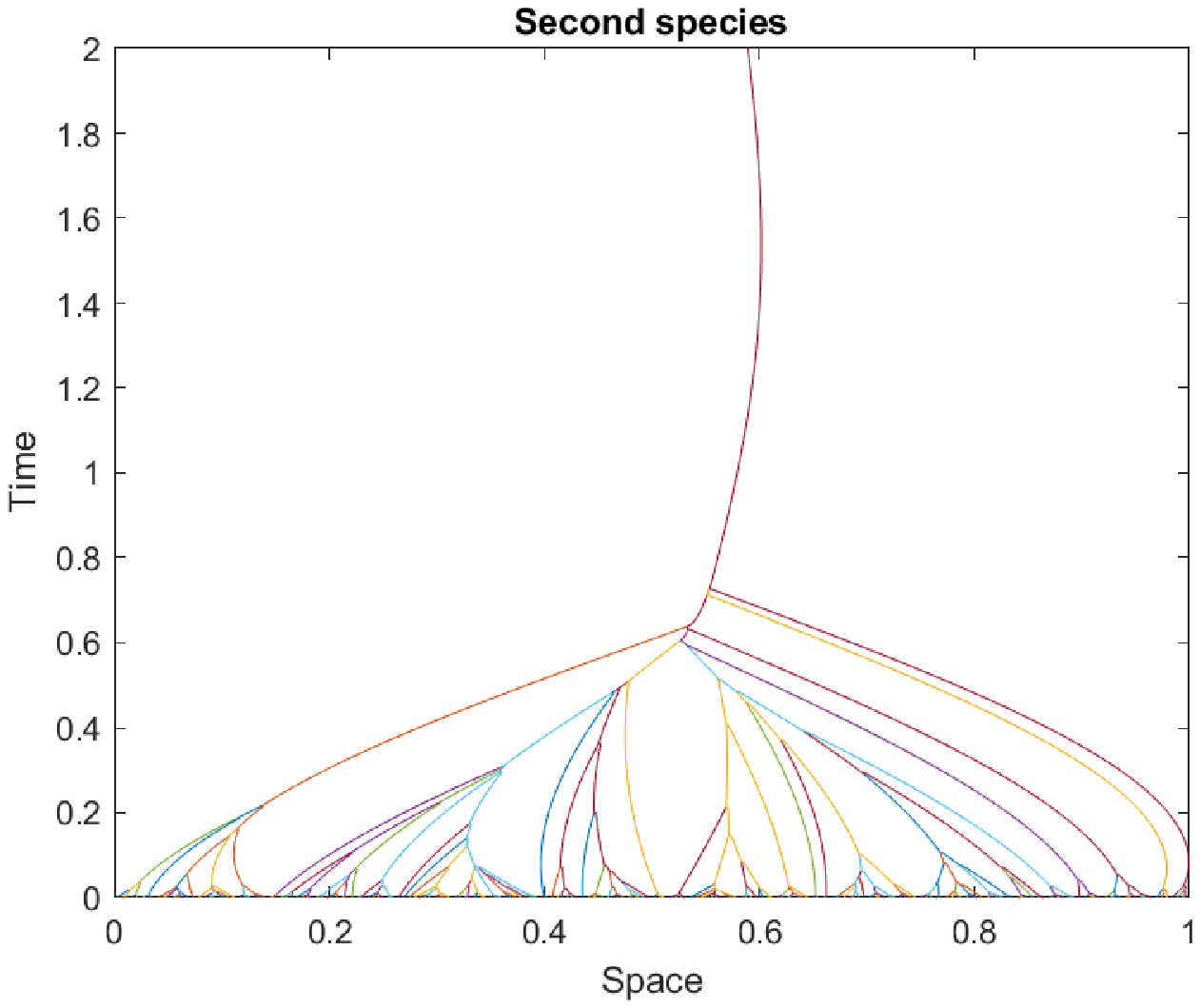}
\caption{Evolution under the action of attractive Newtonian self-potetials and attractive Gaussian cross-potentials given by $H_\rho(x)=H_\eta(x)=-e^{-\abs{x}^2}.$ The external potentials are $A_\rho (x)=\abs{x-\frac{1}{2}}^2$ and $A_\eta (x)=2\abs{x-\frac{1}{2}}^2$. In this example, $N=200$ and $M=210$.}
\label{figure:newattr}
\end{figure}

\begin{figure}[ht]
\centering
\includegraphics[width=.49\textwidth]{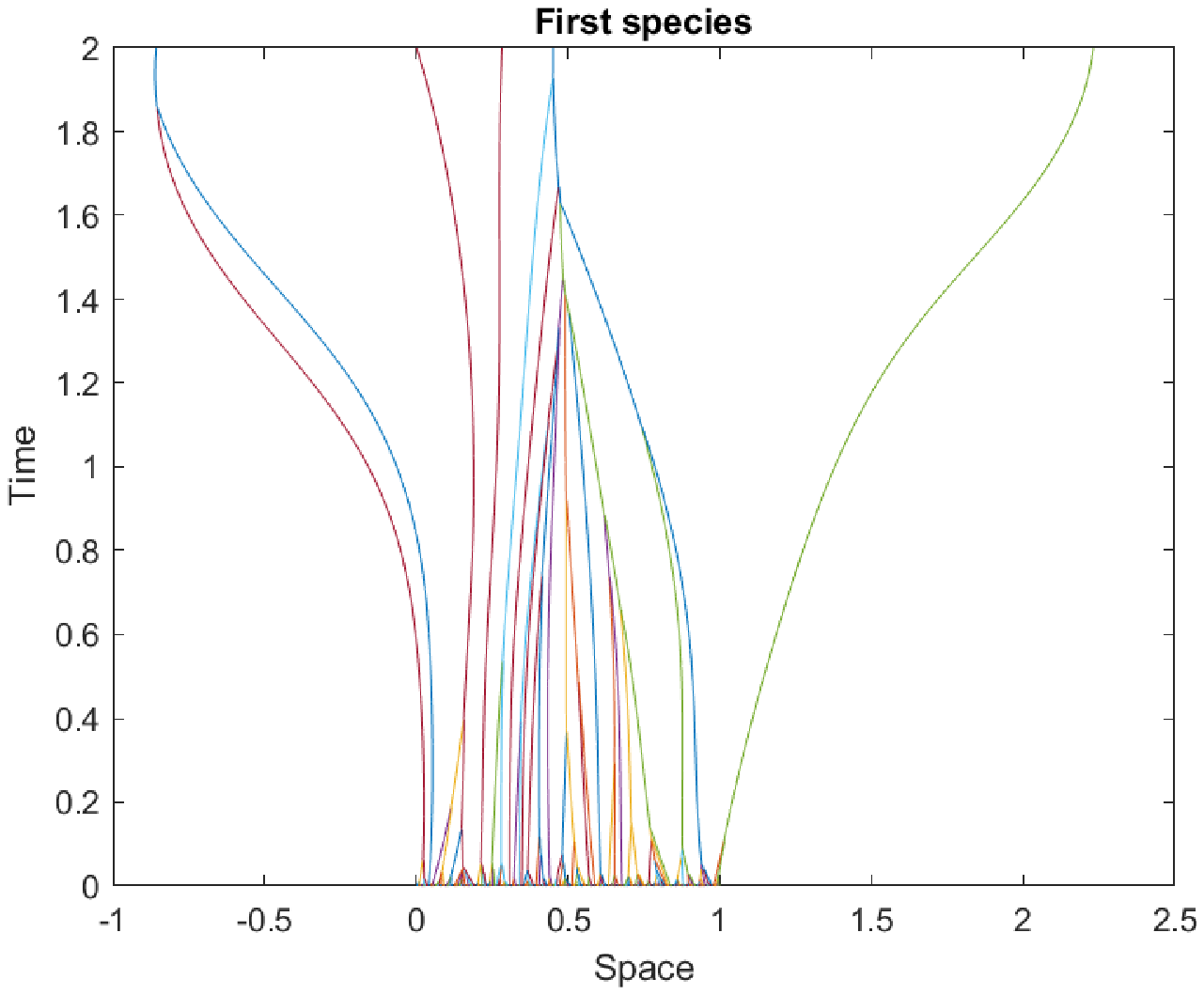}
\includegraphics[width=.49\textwidth]{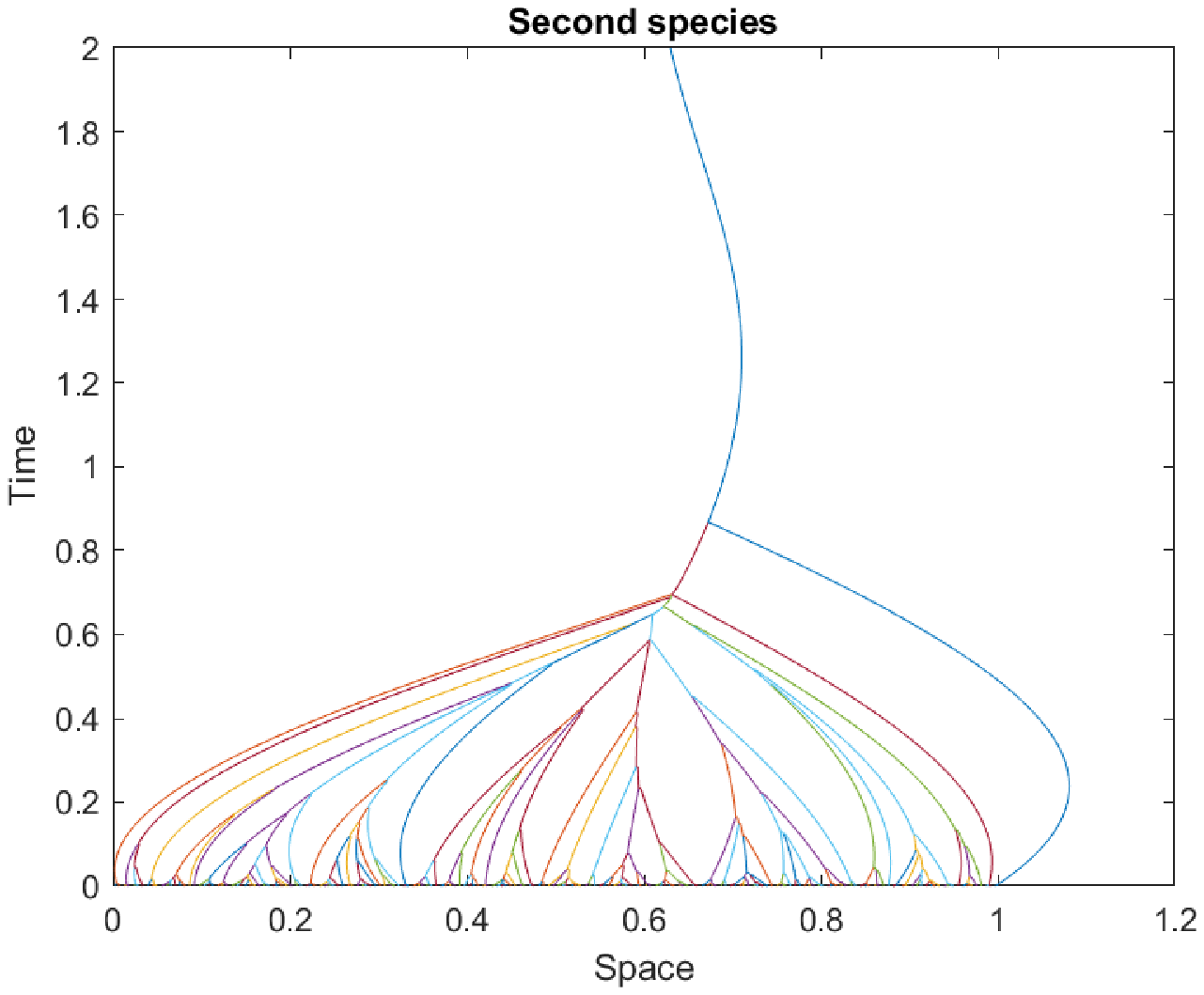}
\caption{In this example, $N=180,$ $M=190,$ the self-potentials are Newtonian attractive and the cross-potentials are equal and repulsive. In particular they are $H_\rho(x)=H_\eta(x)=3e^{-\abs{x}^4}.$ The external potentials are $A_\rho (x)=\frac{1}{2}\abs{x-\frac{1}{2}}^2$ and $A_\eta (x)=5\abs{x-\frac{1}{2}}^2.$}
\label{figure:newrep}
\end{figure}

\begin{figure}[ht]
\centering
\includegraphics[width=.49\textwidth]{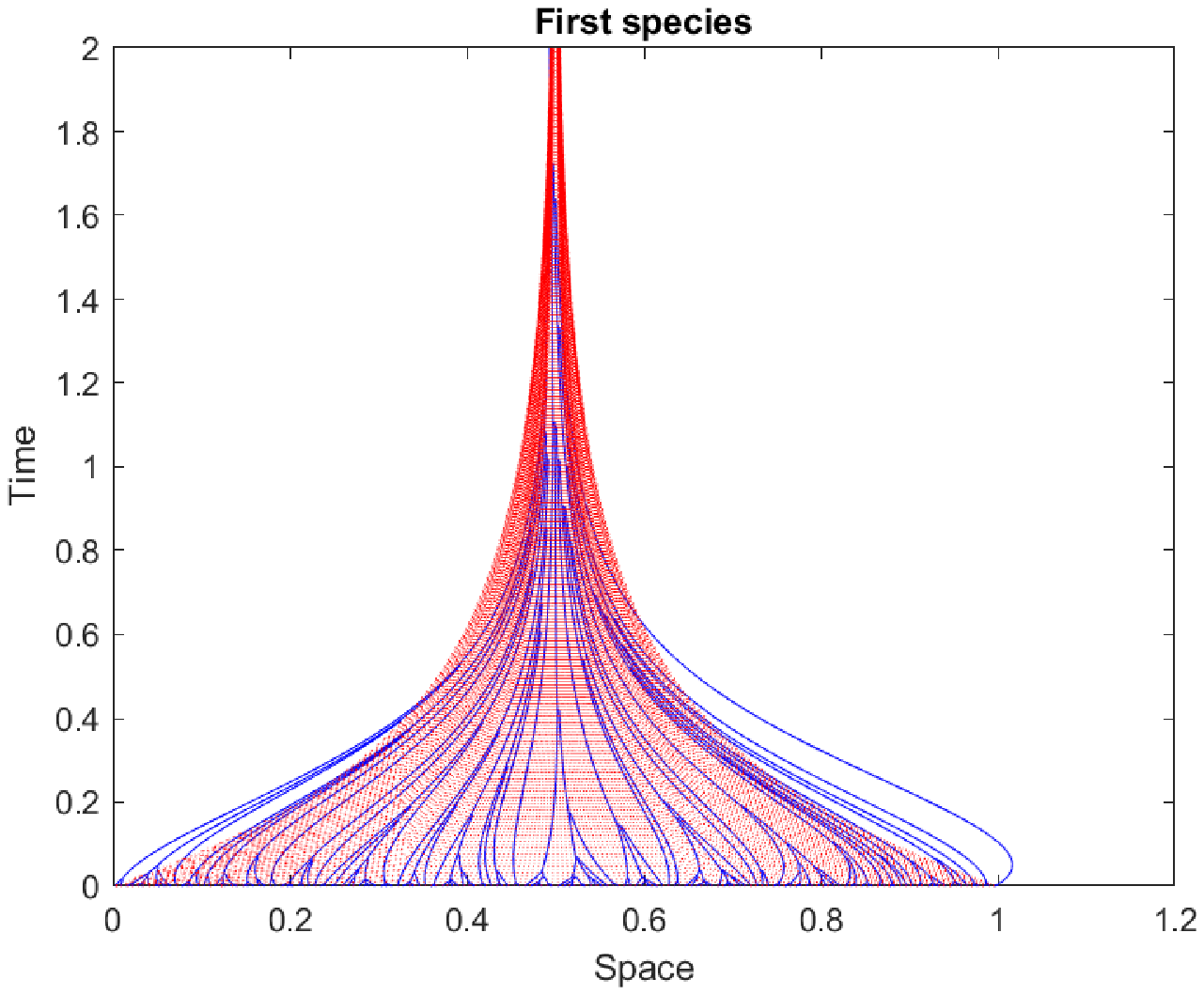}
\includegraphics[width=.49\textwidth]{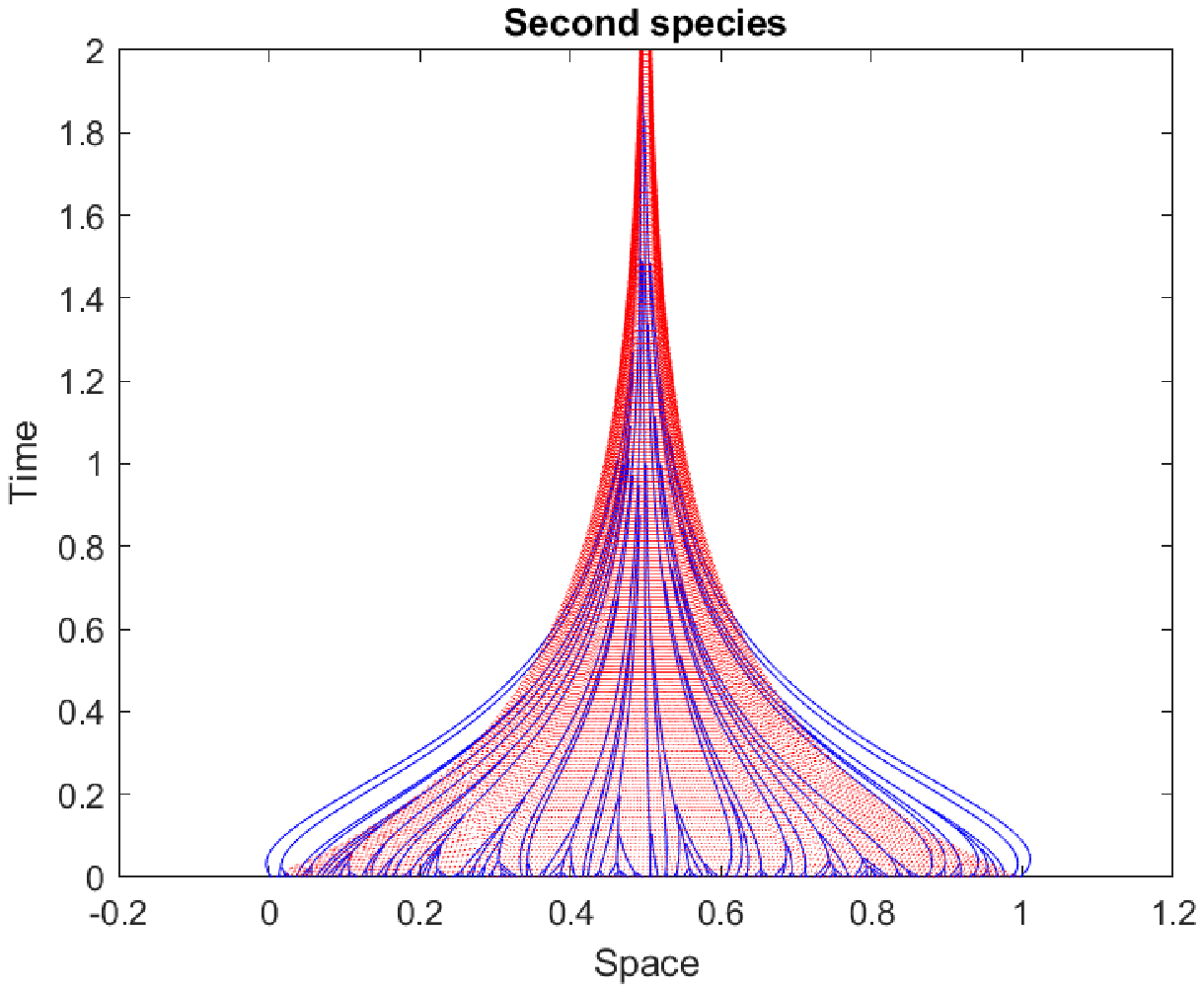}\\
\includegraphics[width=.49\textwidth]{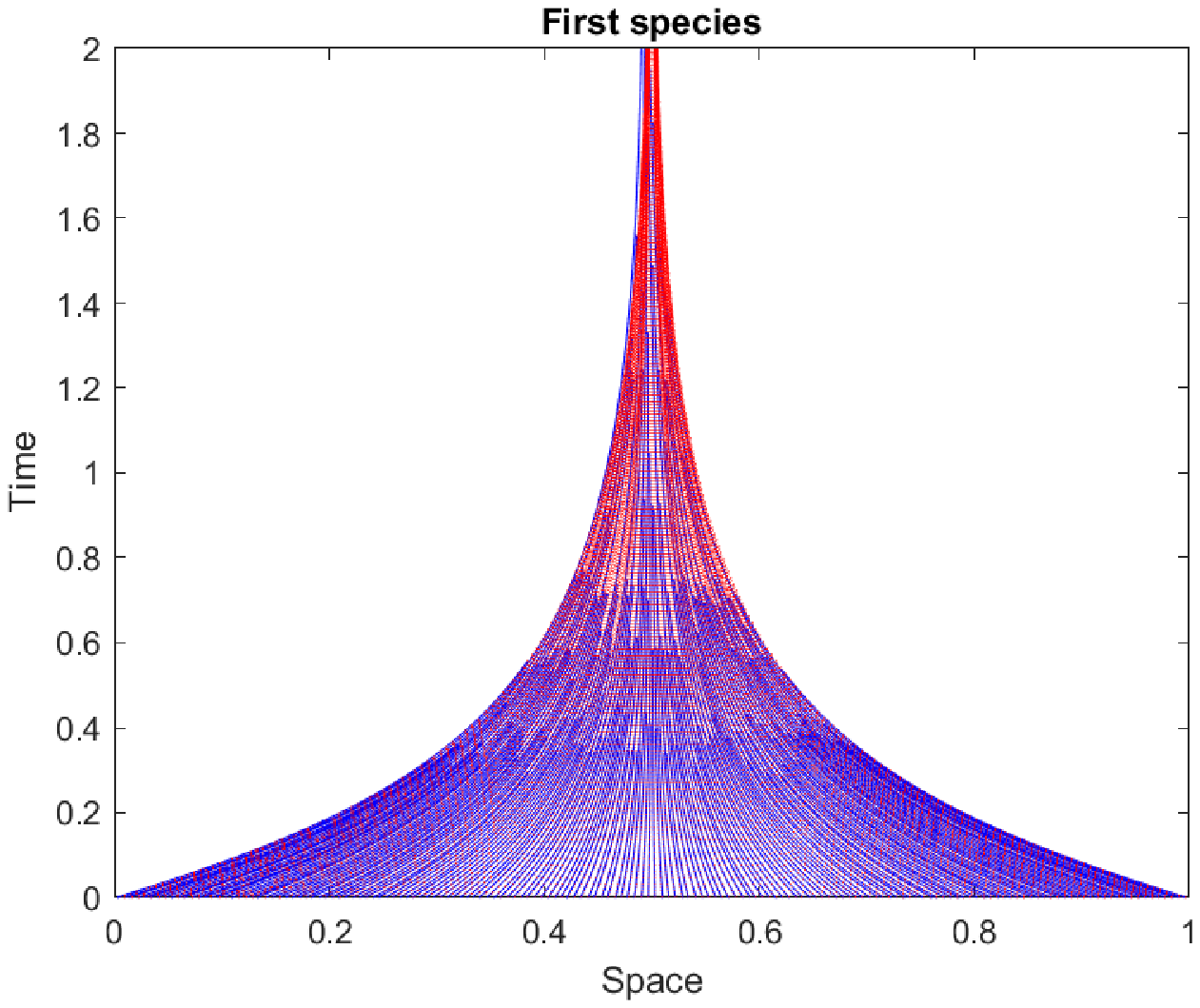}
\includegraphics[width=.49\textwidth]{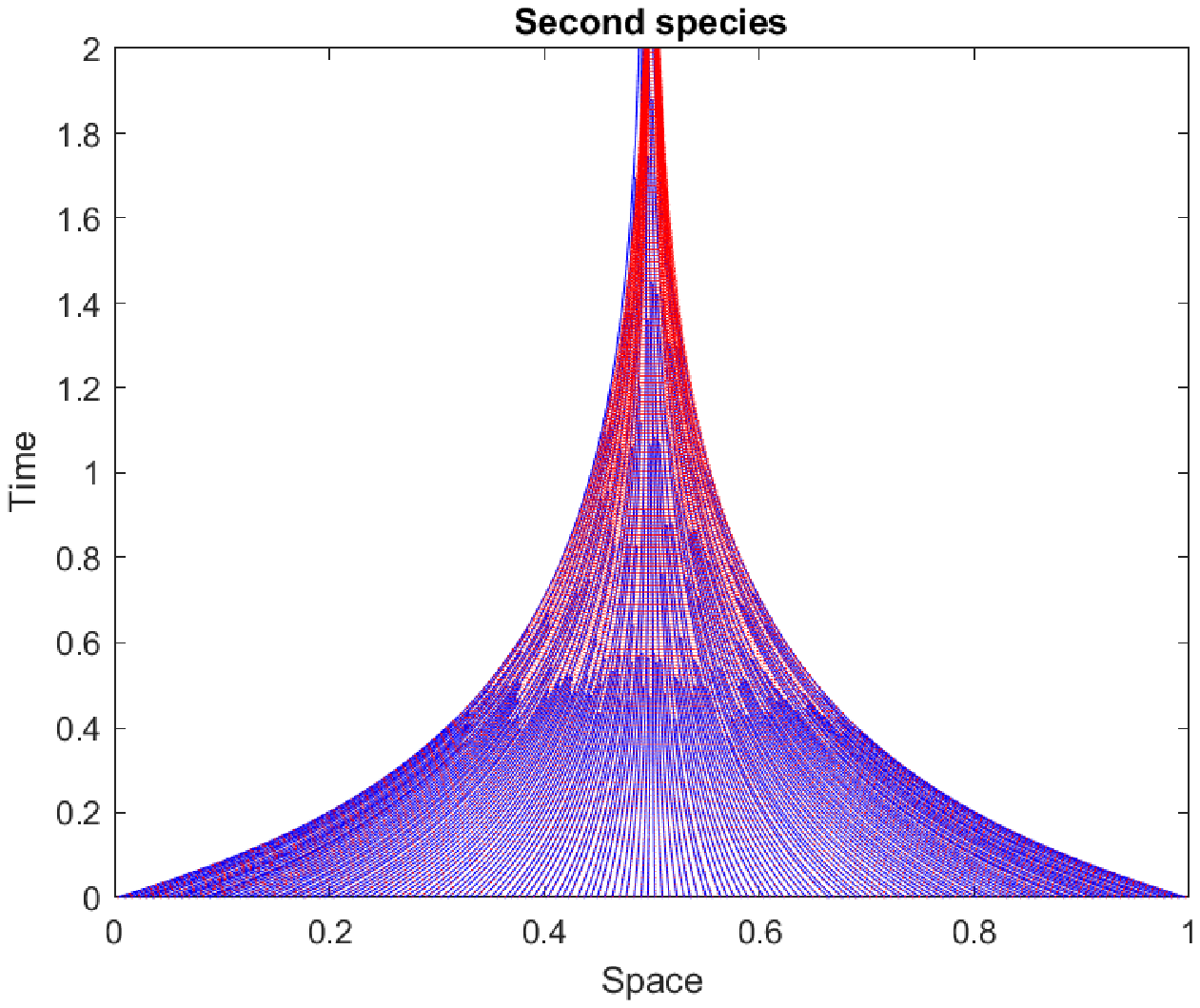}
\caption{Solutions to the  second-order system (blue in the online version) and solutions to first-order system   \eqref{eq:limitsystem} (in red) under the action of the following potentials are $K_\rho(x)=-e^{-\abs{x}^3},$ $K_\eta(x)=-e^{-\abs{x}^4},$ $H_\rho(x)=H_\eta(x)=-e^{-\abs{x}^2}$. In this simulation we set $N=160$, $M=150$ and $\sigma=10$ (top) and $\sigma=1000$ (bottom).}
\label{figure:confronto1a}
\end{figure}


\begin{figure}[ht]
\centering
\includegraphics[width=.49\textwidth]{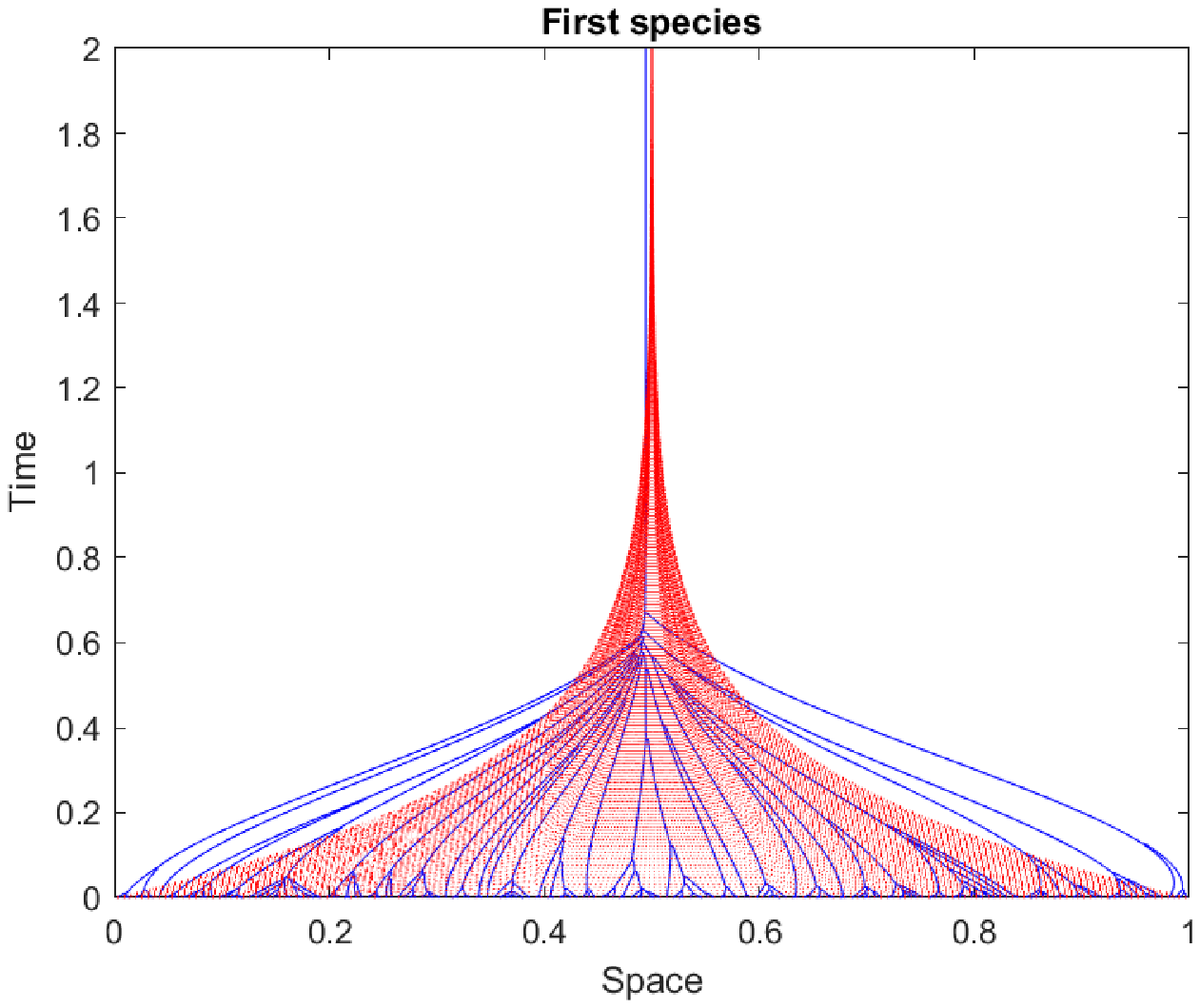}
\includegraphics[width=.49\textwidth]{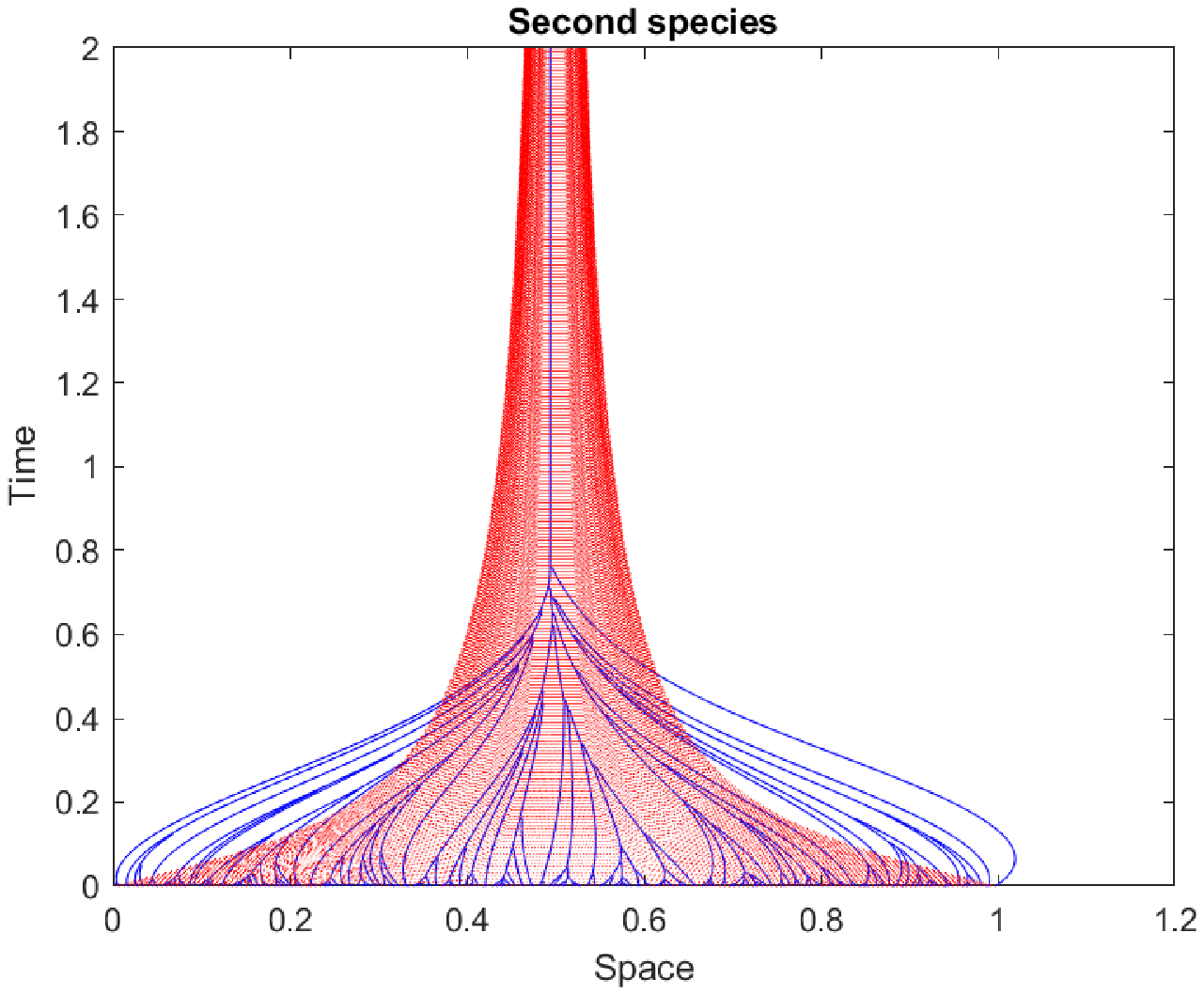}\\
\includegraphics[width=.49\textwidth]{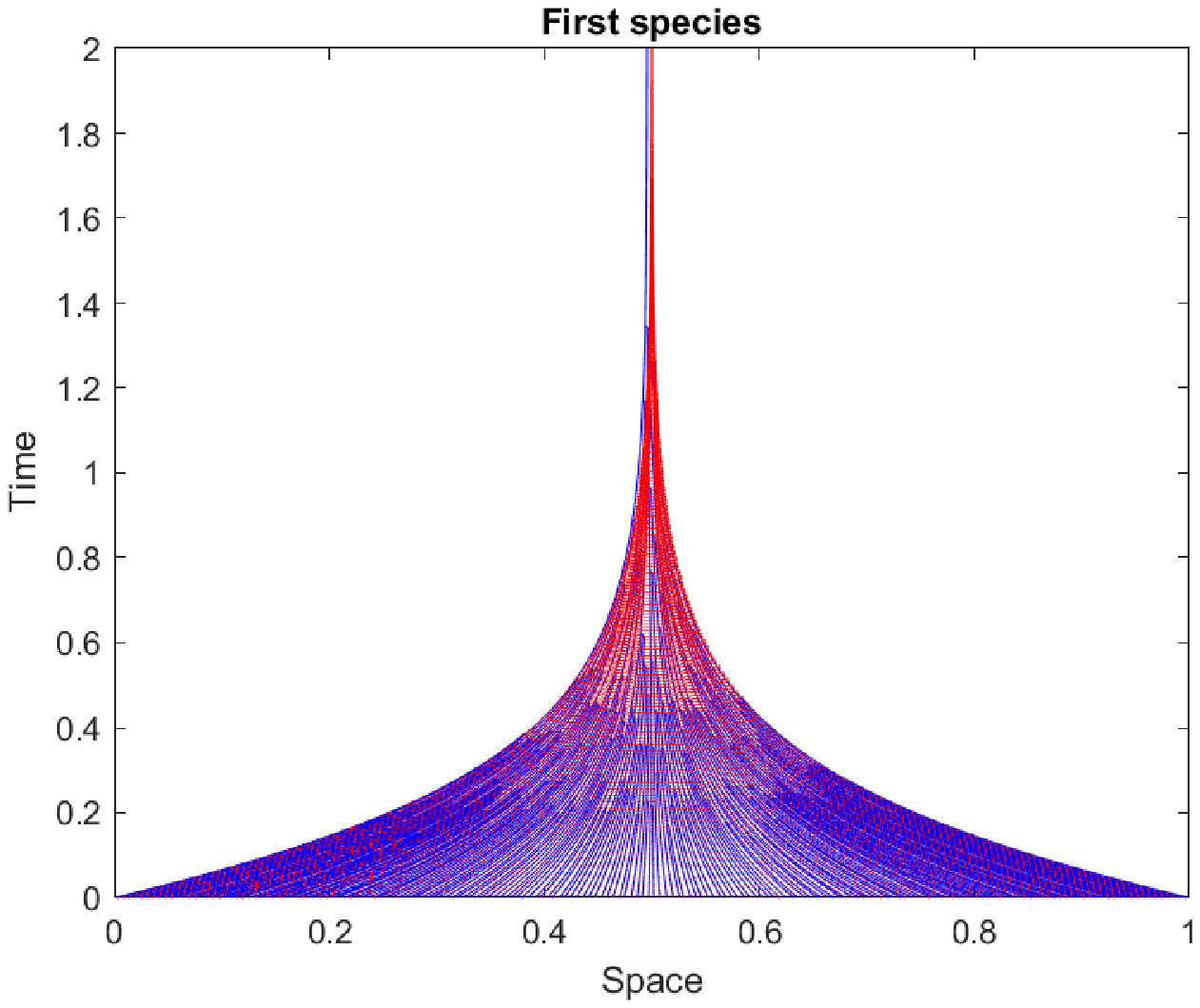}
\includegraphics[width=.49\textwidth]{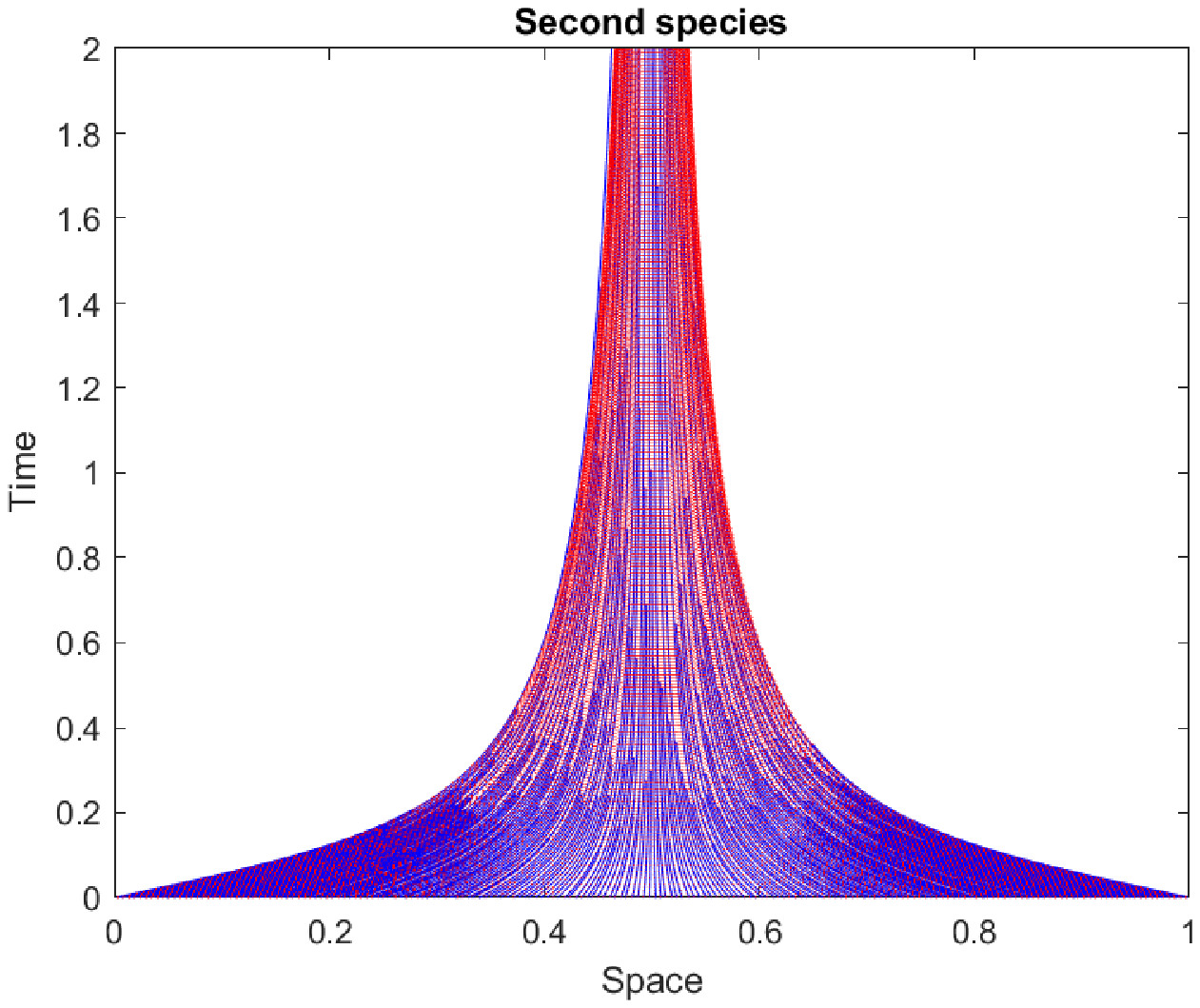}
\caption{Solutions to the  second-order system (blue in the online version) and solutions to first-order system   \eqref{eq:limitsystem} (in red) under the action of the following potentials are $K_\rho(x)=-e^{-\abs{x}^2},$ $K_\eta(x)=-e^{-3\abs{x}^3},$ $H_\rho(x)=\abs{x}^2$, $H_\eta(x)=-e^{-2\abs{x}^4}.$. In this simulation we set $N=180$, $M=190$ and $\sigma=5$ (top) and $\sigma=900$ (bottom). }
\label{figure:confronto2b}
\end{figure}

\begin{figure}[ht]
\centering
\includegraphics[width=.49\textwidth]{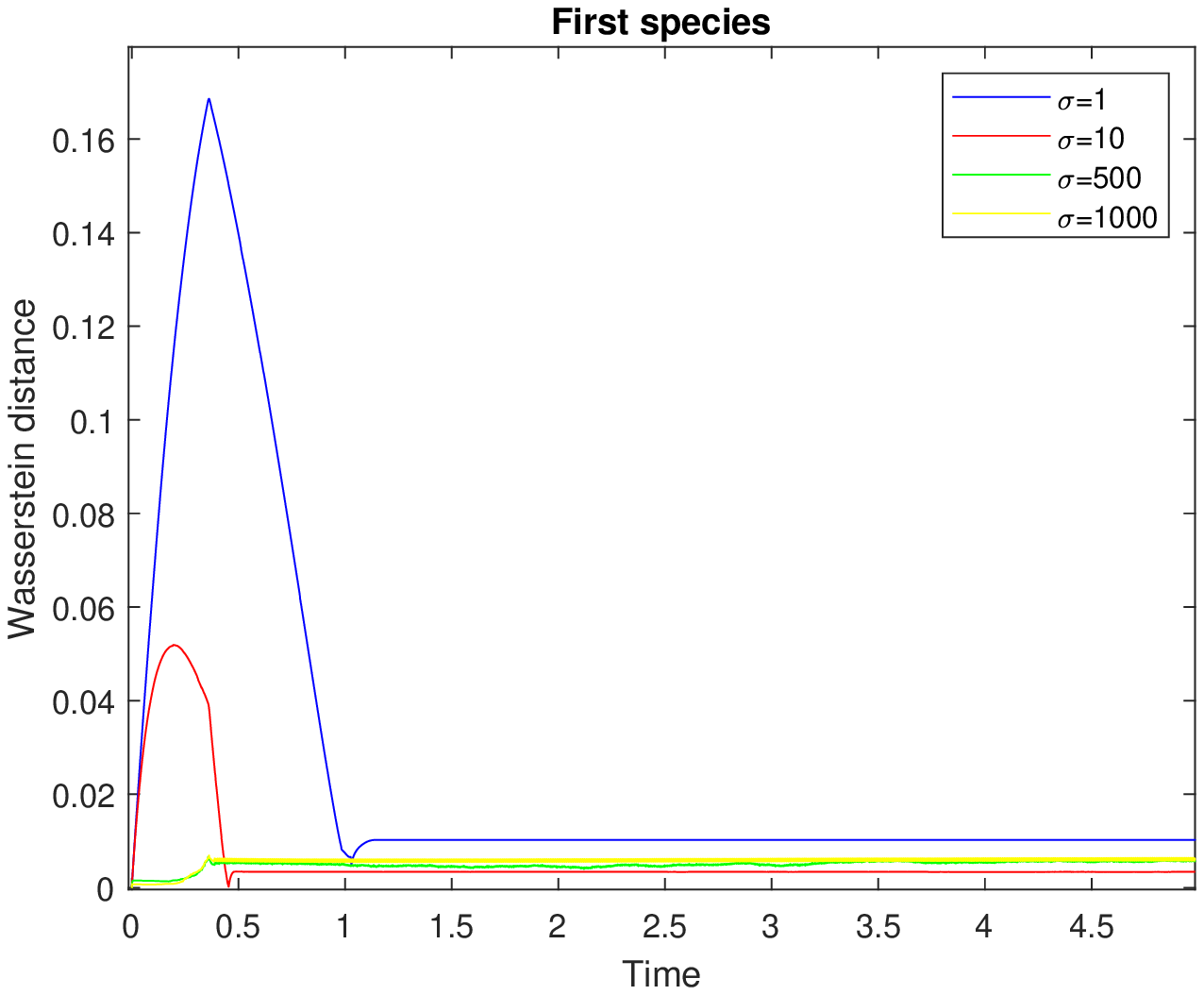}
\includegraphics[width=.49\textwidth]{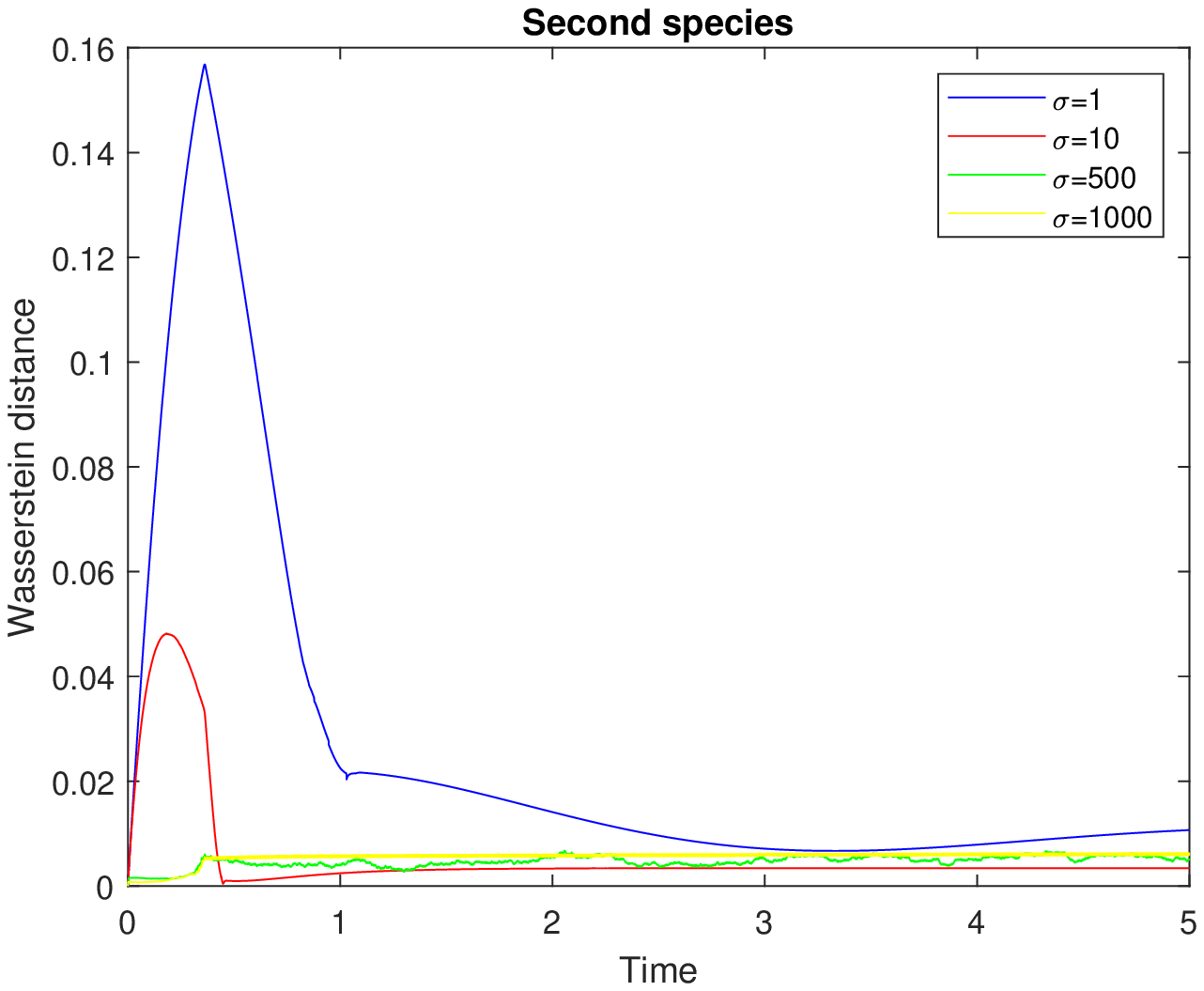}
\caption{Behaviour of the Wasserstein distance between solutions of the first order system and solutions of the second order system. The self-potentials are Newtonian attractive potentials, while the cross-potentials are given by $H_\rho(x)=H_\eta(x)=-e^{-\abs{x}^2}.$ Increasing the damping parameter Wasserstein distance remain controlled.}
\label{figure:wass}
\end{figure}


System \eqref{syst:numerics} will be coupled with an uniform distributed set of particles in the space interval $[0,1]$ and a randomly distrubution for the velocities. We then let the particles evolve by using an explicit second-order three steps Runge-Kutta method, (cf. \cite{tadmor}) up to the first \emph{collision}. In order to detect \emph{collisions} between particles  we  fix a tolerance parameter $toll$ and we
assume that it occurs when the distance between two consecutive particles of the same species, for instance $x_i$ and $x_{i+1}$, is smaller than $toll$. Once two consecutive particles \emph{collide} they are replaced by a single particle with new position and velocity given by
\begin{align*}
        x_{i+\frac{1}{2}}(t)&=\frac{x_i(t)+x_{i+1}(t)}{2}, \\ v_{i+\frac{1}{2}}(t)&=\frac{v_i(t)+v_{i+1}(t)}{2},
    \end{align*}
and doubled mass, and we let the system evolving again with this new set of particels.  In all the simulations below we fix $toll = 0.002$.

We study numerical solutions to the system \eqref{syst:numerics} both in case of smooth potentials and in case of Newtonian self-potentials. Several examples are presented in the smooth case, where we highlight the possibility of a sticky dynamics, both in attractive and repulsive regime. Furthermore, we will compare  solutions to second-order system with  solution to first-order one as the increasing values of the damping parameter $\sigma$, also comparing the Wasserstein distance between the solution to the second-order system and the solution to the first-order system as $\sigma$ varies. Wasserstein distance is computed using its one-dimensional equivalence with the $L^2-$norm at the level of monotone rearrangements.

The first examples we provide concern the evolution of particles subject to the action of radial smooth potentials. Figure \ref{figure1} displays the sticky particle dynamics when all the potentials are smooth and attractive. Instead, in Figure \ref{figure2} the self-potentials are attractive and the cross-potentials are repulsive, while in Figure \ref{figure3} the self-potentials are repulsive and the cross-potentials are attractive. In particular, we highlight how the behaviour is strongly different by comparing two simulations performed with the same potentials, number of particles and initial position, but different set of initial (random) velocities.

We then show a couple of simulations in which the self-potentials are attractive Newtonian, while the cross-potentials are symmetric, radial and smooth. In particular, in Figure \ref{figure:newattr}, the cross-potentials are attractive, indeed the particles collide, while in Figure \ref{figure:newrep}, they are repulsive and not all the particles collide. According to results in Section \ref{sec:newtonian} also the effect of external potentials is taken into account.

We then focus on the numerical investigation of the large damping regime. Figures \ref{figure:confronto1a} and \ref{figure:confronto2b} show a comparison between the particle evolution associated to the second order system and the ones associated to the first order system \eqref{eq:limitsystem}, for various choices of potentials. We highlight numerically the relevance of the dumping parameter $\sigma$ in the evolution: increasing the value of $\sigma$ solutions of the two different problems becomes indistinguishable. 

Finally in Figure \ref{figure:wass}, considering the same potentials in Figure \ref{figure:newattr}, we display the Wasserstein distance between the solution to the second-order system and the ones to the first-order system for different values of $\sigma$. For small values of $\sigma$, the Wasserstein distance grows initially, and then decays in time. When $\sigma$ is bigger, the distance remains controlled for all times.

\section*{Acknowledgments}

The research of MDF and SF is  supported by the Ministry of University and Research (MIUR), Italy under the grant PRIN 2020- Project N. 20204NT8W4, Nonlinear Evolutions PDEs, fluid dynamics and transport equations: theoretical foundations and applications. The research of SF and VI is supported by the Italian INdAM project N. E55F22000270001 ``Fenomeni di trasporto in leggi di conservazione e loro applicazioni''. SF is also supported by University of L'Aquila 2021 project 04ATE2021 - ``Mathematical Models For Social Innovations: Vehicular And Pedestrian Traffic, Opinion Formation And Seismology.''

\bibliographystyle{plain}

\end{document}